\documentclass[11pt]{article}

\title{
Representation stability and finite linear groups}
\author{Andrew Putman\thanks{Supported in part by NSF grant DMS-1255350 and the Alfred P.\ Sloan Foundation} \and Steven V Sam\thanks{Supported by a Miller research fellowship}}

\date{January 13, 2017}

\usepackage{tocloft}
\usepackage{bbm}
\newcommand{\arxiv}[1]{\href{http://arxiv.org/abs/#1}{{\tt arXiv:#1}}}

\usepackage{amsmath,amssymb,amsthm,amscd,amsfonts}
\usepackage{epsfig,pinlabel,paralist}
\usepackage[vmargin=0.9in, hmargin=1.2in]{geometry}
\usepackage[font=small,format=plain,labelfont=bf,up,textfont=it,up]{caption}
\usepackage{type1cm}
\usepackage{calc}
\usepackage{enumerate}
\usepackage{bm}
\usepackage{microtype}
\usepackage[all,cmtip]{xy}

\usepackage{lmodern}
\usepackage[T1]{fontenc}

\usepackage[dvipdfm, 
bookmarks, bookmarksdepth=2, colorlinks=true, linkcolor=blue, citecolor=blue, urlcolor=blue]{hyperref}

\usepackage{etoolbox}
\apptocmd{\thebibliography}{\raggedright}{}{}

\numberwithin{equation}{section}

\theoremstyle{plain}
\newtheorem{theorem}{Theorem}[section]
\newtheorem{maintheorem}{Theorem}
\newtheorem{proposition}[theorem]{Proposition}
\newtheorem{lemma}[theorem]{Lemma}

\newtheorem{corollary}[theorem]{Corollary}

\newtheorem{observation}[theorem]{Observation}
\newtheorem*{claim}{Claim}

\theoremstyle{definition}

\newtheorem{defn}[theorem]{Definition}
\newenvironment{definition}[1][]{\begin{defn}[#1]\pushQED{\qed}}{\popQED \end{defn}}

\newtheorem{rmk}[theorem]{Remark}
\newenvironment{remark}[1][]{\begin{rmk}[#1] \pushQED{\qed}}{\popQED \end{rmk}}

\newtheorem{eg}[theorem]{Example}
\newenvironment{example}[1][]{\begin{eg}[#1] \pushQED{\qed}}{\popQED \end{eg}}

\DeclareMathOperator{\Hom}{Hom}
\DeclareMathOperator{\Iso}{Iso}

\DeclareMathOperator{\Ker}{ker}

\DeclareMathOperator{\Image}{Im}

\DeclareMathOperator{\Mod}{Mod}
\DeclareMathOperator{\MCG}{MCG}

\DeclareMathOperator{\Sp}{Sp}
\DeclareMathOperator{\GL}{GL}
\DeclareMathOperator{\SL}{SL}

\newcommand\R{\ensuremath{\mathbb{R}}}

\newcommand\Z{\ensuremath{\mathbb{Z}}}
\newcommand\Q{\ensuremath{\mathbb{Q}}}

\newcommand\Field{\ensuremath{\mathbb{F}}}

\DeclareMathOperator{\HH}{H}

\newcommand\Chain{\ensuremath{{\rm C}}}

\DeclareMathOperator{\Max}{max}

\DeclareMathOperator{\Aut}{Aut}

\DeclareMathOperator{\Interior}{Int}

\newcommand\Set[2]{\ensuremath{\{\text{#1 $|$ #2}\}}}

\newcommand\Figure[4]{
\begin{figure}[t]
\centering
\centerline{\psfig{file=#2,scale=#4}}
\caption{#3}
\label{#1}
\end{figure}}

\newcommand\Surface[2]{\ensuremath{\mathcal{S}_{#1}^{#2}}}
\renewcommand{\O}{\mathcal{O}}
\DeclareMathOperator{\SSset}{\text{\tt SS}}

\DeclareMathOperator{\Ind}{Ind}
\DeclareMathOperator{\FI}{\text{\tt FI}}
\DeclareMathOperator{\VI}{\text{\tt VI}}
\DeclareMathOperator{\VIC}{\text{\tt VIC}}
\DeclareMathOperator{\SI}{\text{\tt SI}}
\DeclareMathOperator{\OVIC}{\text{\tt OVIC}}
\DeclareMathOperator{\OSI}{\text{\tt OSI}}
\DeclareMathOperator{\Fr}{\text{\tt Fr}}
\DeclareMathOperator{\Surf}{\text{\tt Surf}}

\DeclareMathOperator{\bFI}{\ensuremath{\text{\tt \bf FI}}}
\DeclareMathOperator{\bVI}{\text{\tt \bf VI}}
\DeclareMathOperator{\btV}{\text{\tt \bf V}}
\DeclareMathOperator{\bwV}{\ensuremath{\text{\tt \bf V}^{\boldsymbol{\prime}}}}
\DeclareMathOperator{\bVIC}{\text{\tt \bf VIC}}
\DeclareMathOperator{\bSI}{\text{\tt \bf SI}}
\DeclareMathOperator{\bOVIC}{\text{\tt \bf OVIC}}
\DeclareMathOperator{\bOSI}{\text{\tt \bf OSI}}

\DeclareMathOperator{\bSurf}{\text{\tt \bf Surf}}

\DeclareMathOperator{\ialg}{\hat{\imath}}
\DeclareMathOperator{\OrderedShift}{\Sigma}

\DeclareMathOperator{\ShiftSet}{\mathfrak{I}}
\DeclareMathOperator{\HHH}{\mathcal{H}}
\DeclareMathOperator{\HHHH}{\mathfrak{H}}

\DeclareMathOperator{\Conf}{Conf}
\DeclareMathOperator{\PConf}{PConf}
\DeclareMathOperator{\Emb}{Emb}

\DeclareMathOperator{\Unit}{\mathfrak{U}}

\DeclareMathOperator{\HDelta}{\widehat{\Delta}}

\DeclareMathOperator{\PartialBases}{{\rm PB}}
\DeclareMathOperator{\PartialSpBases}{{\rm PSB}}

\DeclareMathOperator{\TetheredChains}{{\rm TC}}
\DeclareMathOperator{\TetheredTori}{{\rm TT}}

\newcommand{\tilf}{\ensuremath{\widetilde{f}}}
\newcommand{\tilg}{\ensuremath{\widetilde{g}}}
\newcommand{\tilh}{\ensuremath{\widetilde{h}}}
\newcommand{\tilV}{\ensuremath{\widetilde{V}}}
\newcommand{\tilW}{\ensuremath{\widetilde{W}}}
\newcommand{\tilC}{\ensuremath{\widetilde{C}}}

\newcommand{\monpro}{\ensuremath{\circledast}}

\newcommand{\bb}{\mathbf{b}}

\newcommand{\bk}{\mathbf{k}}

\newcommand{\tA}{{\tt A}}
\newcommand{\tB}{{\tt B}}
\newcommand{\tC}{{\tt C}}
\newcommand{\tD}{{\tt D}}

\newcommand{\tV}{{\tt V}}
\DeclareMathOperator{\wV}{\ensuremath{\text{\tt V}^\prime}}
\newcommand{\cP}{\mathcal{P}}

\newcommand{\Sym}{\mathfrak{S}}
\newcommand{\init}{\mathrm{init}}

\newcommand{\Rect}{\mathfrak{R}}

\newcommand{\mybigwedge}[1]{\bigwedge\nolimits^{\!#1}}

\usepackage{ulem}

\newskip\bigskipamount \bigskipamount =5pt plus 3pt minus 2pt
\renewcommand{\paragraph}[1]{\bigskip \noindent {\bf #1}}

\begin{document}

\maketitle

\vspace{-24pt}
\begin{abstract}
\noindent
We study analogues of FI-modules where the role of the symmetric group is played by the general linear groups and the symplectic groups over finite rings and prove basic structural properties such as local Noetherianity. Applications include a proof of the Lannes--Schwartz Artinian conjecture in the generic representation theory of finite fields, very general homological stability theorems with twisted coefficients for the general linear and symplectic groups over finite rings, and representation-theoretic versions of homological stability for congruence subgroups of the general linear group, the automorphism group of a free group, the symplectic group, and the mapping class group.
\end{abstract}

\setcounter{tocdepth}{1}
\tableofcontents

\section{Introduction}

The language of {\bf representation stability} was introduced by
Church--Farb \cite{ChurchFarbRepStability} to describe representation-theoretic
patterns they observed in many parts of mathematics.  In later work
with Ellenberg \cite{ChurchEllenbergFarb}, they introduced {\bf $\FI$-modules},
which gave a clearer framework for these patterns for
the representation theory of the symmetric group.  The key technical innovation
of their work was a certain {\bf Noetherian} property of $\FI$-modules, which
they later established in great generality with Nagpal 
\cite{ChurchEllenbergFarbNagpal}.  This Noetherian property allowed them
to prove an {\bf asymptotic structure theorem} for finitely generated
$\FI$-modules, giving an elegant mechanism behind the patterns
observed earlier in \cite{ChurchFarbRepStability}.
The theory of $\FI$-modules fits into the larger picture of functor categories, which
play an important role in algebraic topology and algebraic K-theory;
see \cite{FFPSBook, FTBook} for surveys.  

In this paper, we study functor categories that can be viewed as analogues of $\FI$-modules where the role of the
symmetric group is played by the classical finite groups of matrices, namely
the general linear groups and the symplectic groups over finite rings.
We establish analogues of the aforementioned Noetherian and asymptotic
structure results.  A byproduct of our work here is a proof of the
Lannes--Schwartz Artinian conjecture (see Theorem \ref{maintheorem:artinianconj}).  
We have two families of applications.
The first are very general homological stability theorems with twisted coefficients
for the general linear groups and symplectic groups over finite rings.
The second are representation-theoretic analogues of homological stability for congruence subgroups
of the general linear group, the automorphism group of a free group,
the symplectic group, and the mapping class group.

\subsection{Review of \texorpdfstring{$\bFI$}{FI}-modules}
\label{section:fireview}

Since our theorems parallel results about $\FI$-modules, we begin by reviewing the theory of $\FI$-modules.
Unless otherwise specified, {\it all rings in this paper are commutative and contain $1$}. Fix a Noetherian ring $\bk$.

\paragraph{Configuration spaces.}
We start with an example. Let $X$ be a compact oriented manifold with $\partial X \neq \emptyset$. Let
\[
\PConf_n(X) = \{ (x_1,\dots,x_n) \in X^n \mid x_i \ne x_j \text{ for } i \ne j\}
\]
be the configuration space of $n$ ordered points in $X$. The symmetric group $\Sym_n$ acts freely on $\PConf_n(X)$ by
permuting points.  Let
\[
\Conf_n(X) = \PConf_n(X) / \Sym_n
\]
be the configuration space of $n$ unordered points in $X$.
McDuff \cite{McDuffConfig} proved that $\Conf_n(X)$ satisfies
{\bf cohomological stability}: for $k \geq 1$, we have
$\HH^k(\Conf_{n}(X);\bk) \cong \HH^k(\Conf_{n+1}(X);\bk)$ for $n \gg 0$.
For $\PConf_n(X)$, easy examples show that cohomological stability does not hold.  Observe that $\HH^k(\PConf_n(X);\bk)$
is a representation of $\Sym_n$.  Building on work of Church \cite{ChurchConfiguration},
Church--Ellenberg--Farb--Nagpal \cite{ChurchEllenbergFarbNagpal} proved that
$\HH^k(\PConf_n(X);\bk)$ stabilizes as a representation of $\Sym_n$ in a natural sense.
The key to this is the theory of $\FI$-modules.

\paragraph{Basic definitions.}
Let $\FI$ be the category of finite sets and injections.
An {\bf $\FI$-module} over a ring $\bk$ is a functor $M\colon \FI \rightarrow \Mod_{\bk}$.
More concretely, $M$ consists of the following.
\begin{compactitem}
\item For each finite set $I$, a $\bk$-module $M_I$.
\item For each injection $f\colon I \hookrightarrow J$ between finite sets $I,J$,
a homomorphism $M_f\colon M_I \rightarrow M_J$.  These homomorphisms must satisfy
the obvious compatibility conditions.
\end{compactitem}
For $I \in \FI$, the $\FI$-endomorphisms of $I$ are exactly the symmetric
group $\Sym_I$, so $\Sym_I$ acts on $M_I$.
In particular, setting $[n] = \{1,\ldots,n\}$ and $M_n = M_{[n]}$,
the group $\Sym_n$ acts on $M_n$.

\begin{example}
\label{example:configspace}
Given a finite set $I$ (with the discrete topology), let $\Emb(I,X)$ be the space of injective functions $I \to X$. There is an $\FI$-module $M$ with $M_I = \HH^k(\Emb(I,X);\bk)$ for $I \in \FI$. Observe that $M_{n} = \HH^k(\PConf_n(X);\bk)$.
\end{example}

\paragraph{General situation.}
More generally, let $\tC$ be any category.  A {\bf $\tC$-module} over
a ring $\bk$ is a functor $M$ from $\tC$ to $\Mod_{\bk}$ (see \cite{FFPSBook, FTBook} for surveys).  For each
$C \in \tC$, we thus have a $\bk$-module $M_C$, and for each $\tC$-morphism $f\colon C \rightarrow C'$ we
have a $\bk$-module morphism $M_f\colon M_{C} \rightarrow M_{C'}$.  These morphisms satisfy compatibility
conditions coming from the composition law in $\tC$.  The class
of $\tC$-modules over $\bk$ forms a category whose morphisms are natural transformations.  A morphism 
$\phi\colon M \rightarrow N$ of $\tC$-modules thus consists of $\bk$-module morphisms $\phi_C\colon M_C \rightarrow N_C$
for each $C \in \tC$ such that the diagram
\[\begin{CD}
M_C @>{\phi_C}>> N_C \\
@V{M_f}VV        @VV{N_f}V \\
M_{C'} @>{\phi_{C'}}>> N_{C'}\end{CD}\]
commutes for all $\tC$-morphisms $f\colon C \rightarrow C'$.
We will often omit the $\bk$ and simply speak of $\tC$-modules.

Notions such as $\tC$-submodules, surjective/injective morphisms, etc.\ are
defined pointwise. For example, a $\tC$-module morphism $\phi \colon M \to N$ 
is surjective/injective if for all objects $C$, the linear map $\phi_C \colon M_C \to N_C$ 
is surjective/injective.  A $\tC$-submodule is the image of an injective morphism of $\tC$-modules.
A $\tC$-module $M$ is {\bf finitely generated} if there exist $C_1,\ldots,C_n \in \tC$ and elements
$x_i \in M_{C_i}$ such that the smallest $\tC$-submodule
of $M$ containing the $x_i$ is $M$ itself.  We will say that the category
of $\tC$-modules over $\bk$ is {\bf locally Noetherian} if all submodules of finitely generated
$\tC$-modules over $\bk$ are finitely generated.  We will say
that the category of $\tC$-modules is {\bf locally Noetherian} if the category of
$\tC$-modules over $\bk$ is locally Noetherian for all Noetherian rings $\bk$.

\paragraph{Asymptotic structure.}
Let $M$ be a finitely generated $\FI$-module over a Noetherian ring. Church--Ellenberg--Farb--Nagpal \cite{ChurchEllenbergFarbNagpal} proved that
the category of $\FI$-modules is locally Noetherian, and used this to prove three things about $M$.
For $N \geq 0$, let $\FI^N$ be the full subcategory of $\FI$ spanned by $I \in \FI$ with
$|I| \leq N$.
\begin{compactitem}
\item (Injective representation stability) If $f\colon I \rightarrow J$ is an $\FI$-morphism,
then the homomorphism $M_f\colon M_I \rightarrow M_J$ is injective when $|I| \gg 0$.
\item (Surjective representation stability) If $f\colon I \rightarrow J$ is an $\FI$-morphism,
then the $\Sym_J$-orbit of the image of $M_f\colon M_I \rightarrow M_J$ spans $M_J$ when $|I| \gg 0$.
\item (Central stability) For $N \gg 0$, the functor $M$ is the left Kan extension
to $\FI$ of the restriction of $M$ to $\FI^N$.
\end{compactitem}
Informally, central stability says
that for $I \in \FI$ with $|I| > N$, the $\bk$-module
$M_I$ can be constructed from $M_{I'}$ for $|I'| \leq N$ using the ``obvious''
generators and the ``obvious'' relations.
It should be viewed as a sort of finite presentability condition.

\begin{remark}
The statement of central stability in \cite{ChurchEllenbergFarbNagpal} 
does not refer to Kan extensions, but is easily seen to be equivalent
to the above.  The original definition of central stability in
the first author's paper \cite{PutmanRepStabilityCongruence} appears to be
very different: it gives a presentation for $M_I$ in terms of those $M_{I'}$ with $|I'| \in \{|I|-1, |I|-2\}$.  One can show that
this is equivalent to the above statement.  We will
prove a more general result during the proof of 
Theorem \ref{maintheorem:asymptoticstructure} below.
\end{remark}

\paragraph{Conclusion.}
The upshot is that the locally Noetherian property implies that to describe how $\HH^k(\PConf_n(X);\bk)$ changes as $n \rightarrow \infty$, it was enough for Church--Ellenberg--Farb--Nagpal to prove that the $\FI$-module in Example \ref{example:configspace} is finitely generated.  Our goal is to generalize this kind of result to other settings.

\begin{remark}
In a different direction, Wilson \cite{WilsonFIW} has generalized the theory of $\FI$-modules
to the other classical families of Weyl groups.
\end{remark}

\subsection{Introduction to the rest of the introduction}

It will take some time to state our results, so to help orient the reader
we give a brief outline of the rest of the introduction and provide a summary of
our four main families of results.  We begin in \S \ref{section:introcategories} by introducing
the four main categories we study in this paper.  They take as input a commutative ring $R$ and
are called  $\VI(R)$, $\tV(R)$, $\wV(R)$, $\VIC(R,\Unit)$, and $\SI(R)$.  This section
states our first main family of theorems.
\begin{compactitem}
\item {\bf Theorems \ref{maintheorem:vinoetherian}/\ref{maintheorem:artinianconj}/\ref{maintheorem:wvnoetherian}/\ref{maintheorem:vicnoetherian}/\ref{maintheorem:sinoetherian}}: When $R$ is a finite
commutative ring, the categories of $\VI(R)$-, $\tV(R)$-, $\wV(R)$-, $\VIC(R,\Unit)$-, and $\SI(R)$-modules are locally Noetherian.
\end{compactitem}

\begin{remark}
Our proofs of Theorems \ref{maintheorem:vinoetherian}--\ref{maintheorem:sinoetherian} use analogues of the theory
of Gr\"{o}bner bases.  Such techniques are developed systematically by Sam--Snowden in \cite{tca-gb}; however, the results
of this paper can {\bf not} be proved using the main result of \cite{tca-gb}.  See Remark~\ref{rmk:no-term-order}
for more on this.
\end{remark}

\begin{remark}
\label{remark:nonnoetherian}
The condition that $R$ is finite is necessary; indeed, in Theorem \ref{maintheorem:nonnoetherian} we prove that
none of the above categories of modules are locally Noetherian when $R$ is infinite.
\end{remark}

In \S \ref{section:introasymptotic}, we introduce the language of complemented categories, which
provides a common framework for stating and proving our results.  Examples include $\FI$ and $\VIC(R,\Unit)$
and $\SI(R)$ (though not $\VI(R)$ or $\tV(R)$ or $\wV(R)$; we will explain how to handle
these in Remark \ref{remark:handlenotcomplemented} below).  This section states our second main theorem:
\begin{compactitem}
\item {\bf Theorem \ref{maintheorem:asymptoticstructure}}: If $\tC$ is a complemented category such that
the category of $\tC$-modules is locally Noetherian, then finitely generated $\tC$-modules have a simple
asymptotic structure analogous to Church--Farb's theorem from the previous section.
\end{compactitem}
The next section (\S \ref{section:introrepresentation}) provides large numbers of examples of
finitely generated $\tC$-modules to which we can apply the previous result and obtain interesting
asymptotic structure theorems.  Letting $\O_K$ be the ring of integers in an algebraic number field, the main theorems of this section are as follows:
\begin{compactitem}
\item {\bf Theorems \ref{maintheorem:glfinite}/\ref{maintheorem:slfinite}/\ref{maintheorem:autfinite}/\ref{maintheorem:spfinite}/\ref{maintheorem:modfinite}}: The homology groups of congruence subgroups in $\GL_n(\O_K)$, $\SL_n(\O_K)$, the
automorphism group of a free group, $\Sp_{2n}(\O_K)$, and the mapping class group form finitely generated
modules over appropriate categories.
\end{compactitem}
Combining these theorems with Theorem \ref{maintheorem:asymptoticstructure}, we obtain 
a version of representation stability for these homology groups.

\begin{remark}
The proofs of Theorems \ref{maintheorem:glfinite}--\ref{maintheorem:modfinite} all use a machine we develop later in this
paper.  The key input to this machine is an appropriate local Noetherianity theorem (like Theorems 
\ref{maintheorem:vinoetherian}--\ref{maintheorem:sinoetherian} above).  See Theorem \ref{theorem:finitegen} for a statement of how this machine works.
\end{remark}

The final section (\S \ref{section:introtwisted}) of our introduction concerns twisted homological stability.  Its main
theorems can be summarized as follows:
\begin{compactitem}
\item {\bf Theorems \ref{maintheorem:gltwisted}/\ref{maintheorem:sptwisted}}: If $R$ is a finite commutative ring, then
the twisted homology groups of $\GL_n(R)$, $\SL_n(R)$, and $\Sp_{2n}(R)$ with coefficients in a finitely
generated module over an appropriate category stabilize.
\end{compactitem}

\begin{remark}
The proofs of Theorems \ref{maintheorem:gltwisted} and \ref{maintheorem:sptwisted} both use another machine we develop later in this
paper to prove twisted homological stability theorems from appropriate local Noetherianity theorems.  See Theorem 
\ref{theorem:twistedstability} for a statement of how this machine works.
\end{remark}

\begin{remark}
Some readers may notice a parallel between Theorems \ref{theorem:finitegen} 
and \ref{theorem:twistedstability} and a program outlined independently by Church 
(in lectures and personal communication, but never published).  Church's program 
had similar inputs and would have obtained these same conclusions (finite generation 
for homology of congruence subgroups, and stability for twisted homology), but with 
one key difference: if successful, it would have yielded effective quantitative 
bounds (on the degree of generators in the former case, and on the stable range in the 
latter). However, Church informs us that technical difficulties arose that mean his 
proposed program cannot be carried out, and so it will not appear.  He points out that, 
although the results in this paper were obtained independently, they can be viewed after 
the fact as a demonstration that the framework of his approach can be salvaged by using 
our local Noetherianity theorems to avoid these technical difficulties, at the cost of 
giving up effective bounds.
\end{remark}

The introduction closes with some comments and questions together with an outline of the remainder of the paper.  This
last section also contains Theorem \ref{maintheorem:nonnoetherian}, which is the non-Noetherian result discussed
in Remark \ref{remark:nonnoetherian}.

\subsection{Categories of free modules: \texorpdfstring{$\bVI$}{VI}, 
\texorpdfstring{$\btV$}{V}, \texorpdfstring{$\bwV$}{V'},
\texorpdfstring{$\bVIC$}{VIC}, and \texorpdfstring{$\bSI$}{SI}}
\label{section:introcategories}

We now introduce the categories we will consider in this paper.  Fix a ring $R$ (assumed as always to be commutative).

\paragraph{The category $\bVI$.}
A linear map between free $R$-modules is {\bf splittable} if it has a left inverse.
Let $\VI(R)$ be the category whose objects are finite-rank free $R$-modules and
whose morphisms are splittable injections.  For $V \in \VI(R)$, the monoid of
$\VI(R)$-endomorphisms of $V$ is $\GL(V)$.  Thus if $M$ is a $\VI(R)$-module, then $M_V$ is a representation of $\GL(V)$ for all $V \in \VI(R)$.
The category $\VI(R)$ was introduced by Scorichenko in his thesis (\cite{ScorichenkoThesis}; see \cite{FFPSBook} for a published account).
Our main theorem about $\VI(R)$ is as follows.

\begin{maintheorem}
\label{maintheorem:vinoetherian}
Let $R$ be a finite ring.  Then the category of $\VI(R)$-modules is locally Noetherian.
\end{maintheorem}

\begin{remark}
In contrast, $\VI(\Z)$ is not locally Noetherian, and in fact $\VI(R)$ is not locally Noetherian whenever $R$ is an infinite ring; see Theorem~\ref{maintheorem:nonnoetherian} below.
\end{remark}

\begin{remark}
While preparing this article, we learned that Gan and Li \cite{ganli} have recently independently
shown that the category of functors $\VI(R) \to \Mod_{\bk}$ is locally Noetherian when $R$ is a finite field and $\bk$ is a field of characteristic $0$.  For our applications, it is important that both $R$ and especially $\bk$ be more general.
\end{remark}

\paragraph{A variant: the Artinian conjecture.}
Define $\tV(R)$ to be the category whose objects are finite-rank free $R$-modules and
whose morphisms are all splittable linear maps (not necessarily injective).  We then have the following.

\begin{maintheorem}[Artinian conjecture]
\label{maintheorem:artinianconj}
Let $R$ be a finite ring.  Then the category of $\tV(R)$-modules is
locally Noetherian.
\end{maintheorem}

Theorem \ref{maintheorem:artinianconj} generalizes an old
conjecture of Jean Lannes and Lionel Schwartz which asserts that the category
of $\tV(\Field_q)$-modules over $\Field_q$ is locally Noetherian.
This is often stated in a dual form and is thus
known as the Artinian conjecture.  It first appeared in print
in \cite[Conjecture 3.12]{KuhnGenericII} and now has a large literature
containing many partial results (mostly focused
on the case $q=2$, though even there nothing close to the full conjecture
was previously known).  See, e.g.,
\cite{KuhnInvariant, PowellArtinian, PowellStructure, PowellArtinianObjects,
DjamentFoncteurs, DjamentLeFoncteur}.  See also \cite[Chapter 5]{SchwartzBook}
for connections to the Steenrod algebra. A proof of this conjecture also appears in \cite{tca-gb}, see Remark~\ref{rmk:schwartz-alt-proof}.

A variant of $\tV(R)$ is the category $\wV(R)$ whose objects are finite-rank 
free $R$-modules and whose morphisms are all linear maps (not just the splittable ones). 
Of course, $\tV(R) = \wV(R)$ if $R$ is a field, but they differ in general.  For
$\wV(R)$, we have the following theorem.

\begin{maintheorem}
\label{maintheorem:wvnoetherian}
Let $R$ be a finite ring.  Then the category of $\wV(R)$-modules is
locally Noetherian.
\end{maintheorem}

\paragraph{Problems with $\bVI$.}
The category $\VI(R)$ is not rich enough to allow many natural constructions.  For instance,
the map $V \mapsto \GL(V)$ does not extend to a functor from $\VI(R)$ to
the category of groups.  The problem is that if $V$ is a submodule of $W$, then
there is no canonical way to extend an automorphism of $V$ to an automorphism of $W$.
We overcome this problem by introducing a richer category.

\paragraph{The category $\bVIC$, take I.}
Define $\VIC(R)$ to be the following category.
The objects of $\VIC(R)$ are finite-rank free $R$-modules $V$.  
If $V, W \in \VIC(R)$, then a $\VIC(R)$-morphism from
$V$ to $W$ consists of a pair $(f,C)$ as follows.
\begin{compactitem}
\item $f \colon  V \rightarrow W$ is an $R$-linear injection.
\item $C \subset W$ is a free $R$-submodule such that $W = f(V) \oplus C$. 
\end{compactitem}
If $(f,C) \colon V \to W$ and $(g,D) \colon W \to X$ are two $\VIC(R)$-morphisms, then the composition
$(g,D) \circ (f,C)$ equals $(g \circ f, g(C) \oplus D)$.

The monoid of $\VIC(R)$-endomorphisms of $V \in \VIC(R)$ is $\GL(V)$.  Observe that in contrast to $\VI(R)$, the assignment $V \mapsto \GL(V)$ extends to a functor from $\VIC(R)$ to
the category of groups: if $(f,C) \colon V \rightarrow W$ is a morphism, then there is an induced map $\GL(V) \rightarrow \GL(W)$ which extends automorphisms over $C$ by the identity, and this assignment is functorial.  The category $\VIC(R)$ was first introduced by Djament \cite{DjamentUnitary}.

\begin{remark}
\label{remark:vicalternate}
We will sometimes use an alternative description of the morphisms in $\VIC(R)$ which works whenever all stably free
$R$-modules are free (for
instance, when $R$ is a finite commutative ring).
Namely, the data $(f,C)$ which defines a morphism $V \to W$ is equivalent 
to a pair of $R$-linear maps $V \to W \to V$ whose composition is the identity.  Here 
the first map is $f$ and the second map is the projection $W \to V$ away from $C$.
\end{remark}

\paragraph{The category $\bVIC$, take II (cheap).}
Let ${\Unit} \subset R^\times$ be a subgroup of the group of multiplicative units of $R$ and let $\SL^{\Unit}(V) = \Set{$f \in \GL(V)$}{$\det(f) \in {\Unit}$}$.
For our applications, we will need a variant of $\VIC(R)$ whose automorphism groups
are $\SL^{\Unit}(V)$ rather than $\GL(V)$.  The category $\VIC(R)$ is equivalent to the full
subcategory generated by $\Set{$R^n$}{$n \geq 0$}$.  Define $\VIC(R,{\Unit})$ to be the category
whose objects are $\Set{$R^n$}{$n \geq 0$}$ and whose morphisms from $R^n$ to $R^{n'}$ are
$\VIC(R)$-morphisms $(f,C)\colon R^n \rightarrow R^{n'}$ such that $\det(f) \in {\Unit}$ if $n = n'$.
Thus the monoid of $\VIC(R,{\Unit})$-endomorphisms of $R^n$ is $\SL_n^{\Unit}(R) := \SL^{\Unit}(R^n)$.

\paragraph{Orientations.}
For many of our constructions, it will be useful to extend $\VIC(R,{\Unit})$ to a category whose
objects are all finite-rank free $R$-modules rather than just the $R^n$.
An {\bf orientation} of a rank $n$ free $R$-module $V$ is a generator
for the rank $1$ free $R$-module $\wedge^n V$.  The group of units of $R$ acts
transitively on the set of orientations on $V$, and a {\bf ${\Unit}$-orientation} on $V$
consists of an orbit under the action of ${\Unit}$.  A {\bf ${\Unit}$-oriented free $R$-module}
is a finite-rank free $R$ module $V$ equipped with a ${\Unit}$-orientation.
If $W_1$ and $W_2$ are ${\Unit}$-oriented free $R$-modules, then $W_1 \oplus W_2$ is 
in a natural way a ${\Unit}$-oriented free $R$-module.  A {\bf ${\Unit}$-oriented free $R$-submodule}
of $V$ is an $R$-submodule $V'$ of $V$ which is a free $R$-module 
equipped with a ${\Unit}$-orientation; if $V'=V$, then
the ${\Unit}$-orientations on $V'$ and $V$ must be the same, and if $V' \subsetneqq V$ 
then there is no condition on the ${\Unit}$-orientation of $V'$.

\paragraph{The category $\bVIC$, take II (better).}
The category $\VIC(R,{\Unit})$ can now be defined as follows (this yields 
a category equivalent to the
previous one).  The objects of $\VIC(R,{\Unit})$ are ${\Unit}$-oriented
free $R$-modules $V$.  If $V, W \in \VIC(R,\Unit)$, then a morphism of $\VIC(R,\Unit)$ 
from $V$ to $W$ consists of a pair $(f,C)$ as follows.
\begin{compactitem}
\item $f \colon  V \rightarrow W$ is an $R$-linear injection.
\item $C \subset W$ is a ${\Unit}$-oriented free $R$-submodule such that $W = f(V) \oplus C$.
\end{compactitem}
In the second bullet point, $f(V)$ is equipped with the ${\Unit}$-orientation $f_{\ast}(\omega)$, where
$\omega$ is the ${\Unit}$-orientation on $V$, and the direct sum decomposition $f(V) \oplus C$ is 
as ${\Unit}$-oriented free $R$-modules.
Note that when $C$ is nonzero, there is a unique ${\Unit}$-orientation on $C$ such that the ${\Unit}$-orientations on $f(V) \oplus C$ and $W$ are the same, so the morphisms in $\VIC(R,{\Unit})$ and $\VIC(R)$ are only different when $\dim(V) = \dim(W)$.
The monoid of $\VIC(R,{\Unit})$-endomorphisms of $V \in \VIC(R,{\Unit})$ is
$\SL^{\Unit}(V)$.  Observe that $\VIC(R,{\Unit}) \subseteq \VIC(R)$ with equality when ${\Unit}$ is the entire
group of units.
Our main theorem about $\VIC(R,\Unit)$ is as follows.  For $\VIC(\Field)$ with $\Field$ a finite
field, this theorem was conjectured by Church \cite{ChurchHomologicalAlgebra}.

\begin{maintheorem}
\label{maintheorem:vicnoetherian}
Let $R$ be a finite ring and let ${\Unit} \subset R$ be a subgroup of the group of units.  
Then the category of $\VIC(R,{\Unit})$-modules is locally Noetherian.
\end{maintheorem}

\begin{remark}
Just like for $\VI(R)$, it is necessary for $R$ to be finite; see
Theorem \ref{maintheorem:nonnoetherian}.
\end{remark}

\paragraph{Symplectic forms.}
We now turn to the symplectic analogue of $\VI(R)$-modules.
A {\bf symplectic form} on a finite-rank free $R$-module $V$ is 
a bilinear form $\ialg \colon  V \times V \rightarrow R$ which is alternating (i.e., $\ialg(v,v)=0$ for all $v \in V$) and which induces an isomorphism from $V$ to its dual
$V^{\ast} = \Hom_R(V,R)$.  A {\bf symplectic $R$-module} is a finite-rank free $R$-module equipped with a symplectic form.
A {\bf symplectic $R$-submodule} of a symplectic $R$-module $V$ is an $R$-submodule $W$
such that the symplectic form on $V$ restricts to a symplectic form on $W$.
A {\bf symplectic map} between
symplectic $R$-modules is an $R$-linear map that preserves the symplectic form.  Observe that symplectic maps must be injective.

\paragraph{The category $\bSI$.}
Define $\SI(R)$ to be the category whose objects are symplectic $R$-modules and
whose morphisms are symplectic maps.  The monoid of
$\SI(R)$-endomorphisms of $V \in \SI(R)$ is $\Sp(V)$, i.e., the
group of $R$-linear isomorphisms of $V$ that preserve the symplectic form.
Just like for $\VIC(R)$, the map
$V \mapsto \Sp(V)$ extends to a functor from $\SI(R)$ to the category of
groups: a symplectic map $f \colon V \rightarrow W$ induces a map $\Sp(V) \rightarrow \Sp(W)$ which extends automorphisms over $f(V)^{\perp}$ by the identity.
We then have the following.

\begin{maintheorem}
\label{maintheorem:sinoetherian}
Let $R$ be a finite ring.  Then the category of $\SI(R)$-modules is
locally Noetherian.
\end{maintheorem}

\begin{remark}
Just like for $\VI(R)$, it is necessary for $R$ to be finite; see
Theorem \ref{maintheorem:nonnoetherian}.
\end{remark}

\subsection{Complemented categories and their asymptotic structure}
\label{section:introasymptotic}

Our next result says that finitely generated modules over the above categories have
a simple asymptotic structure (like the one for $\FI$-modules we discussed earlier).  To
avoid having to give separate proofs for all our categories, we introduce an abstract
framework.

\paragraph{Complemented categories.}
A {\bf weak complemented category} is a monoidal category $(\tA,\monpro)$ satisfying:

\begin{compactitem}
\item Every morphism in $\tA$ is a monomorphism.  Thus
for all morphisms $f \colon V \rightarrow V'$,
it makes sense to talk about the subobject $f(V)$ of $V'$.
\item The identity object $\mathbbm{1}$ of $(\tA,\monpro)$ is initial.  Thus
for $V,V' \in \tA$ there exist natural morphisms $V \rightarrow V \monpro V'$
and $V' \rightarrow V \monpro V'$, namely the compositions
\[V \xrightarrow{\cong} V \monpro \mathbbm{1} \longrightarrow V \monpro V' \quad \quad \text{and} \quad \quad V' \xrightarrow{\cong} \mathbbm{1} \monpro V' \longrightarrow V \monpro V'.\]
We will call these the {\bf canonical morphisms}.
\item For $V,V',W \in \tA$, the map $\Hom_{\tA}(V \monpro V',W) \rightarrow \Hom_{\tA}(V,W) \times \Hom_{\tA}(V',W)$
obtained by composing morphisms with the canonical morphisms is an injection.
\item Every subobject $C$ of an object $V$ has a unique {\bf complement}, that is, a unique subobject $D$ of $V$ such that there is an isomorphism $C \monpro D \xrightarrow{\cong} V$ where the compositions $C \rightarrow C \monpro D \xrightarrow{\cong} V$ and
$D \rightarrow C \monpro D \xrightarrow{\cong} V$ of the isomorphism with the canonical morphisms are the inclusion morphisms.
\item The canonical map $\Sym_p \wr \Aut(X) \to \Aut(X^p)$ is injective. (Here $\wr$ denotes wreath product.)
\end{compactitem}
\noindent
A {\bf complemented category} is a weak complemented category $(\tA,\monpro)$ 
whose monoidal structure is symmetric (more precisely, which is equipped
with a symmetry).

\begin{remark}
Things like complemented categories go back to work of Quillen; see \cite[p.\ 3]{GraysonKTheoryII}.
After a version of this paper was distributed, Nathalie Wahl (and later with Oscar Randal-Williams) posted her paper \cite{WahlStability}, which uses an axiomatization similar to complemented categories to prove homological stability results of a more classical flavor than ours. A related axiomatization also appears in work of Djament--Vespa \cite{DjamentVespa1}.
\end{remark}

\paragraph{Generators for a category.}
If $(\tA,\monpro)$ is a monoidal category, then a {\bf generator} for $\tA$ is an object $X$ of $\tA$ such that all objects $V$ of $\tA$ are isomorphic to $X^i$ for some unique $i \geq 0$.  We will call $i$ the
{\bf $X$-rank} of $V$.

\begin{example}
\label{example:fi}
The category $\FI$ of finite sets and injections is a complemented category whose monoidal structure is the disjoint union. A generator is the set $\{1\}$.
\end{example}

\begin{example}
\label{example:vic}
The category $\VIC(R)$ is a complemented category whose monoidal structure is the direct sum. The subobject attached to a morphism $V \to V'$ can be interpreted as the pair $(W, W')$ where $W$ is the image of the linear injection $V \to V'$ and $W'$ is the chosen complement. The complement to this subobject is the pair $(W', W)$. A generator is $R^1$.
\end{example}

\begin{example}
\label{example:si}
The category $\SI(R)$ is a complemented category whose monoidal structure is the orthogonal direct sum.  The complement of a symplectic submodule $W$ of a symplectic $R$-module $(V, \ialg)$ is $W^{\perp} = \Set{$v \in V$}{$\ialg(v,w)=0$ for all $w \in W$}$. A generator is $R^2$ equipped with the standard symplectic form.
\end{example}

\begin{remark}
\label{remark:handlenotcomplemented}
The category $\VI(R)$ is not a complemented category.  However, there is
a forgetful functor $\Psi\colon \VIC(R) \rightarrow \VI(R)$.  If $R$ is a finite ring and
$M$ is a finitely generated $\VI(R)$-module, then the pullback $\Psi^{\ast}(M)$ is a finitely generated $\VIC(R)$-module (see the proof of Theorem~\ref{maintheorem:vinoetherian} in \S\ref{section:noetherianvect}) which can be analyzed using the technology we will discuss below.  A similar remark applies to $\tV(R)$ and $\wV(R)$.
\end{remark}

\paragraph{Asymptotic structure theorem.}
The following theorem generalizes the aforementioned asymptotic structure theorem for $\FI$-modules. 

\begin{maintheorem}
\label{maintheorem:asymptoticstructure}
Let $(\tA,\monpro)$ be a complemented category with generator $X$.
Assume that the category of $\tA$-modules is locally Noetherian,
and let $M$ be a finitely generated $\tA$-module.  Then the following hold.
For $N \geq 0$, let $\tA^N$ denote the full subcategory of $\tA$
spanned by elements whose $X$-ranks are at most $N$.
\begin{compactitem}
\item (Injective representation stability) If $f\colon V \rightarrow W$ is an $\tA$-morphism,
then the homomorphism $M_f\colon M_V \rightarrow M_W$ is injective when the $X$-rank of $V$ is
sufficiently large.
\item (Surjective representation stability) If $f\colon V \rightarrow W$ is an $\tA$-morphism,
then the orbit under $\Aut_{\tA}(W)$ of the image of $M_f\colon M_V \rightarrow M_W$ spans $M_W$ when the $X$-rank of $V$ is sufficiently large.
\item (Central stability)
For $N \gg 0$, the functor $M$ is the left Kan extension
to $\tA$ of the restriction of $M$ to $\tA^N$.
\end{compactitem}
\end{maintheorem}

\begin{remark}
See Theorem \ref{theorem:resolution} below for a concrete description
of the presentation for $M$ implied by central stability; the proof of Theorem \ref{maintheorem:asymptoticstructure}
immediately follows the proof of Theorem \ref{theorem:resolution}.
\end{remark}

\subsection{Representation stability for congruence subgroups}
\label{section:introrepresentation}

Our next theorems apply the above technology to prove representation-theoretic
analogues of homological stability for certain kinds of congruence subgroups.

\paragraph{GL over number rings.}
Let $\O_K$ be the ring of integers
in an algebraic number field $K$.  Charney \cite{CharneyDedekind} proved that $\GL_n(\O_K)$
satisfies homological stability.
Rationally, these stable values were computed by Borel \cite{BorelStability}.
Let $\alpha \subset \O_K$ be a proper nonzero ideal.  The {\bf level
$\alpha$ congruence subgroup} of $\GL_n(\O_K)$,
denoted $\GL_n(\O_K,\alpha)$, is the kernel of the
map $\GL_n(\O_K) \rightarrow \GL_n(\O_K/\alpha)$.  It follows from Borel's
work \cite{BorelStability} that $\HH_k(\GL_n(\O_K,\alpha);\bk)$ stabilizes
when $\bk = \Q$; however, Lee--Szczarba
\cite{LeeSzczarbaCongruence} proved that it does not stabilize integrally,
even for $k=1$.  Few of these homology groups are known,
and the existing computations do not seem to fit into any
pattern.  However, they have an important additional piece of structure: for each
$n \geq 1$, the conjugation action of $\GL_n(\O_K)$ on $\GL_n(\O_K,\alpha)$
descends to an action of $\SL_n^{\Unit}(\O_K/\alpha)$ on $\HH_k(\GL_n(\O_K,\alpha);\bk)$, where
${\Unit} \subset \O_K/\alpha$ is the image of the group of units of $\O_K$.
Our goal is to describe how this representation changes as $n$ increases.

\paragraph{Congruence subgroups of GL as VIC-modules.}
As we discussed above, there is a functor from $\VIC(\O_K)$ to the category
of groups that takes $V$ to $\GL(V)$.  Similarly, there is a functor $\GL(\O_K,\alpha)$
from $\VIC(\O_K)$ to the category of groups defined by the formula
\[\GL(\O_K,\alpha)_V = \Ker(\GL(V) \rightarrow \GL(V \otimes_{\O_K} (\O_K/\alpha))).\]
Passing to homology, we get a $\VIC(\O_K)$-module $\HHH_k(\GL(\O_K,\alpha);\bk)$
with
\[\HHH_k(\GL(\O_K,\alpha);\bk)_V = \HH_k(\GL(\O_K,\alpha)_V;\bk) \quad \quad (V \in \VIC(\O_K)).\]
Since $\O_K$ is infinite, the category of $\VIC(\O_K)$-modules is poorly
behaved.  However,
we will prove in \S \ref{section:glfinite} that $\HHH_k(\GL(\O_K,\alpha);\bk)$
actually descends to a $\VIC(\O_K/\alpha,{\Unit})$-module, where ${\Unit} \subset \O_K/\alpha$ is the image
of the group of units of $\O_K$.  Since $\O_K/\alpha$ is finite,
our theorems apply to $\VIC(\O_K/\alpha,{\Unit})$.  All that is necessary is
to prove that the $\VIC(\O_K/\alpha,{\Unit})$-module
$\HHH_k(\GL(\O_K,\alpha);\bk)$ is finitely generated.

\paragraph{Relation to $\bFI$-modules.}
It turns out that the needed finite generation follows from known results.
The conjugation action of the set of permutation matrices
on $\GL_n(\O_K,\alpha)$
turns $\HH_k(\GL_n(\O_K,\alpha);\bk)$ into a representation of $\Sym_n$.
In \cite{PutmanRepStabilityCongruence}, the first author proved that if $\bk$ is
a field whose characteristic is sufficiently large, then the $\Sym_n$-representations
$\HH_k(\GL_n(\O_K,\alpha);\bk)$ satisfy a version of central stability.
This was generalized to arbitrary Noetherian rings $\bk$ by
Church--Ellenberg--Farb--Nagpal \cite{ChurchEllenbergFarbNagpal}, who actually
proved that the groups $\HH_k(\GL_n(\O_K,\alpha);\bk)$ form a finitely generated
$\FI$-module.  More precisely, the following defines an embedding
of categories $\Phi\colon \FI \rightarrow \VIC(\O_K/\alpha,\Unit)$.
\begin{compactitem}
\item For $I \in \FI$, set $\Phi(I) = (\O_K/\alpha)^I$.
\item For an $\FI$-morphism $\sigma\colon I \hookrightarrow J$, set $\Phi(\sigma)=(f_{\sigma},C_{\sigma})$, where
$f_{\sigma}\colon (\O_K/\alpha)^I \rightarrow (\O_K/\alpha)^J$ is defined by 
$f_{\sigma}(\bb_i) = \bb_{\sigma(i)}$ for $i \in I$ and $C_{\sigma} = \langle \text{$\bb_j$ $|$ $j \notin \sigma(I)$} \rangle$.
Here $\Set{$\bb_i$}{$i \in I$}$ and $\Set{$\bb_j$}{$j \in J$}$ are the standard bases for
$(\O_K/\alpha)^I$ and $(\O_K/\alpha)^J$, respectively.
\end{compactitem}
Then \cite{ChurchEllenbergFarbNagpal} proves that the restriction of $\VIC(\O_K/\alpha,\Unit)$ to
$\Phi(\FI) \subset \VIC(\O_K/\alpha,\Unit)$ is finitely generated, so
the $\VIC(\O_K/\alpha,{\Unit})$-module $\HHH_k(\GL(\O_K,\alpha);\bk)$ is finitely generated:

\begin{maintheorem}
\label{maintheorem:glfinite}
Let $\O_K$ be the ring of integers in an algebraic number field $K$, let
$\alpha \subset \O_K$ be a proper nonzero ideal, let ${\Unit} \subset \O_K/\alpha$ be the
image of the group of units of $\O_K$, and let $\bk$ be a Noetherian
ring.  Then the $\VIC(\O_K/\alpha,{\Unit})$-module
$\HHH_k(\GL(\O_K,\alpha);\bk)$ is finitely generated for all $k \geq 0$.
\end{maintheorem}

\noindent
We can therefore apply Theorem \ref{maintheorem:asymptoticstructure} to
$\HHH_k(\GL(\O_K,\alpha);\bk)$.  The result is an asymptotic structure
theorem for $\HH_k(\GL_n(\O_K,\alpha);\bk)$.  This partially verifies
a conjecture of Church--Farb \cite[Conjecture 8.2.1]{ChurchFarbRepStability}.
As a simple illustration of our machinery, we will give a direct proof of
Theorem \ref{maintheorem:glfinite} in this paper.

\begin{remark}
Church--Ellenberg--Farb--Nagpal's theorem gives a
central stability description of
$\HH_k(\GL_n(\O_K,\alpha);\bk)$ as an $\FI$-module; however, the central
stability conclusion from Theorem \ref{maintheorem:asymptoticstructure} gives
more information than this since it
deals with $\HH_k(\GL_n(\O_K,\alpha);\bk)$ as a representation of
$\SL_n^{\Unit}(\O_K/\alpha)$ and not merely as a representation of $\Sym_n$.
\end{remark}

\begin{remark}
The other examples that we discuss also form $\FI$-modules, but it seems
very hard to prove that they are finitely generated as $\FI$-modules.
\end{remark}

\begin{remark}
One could also consider congruence subgroups of $\SL_n(R)$ for other rings $R$.  In
fact, the first author's results in \cite{PutmanRepStabilityCongruence} work
in this level of generality; however, we have to restrict ourselves to rings
of integers so we can cite a theorem of Borel--Serre \cite{BorelSerreCorners}
which asserts that the homology groups in question are all finitely generated
$\bk$-modules.  This is the only property of rings of integers we use, and our
methods give a similar result whenever this holds (another example would be
$R$ a finite commutative ring).
\end{remark}

\paragraph{Variant: SL.}
We can define $\SL_n(\O_K,\alpha) \subset \SL_n(\O_K)$ in the obvious
way.  The homology groups $\HH_k(\SL_n(\O_K,\alpha);\bk)$ are now representations of
$\SL_n(\O_K/\alpha) = \SL_n^1(\O_K/\alpha)$.  Similarly to $\GL_n(\O_K,\alpha)$, we will 
construct a
$\VIC(\O_K/\alpha,1)$-module $\HHH_k(\SL(\O_K,\alpha);\bk)$ with
\[\HHH_k(\SL(\O_K,\alpha);\bk)_{\O_K^n} = \HH_k(\SL_n(\O_K,\alpha);\bk) \quad \quad (n \geq 0)\]
and prove the following variant of Theorem \ref{maintheorem:glfinite}.

\begin{maintheorem}
\label{maintheorem:slfinite}
Let $\O_K$ be the ring of integers in an algebraic number field $K$, let
$\alpha \subset \O_K$ be a proper nonzero ideal, and let $\bk$ be a Noetherian
ring.  Then the $\VIC(\O_K/\alpha,1)$-module
$\HHH_k(\SL(\O_K,\alpha);\bk)$ is finitely generated for all $k \geq 0$.
\end{maintheorem}

\begin{remark}
Just like for Theorem \ref{maintheorem:glfinite}, this could be deduced
from work of Church--Ellenberg--Farb--Nagpal \cite{ChurchEllenbergFarbNagpal},
though we will give a proof using our machinery.
\end{remark}

\paragraph{Automorphism groups of free groups.}
Our second example of a $\VIC(R,\Unit)$-module comes from the automorphism
group $\Aut(F_n)$ of the free group $F_n$ of rank $n$.  Just like for
$\GL_n(\O_K)$, work of Hatcher \cite{HatcherAut} and
Hatcher--Vogtmann \cite{HatcherVogtmannCerf} shows that $\Aut(F_n)$ satisfies homological stability. These stable values were computed by Galatius \cite{GalatiusAut}, who proved in particular that they rationally vanish.  For $\ell \geq 2$, the
{\bf level $\ell$ congruence subgroup} of $\Aut(F_n)$, denoted
$\Aut(F_n,\ell)$, is the kernel of the map
$\Aut(F_n) \rightarrow \SL_n^{\pm 1}(\Z/\ell)$ arising from the action
of $\Aut(F_n)$ on $\HH_1(F_n;\Z/\ell)$.
Satoh \cite{SatohCongruence} proved that $\HH_1(\Aut(F_n, \ell); \Z)$ does not stabilize.  Similarly to $\SL_n(\O_K,\alpha)$, few concrete computations are known.

\begin{remark}
Unlike for $\SL_n(\O_K,\alpha)$ it is not known in general 
if $\HH_k(\Aut(F_n,\ell);\Q)$ stabilizes,
though this does hold for $k \leq 2$ (see \cite{DayPutman}).
\end{remark}

\noindent
The conjugation action of $\Aut(F_n)$ on $\Aut(F_n,\ell)$ descends to an action
of $\SL_n^{\pm 1}(\Z/\ell)$ on $\HH_k(\Aut(F_n,\ell);\bk)$.
In \S \ref{section:autfinite}, we will construct a $\VIC(\Z/\ell,\pm 1)$-module
$\HHH_k(\Aut(\ell);\bk)$ with
\[
\HHH_k(\Aut(\ell);\bk)_{(\Z/\ell)^n} = \HH_k(\Aut(F_n,\ell);\bk) \quad \quad (n \geq 0).
\]
Our main theorem about this module is as follows.

\begin{maintheorem}
\label{maintheorem:autfinite}
Fix some $\ell \geq 2$, and let $\bk$ be a Noetherian ring.
Then the $\VIC(\Z/\ell,\pm 1)$-module $\HHH_k(\Aut(\ell);\bk)$ is
finitely generated for all $k \geq 0$.
\end{maintheorem}

\noindent
Applying Theorem \ref{maintheorem:asymptoticstructure}, we 
get an asymptotic structure theorem for $\HH_k(\Aut(F_n,\ell);\bk)$,
partially verifying a conjecture of Church--Farb (see the sentence after \cite[Conjecture 8.5]{ChurchFarbRepStability}).

\begin{remark}
The kernel of $\Aut(F_n) \to \GL_n(\Z)$ is denoted ${\rm IA}_n$ and its homology forms 
a module over $\VIC(\Z)$.  It would be interesting to understand this module; however, this seems
to be a very difficult problem.  Our techniques cannot handle this example because $\Z$ is an infinite ring (c.f.\ Theorem
\ref{maintheorem:nonnoetherian}).
\end{remark}

\paragraph{Symplectic groups over number rings.}
Let $\O_K$ and $K$ and $\alpha \subset \O_K$ be as above.
The {\bf level $\alpha$ congruence
subgroup} of $\Sp_{2n}(\O_K)$, denoted $\Sp_{2n}(\O_K,\alpha)$, is the kernel of the
map $\Sp_{2n}(\O_K) \rightarrow \Sp_{2n}(\O_K/\alpha)$.
The groups $\Sp_{2n}(\O_K)$ and $\Sp_{2n}(\O_K,\alpha)$ are similar
to $\GL_n(\O_K)$ and $\GL_n(\O_K,\alpha)$.  In particular, the following hold.
\begin{compactitem}
\item Charney \cite{CharneyVogtmann} proved that $\Sp_{2n}(\O_K)$ satisfies homological stability.
\item Borel \cite{BorelStability} calculated the stable rational homology of $\Sp_{2n}(\O_K)$.  He
also proved that $\HH_k(\Sp_{2n}(\O_K,\alpha);\bk)$ stabilizes when $\bk = \Q$.
\item For general $\bk$, an argument similar to the one that Lee--Szczarba \cite{LeeSzczarbaCongruence}
used for $\SL_n(\O_K,\alpha)$ shows that $\HH_k(\Sp_{2n}(\O_K,\alpha);\bk)$ does not stabilize,
even for $k=1$.
\item The conjugation action of $\Sp_{2n}(\O_K)$ on $\Sp_{2n}(\O_K,\alpha)$ induces an
action of $\Sp_{2n}(\O_K/\alpha)$ on $\HH_k(\Sp_{2n}(\O_K,\alpha);\bk)$.
\end{compactitem}
In \S \ref{section:spfinite}, we will construct an $\SI(\O_K/\alpha)$-module
$\HHH_k(\Sp(\O_K,\alpha);\bk)$ such that
\[\HHH_k(\Sp(\O_K,\alpha);\bk)_{(\O_K/\alpha)^{2n}} = \HH_k(\Sp_{2n}(\O_K,\alpha);\bk) \quad \quad (n \geq 0).\]
Our main result concerning this functor is as follows.

\begin{maintheorem}
\label{maintheorem:spfinite}
Let $\O_K$ be the ring of integers in an algebraic number field $K$, let
$\alpha \subset \O_K$ be a proper nonzero ideal, and let $\bk$ be a Noetherian
ring.  Then the $\SI(\O_K/\alpha)$-module
$\HHH_k(\Sp(\O_K,\alpha);\bk)$ is finitely generated for all $k \geq 0$.
\end{maintheorem}

\noindent
Applying Theorem \ref{maintheorem:asymptoticstructure}, we obtain an asymptotic structure
theorem for $\HH_k(\Sp_{2n}(\O_K,\alpha);\bk)$.  This partially verifies a conjecture
of Church--Farb \cite[Conjecture 8.2.2]{ChurchFarbRepStability}.

\paragraph{The mapping class group.}
Let $\Surface{g}{b}$ be a compact oriented genus
$g$ surface with $b$ boundary components.  The {\bf mapping class group}
of $\Surface{g}{b}$, denoted $\MCG_g^b$, is
the group of homotopy classes of orientation-preserving homeomorphisms of $\Surface{g}{b}$ that
restrict to the identity on $\partial \Surface{g}{b}$.  This is one of the basic objects
in low--dimensional topology; see \cite{FarbMargalitPrimer} for a survey.
Harer \cite{HarerStability} proved that the mapping class group satisfies a form of homological stability.
If $i\colon \Surface{g}{b} \hookrightarrow \Surface{g'}{b'}$
is a subsurface inclusion, then
there is a map $i_{\ast}\colon\MCG_g^b \rightarrow \MCG_{g'}^{b'}$ that extends mapping classes by the identity.  Harer's
theorem says that for all $k \geq 1$, the induced
map $\HH_k(\MCG_g^b;\bk) \rightarrow \HH_k(\MCG_{g'}^{b'};\bk)$ is an isomorphism for
$g \gg 0$.
For $\bk = \Q$, the stable homology of the mapping
class group was calculated by Madsen--Weiss \cite{MadsenWeiss}.

\paragraph{Congruence subgroups of MCG.}
Fix some $\ell \geq 2$.  The group $\MCG_g^b$ acts on $\HH_1(\Surface{g}{b};\Z/\ell)$.  This action preserves the algebraic intersection
pairing $\ialg(\cdot,\cdot)$.  If $b=0$, then Poincar\'{e} duality implies that $\ialg(\cdot,\cdot)$ is symplectic.  If instead $b=1$, then
gluing a disc to the boundary component does not change $\HH_1(\Surface{g}{b};\Z/\ell)$, so $\ialg(\cdot,\cdot)$ is still
symplectic.  For $g \geq 0$ and $0 \leq b \leq 1$, the action of $\MCG_g^b$
on $\HH_1(\Surface{g}{b};\Z/\ell) \cong (\Z/\ell)^{2g}$ thus induces a representation
$\MCG_g^b \rightarrow \Sp_{2g}(\Z/\ell)$.  It is classical \cite[\S 6.3.2]{FarbMargalitPrimer}
that this is surjective.  Its kernel
is the {\bf level $\ell$ congruence subgroup} $\Mod_g^b(\ell)$ of $\MCG_g^b$.  It
is known that $\HH_k(\MCG_g^b(\ell);\bk)$ does not stabilize, even for $k=1$ (see
\cite[Theorem H]{PutmanPicardGroupLevel}).  Also, the conjugation action of $\MCG_g^b$ on
$\MCG_g^b(\ell)$ induces an action of $\Sp_{2g}(\Z/\ell)$ on $\HH_k(\MCG_g^b(\ell);\bk)$.  

\begin{remark}
For $b > 1$, the algebraic intersection pairing is not symplectic, so we do not get a symplectic representation.  We will not discuss this case.
\end{remark}

\paragraph{Stability.}
We will restrict ourselves to the case where $b=1$; the problem with closed surfaces is that they cannot
be embedded into larger surfaces, so one cannot ``stabilize'' them.  In \S \ref{section:modfinite}, we will
construct an $\SI(\Z/\ell)$-module $\HHH_k(\MCG(\ell);\bk)$ such that
\[\HHH_k(\MCG(\ell);\bk)_{(\Z/\ell)^{2g}} = \HH_k(\MCG_g^1(\ell);\bk) \quad \quad (g \geq 0).\]
Our main theorem concerning this module is as follows.

\begin{maintheorem}
\label{maintheorem:modfinite}
Fix some $\ell \geq 2$, and let $\bk$ be a Noetherian ring.
Then the $\SI(\Z/\ell)$-module $\HHH_k(\MCG(\ell);\bk)$ is
finitely generated for all $k \geq 0$.
\end{maintheorem}

\noindent
Applying Theorem
\ref{maintheorem:asymptoticstructure},
we get an asymptotic structure theorem for $\HH_k(\MCG_g^1(\ell);\bk)$.  This partially verifies
a conjecture of Church--Farb \cite[Conjecture 8.5]{ChurchFarbRepStability}.

\begin{remark}
The kernel of $\MCG^1_g \to \Sp_{2g}(\Z)$ is the Torelli group, denoted ${\mathcal I}^1_g$, and its homology forms a module over $\SI(\Z)$.  Just like for ${\rm IA}_n$, we would
like to understand this module but cannot do so with our techniques since
$\Z$ is an infinite ring.
\end{remark}

\subsection{Twisted homological stability}
\label{section:introtwisted}

We now discuss twisted homological stability.  We wish
to thank Aur\'elien Djament for his help with
this section.

\paragraph{Motivation: the symmetric group.}
A classical theorem of Nakaoka \cite{NakaokaStability} says that
$\Sym_n$ satisfies classical homological stability.  Church 
\cite{ChurchHomologicalAlgebra}
proved a very general version of this with twisted coefficients.  Namely, he
proved that if $M$ is a finitely generated $\FI$-module and $M_n = M_{[n]}$, then
for all $k \geq 0$ the map $\HH_k(\Sym_n;M_n) \rightarrow \HH_k(\Sym_{n+1};M_{n+1})$
is an isomorphism for $n \gg 0$ (in fact, he obtained linear bounds for when
stability occurs).  We remark that this applies to much more general
coefficient systems than earlier work of Betley \cite{BetleyStable}.
The key to Church's theorem is the fact that the category of $\FI$-modules is
locally Noetherian.  We will abstract Church's argument to other categories (though
our results will be weaker in that we will not give bounds for when stability
occurs); the possibility of doing so was one of Church's motivations
for conjecturing Theorem \ref{maintheorem:vicnoetherian}.

\paragraph{General linear group.}
In \cite{VanDerKallenStability}, van der Kallen proved that for
rings $R$ satisfying a mild condition (for instance, $R$ can be a finite
ring), $\GL_n(R)$ satisfies homological stability.
The paper \cite{VanDerKallenStability} also shows that
this holds for certain twisted coefficient systems (those satisfying a
``polynomial'' condition introduced by Dwyer \cite{DwyerStability}; see
\cite{WahlStability} for a general framework for arguments like Dwyer's using
the language of functor categories).  The arguments
in \cite{VanDerKallenStability} also work for $\SL_n^{\Unit}(R)$ for
subgroups ${\Unit} \subset R$ of the group of units.
We will prove the following generalization of van der Kallen's result (at least
for finite commutative rings).  
If $M$ is a $\VIC(R,\Unit)$-module, then denote by $M_n$ the value
$M_{R^n}$.

\begin{maintheorem}
\label{maintheorem:gltwisted}
Let $R$ be a finite ring, let $\Unit \subset R$ be a subgroup
of the group of units, and let $M$ be a finitely generated $\VIC(R,\Unit)$-module
over a Noetherian ring.
Then for all $k \geq 0$, the map
$\HH_k(\SL_n^{\Unit}(R);M_n) \rightarrow \HH_k(\SL_{n+1}^{\Unit}(R);M_{n+1})$
is an isomorphism for $n \gg 0$.
\end{maintheorem}

\begin{remark}
This is more general than van der Kallen's result: if $\bk$ is a field, then
his hypotheses on the coefficient systems $M_n$ imply that $\dim_{\bk}(M_n)$
grows like a polynomial in $n$, which need not hold for finitely generated
$\VIC(R,\Unit)$-modules.  For instance, one can define a finitely generated
$\VIC(R,\Unit)$-module whose dimensions grow exponentially via 
\[M_n = \bk[\Hom_{\VIC(R,\Unit)}(R^1,R^n)]. \qedhere\]
\end{remark}

\paragraph{The symplectic group.}
Building on work of Charney \cite{CharneyVogtmann},
Mirzaii--van der Kallen \cite{MirzaiiVanDerKallenStability} proved that for many
rings $R$ (including all finite rings), the groups $\Sp_{2n}(R)$ satisfy
homological stability.  In fact, though Mirzaii--van der Kallen do not
explicitly say this, Charney's paper shows that the results of
\cite{MirzaiiVanDerKallenStability} also apply to twisted coefficients
satisfying an appropriate analogue of Dwyer's polynomial condition.
We will prove the following generalization of
this.  If $M$ is an $\SI(R)$-module, then denote by $M_n$ the value
$M_{R^{2n}}$.

\begin{maintheorem}
\label{maintheorem:sptwisted}
Let $R$ be a finite ring and let $M$ be a finitely generated
$\SI(R)$-module over a Noetherian ring.  Then for all $k \geq 0$, the
map
$\HH_k(\Sp_{2n}(R);M_n) \rightarrow \HH_k(\Sp_{2n+2}(R);M_{n+1})$
is an isomorphism for $n \gg 0$.
\end{maintheorem}

\begin{remark}
Examples similar to the one we gave after Theorem \ref{maintheorem:gltwisted} show
that this is more general than Mirzaii--van der Kallen's result.
\end{remark}

\subsection{Comments and future directions}

We now discuss some related results and open questions.

\begin{remark}[Homological stability] \label{rmk:hom-stab}
The examples in Theorems~\ref{maintheorem:glfinite}, \ref{maintheorem:slfinite}, \ref{maintheorem:autfinite}, \ref{maintheorem:spfinite}, \ref{maintheorem:modfinite} can be interpreted when $\ell=1$ or when $\alpha$ is the unit ideal. In this case, we interpret the categories $\VIC(0) = \SI(0)$ as follows: there is one object for each nonnegative integer $n \ge 0$ and there is a unique morphism $n \to n'$ whenever $n' \ge n$. Then representation stability of the relevant homology groups is the usual homological stability statements: for example, for each $i \ge 0$, and $n \gg 0$ the map $\HH_i(\SL_n(\Z); \bk) \to \HH_i(\SL_{n+1}(\Z); \bk)$ is an isomorphism.
\end{remark}

\paragraph{Infinite rings.}
We have repeatedly emphasized that it is necessary for our rings $R$ to be finite.  Our final
theorem explains why this is necessary.

\begin{maintheorem}
\label{maintheorem:nonnoetherian}
Let $R$ be an infinite ring. For ${\tt C} \in \{\VI(R), \VIC(R), \SI(R)\}$, the category of ${\tt C}$-modules is not locally Noetherian.
\end{maintheorem}

\noindent
This naturally leads to the question of how to develop representation stability results for categories such as $\VIC(R)$, etc.\ when $R$ is an infinite ring. 

\paragraph{Bounds.}
Our asymptotic structure theorem does not provide bounds for when it begins.  
This is an unavoidable
artifact of our proof, which make essential use of local Noetherianity.  The results of the
first author in \cite{PutmanRepStabilityCongruence} concerning central stability for
the homology groups of congruence subgroups (as representations of $\Sym_n$) are different and
do give explicit bounds.  It would be interesting to prove versions of our theorems
with explicit bounds.

\paragraph{Orthogonal and unitary categories.}
It is also possible to include orthogonal versions of our categories, i.e., finite rank free $R$-modules equipped with a nondegenerate orthogonal form. This is a richer category than the previous ones because the rank of an orthogonal module no longer determines its isomorphism type. While many of the ideas and constructions of this paper should be relevant, it will be necessary to generalize them further to include the example of categories of orthogonal modules. Similarly, one can consider unitary versions. We have omitted discussion of these variations for
reasons of space.

\paragraph{Twisted commutative algebras.}
The category of $\FI$-modules fits into the more general framework of modules over twisted commutative algebras (these are commutative algebras in the category of functors from the groupoid of finite sets) and this perspective is used in \cite{symc1, expos, infrank} to develop the basic properties of these categories. Similarly, the categories in this paper fit into a more general framework of modules over algebras which generalize twisted commutative algebras where the groupoid of finite sets is replaced by another groupoid (see Remark~\ref{rmk:groupoid-alg}). 
When $\bk$ is a field of characteristic $0$, twisted commutative algebras have alternative concrete descriptions in terms of $\GL_\infty(\bk)$-equivariant commutative algebras (in the usual sense) thanks to Schur--Weyl duality. For the other cases in our paper, no such alternative description seems to be available, so it would be interesting to understand these more exotic algebraic structures.

\subsection{Outline and strategy of proof}

We now give an outline of the paper and comment on our proofs.  Let $\tC$ be 
one of the categories we discussed above, i.e., $\VI(R)$, $\tV(R)$, $\wV(R)$, $\VIC(R,\Unit)$, or $\SI(R)$.
The first step is to establish that the category of $\tC$-modules is locally Noetherian, i.e., that the finite generation property is inherited by submodules.
This is done in \S \ref{section:noetherian}, where we prove Theorems
\ref{maintheorem:vinoetherian}, \ref{maintheorem:artinianconj}, \ref{maintheorem:wvnoetherian},
\ref{maintheorem:vicnoetherian}, \ref{maintheorem:sinoetherian}, and
\ref{maintheorem:nonnoetherian}.  
We then prove our asymptotic structure result (Theorem \ref{maintheorem:asymptoticstructure})
in \S \ref{section:abstractnonsense}.
The brief \S \ref{section:twistedstability} shows how to use our locally Noetherian results
to prove our twisted homological stability theorems (Theorems \ref{maintheorem:gltwisted}
and \ref{maintheorem:sptwisted}).
We then give, in \S \ref{section:themachine}, a framework (adapted from the standard proof of homological stability
due to Quillen) for proving that $\tC$-modules arising from congruence subgroups
are finitely generated. Finally, \S \ref{section:finitegen} applies the above framework to prove our results that assert that modules arising from the homology groups of various congruence subgroups are finitely generated, namely Theorems \ref{maintheorem:glfinite},
\ref{maintheorem:slfinite}, \ref{maintheorem:autfinite},
\ref{maintheorem:spfinite}, and \ref{maintheorem:modfinite}.

\paragraph{Acknowledgments.} 
We wish to thank Tom Church, Yves Cornulier, Aur\'elien Djament, Jordan Ellenberg, Nicholas Kuhn, Amritanshu Prasad, Andrew Snowden, and Nathalie Wahl for helpful conversations and correspondence. 
We also thank some anonymous referees for their careful reading of the paper and numerous suggestions and improvements.

\section{Noetherian results}
\label{section:noetherian}

The goal of this section is to prove Theorems~\ref{maintheorem:vinoetherian}, \ref{maintheorem:artinianconj}, \ref{maintheorem:wvnoetherian},
\ref{maintheorem:vicnoetherian}, \ref{maintheorem:sinoetherian}, and
\ref{maintheorem:nonnoetherian}, which assert that modules over the various categories discussed in the introduction
are locally Noetherian.  We begin in \S \ref{section:noetherianprelim} with some preliminary results.  Next,
in \S \ref{section:noetherianovc} we introduce the category $\OVIC(R)$ and in \S \ref{section:noetherianovcnoetherian}
we prove that the category of $\OVIC(R)$-modules
is locally Noetherian.  We will use this in \S \ref{section:noetherianvect} to prove Theorems~\ref{maintheorem:vinoetherian}, \ref{maintheorem:artinianconj}, \ref{maintheorem:wvnoetherian},
and \ref{maintheorem:vicnoetherian}.  These cover all of our categories except $\SI(R)$, which
we deal with in \S \ref{section:noetheriansi}.
We close in \S \ref{section:noetheriannegative}
by proving Theorem \ref{maintheorem:nonnoetherian}, which asserts that our categories are {\bf not} locally Noetherian when $R$ is an infinite ring.

Since it presents no additional difficulties, in this section we allow $\bk$ to be an arbitrary ring (unital, but not necessarily commutative). 
By a $\bk$-module, we mean a left $\bk$-module, and Noetherian means left Noetherian. However, we emphasize that $R$ will always be a (unital) commutative ring.
Also, in this section we will denote the effects of functors by parentheses rather than subscripts, so for instance if $M\colon \tC \to \tD$ is a functor and $x \in \tC$, then
we will write $M(x) \in \tD$ for the image of $x$ under $M$.

\subsection{Preliminary results}
\label{section:noetherianprelim}

We first discuss some preliminary results.  Let $\bk$ be a fixed Noetherian ring.

\paragraph{Functor categories.}
Let $\tC$ be a category.  A $\tC$-module (over $\bk$) is a functor $\tC \to \Mod_\bk$ where 
$\Mod_\bk$ is the category of left $\bk$-modules.  A morphism of $\tC$-modules is a natural transformation. 
For $x \in \tC$, let $P_{\tC,x}$ denote the representable $\tC$-module generated at $x$, i.e., $P_{\tC,x}(y) = \bk[\Hom_\tC(x,y)]$. 
If $M$ is a $\tC$-module, then a morphism $P_{\tC,x} \to M$ is equivalent to a choice of element in $M(x)$. The category of $\tC$-modules is an Abelian category, with kernels, cokernels, exact sequences, etc.\ calculated
pointwise.  A $\tC$-module is {\bf finitely generated} if it is a quotient of a finite direct sum of $P_{\tC,x}$ (allowing different choices of $x$).  This agrees with the definition given in the introduction.

\begin{lemma}
\label{lemma:noetheriancrit}
Let $\tC$ be a category.  The category of $\tC$-modules is locally Noetherian if and only if for all $x \in \tC$, every
submodule of $P_{\tC,x}$ is finitely generated.
\end{lemma}
\begin{proof}
Immediate since Noetherianity is preserved by quotients and finite direct sums.
\end{proof}

\paragraph{Functors between functor categories.}
Let $\Phi \colon \tC \to \tD$ a functor.  Given a $\tD$-module $M \colon \tD \to \Mod_\bk$, we get a $\tC$-module 
$\Phi^{\ast}(M) = M \circ \Phi$.  Note that $\Phi^{\ast}$ is an exact functor from $\tD$-modules to $\tC$-modules.  We 
say that $\Phi$ is {\bf finite} if $\Phi^{\ast}(P_{\tD,y})$ is finitely generated for all $y \in \tD$.
Recall that $\Phi$ is 
{\bf essentially surjective} if every object of $\tD$ is isomorphic to an object of the form $\Phi(x)$ for $x \in \tC$.

\begin{lemma} 
\label{lem:finite-functor}
Let $\Phi \colon \tC \to \tD$ be an essentially surjective and finite functor.  Assume that the category of $\tC$-modules
is locally Noetherian.  Then the category of $\tD$-modules is locally Noetherian.
\end{lemma}

\begin{proof}
Pick $y \in \tD$ and let $M_1 \subseteq M_2 \subseteq \cdots$ be a chain of submodules of $P_{\tD,y}$.  By Lemma
\ref{lemma:noetheriancrit}, it is enough to show that the $M_n$ stabilize.  Since $\Phi^{\ast}$ is exact, we get a 
chain of submodules $\Phi^{\ast}(M_1) \subseteq \Phi^{\ast}(M_2) \subseteq \cdots$ of $\Phi^{\ast}(P_{\tD,y})$.  
Since $\Phi$ is finite, $\Phi^{\ast}(P_{\tD,y})$ is finitely generated, so since $\tC$ is locally Noetherian there is some $N$ 
such that $\Phi^{\ast}(M_n) = \Phi^{\ast}(M_{n+1})$ for $n \ge N$.  Since $\Phi$ is essentially surjective, it preserves strict inclusions of submodules: if $M_n \subsetneqq M_{n+1}$, then there is some $z \in \tD$ such 
that $M_n(z) \subsetneqq M_{n+1}(z)$ and we can find $x \in \tC$ such that $\Phi(x) \cong z$, which means 
that $\Phi^{\ast}(M_n)(x) \subsetneqq \Phi^{\ast}(M_{n+1})(x)$.  Thus $M_n = M_{n+1}$ for $n \geq N$,
as desired.
\end{proof}

\begin{remark}
In all of our applications of Lemma~\ref{lem:finite-functor} in this paper, essential surjectivity will be obvious, as the objects of our categories are indexed by natural numbers.
\end{remark}

\paragraph{Well partial orderings.}
A poset $\cP$ is {\bf well partially ordered} if for any sequence $x_1, x_2, \ldots$ in $\cP$, 
there exists $i < j$ such that $x_i \le x_j$.  See \cite{KruskalSurvey} for a survey. The direct product $\cP_1 \times \cP_2$ of posets is in a natural way a poset with $(x,y) \le (x',y')$ if and only if $x \le x'$ and $y \le y'$.  The following lemma is well-known, see for example \cite[\S 2]{tca-gb} for a proof.

\begin{lemma}\phantomsection \label{lem:poset-noeth}
\begin{compactenum}[\indent \rm (a)]
\item A subposet of a well partially ordered poset is well partially ordered.
\item If $\cP$ is well partially ordered and $x_1, x_2, \dots$ is a sequence, then there is an infinite increasing sequence $i_1 < i_2 < \cdots$ such that $x_{i_1} \le x_{i_2} \le \cdots$.
\item A finite direct product of well partially ordered posets is well partially ordered.
\end{compactenum}
\end{lemma}




\paragraph{Posets of words.}
Let $\Sigma$ be a finite set.  Let $\Sigma^\star$ be the set of words 
$s_1 \cdots s_n$ whose letters come from $\Sigma$.  
Define a poset structure on $\Sigma^{\star}$ by saying that $s_1 \cdots s_n \le s_1' \cdots s_m'$ if there is an increasing function $f \colon [n] \to [m]$ such that $s_i = s'_{f(i)}$ for all $i$. The following is a special case of Higman's lemma.

\begin{lemma}[Higman, \cite{HigmanLemma}] \label{lem:higman}
$\Sigma^\star$ is a well partially ordered poset.
\end{lemma}

\noindent
We will also need a variant of Higman's lemma.  Define a partially
ordered set $\widetilde{\Sigma}^\star$ whose objects are the same
as $\Sigma^\star$ by saying that
$s_1 \cdots s_n \le s'_1 \cdots s'_m$ if there is an increasing function 
$f \colon [n] \to [m]$ such that $s'_{f(i)}=s_i$ for $i =1,\dots,n$ and for 
each $j=1,\dots,m$, there exists $i$ such that $f(i) \le j$ and $s'_j = s'_{f(i)}$.

\begin{lemma} 
\label{lem:word-poset}
$\widetilde{\Sigma}^\star$ is a well partially ordered poset.
\end{lemma}

\noindent
See \cite[Prop. 8.2.1]{tca-gb}; a different proof of this is in \cite[Proof of Prop. 7.5]{draismakuttler}.

\subsection{The category \texorpdfstring{$\bOVIC$}{OVIC}}
\label{section:noetherianovc}

Fix a finite commutative ring $R$.  

\paragraph{Motivation.}
Let $\Unit \subset R^\times$ be a subgroup of the group of multiplicative units.
For the purpose of proving that the category of $\VIC(R,\Unit)$-modules is locally Noetherian,
we will use the ``cheap'' definition of $\VIC(R,\Unit)$ from the introduction and
use the alternate description of its morphisms from Remark \ref{remark:vicalternate}.
Thus $\VIC(R,\Unit)$ is the category whose objects are $\Set{$R^n$}{$n \geq 0$}$ and whose morphisms from $R^n$ to $R^{n'}$
are pairs $(f,f')$, where $f\colon R^n \rightarrow R^{n'}$ and
$f'\colon R^{n'} \rightarrow R^n$ satisfy $f' f = 1_{R^n}$ and $\det(f) \in \Unit$ if $n = n'$.

We will introduce yet another
category $\OVIC(R)$ 
that satisfies $\OVIC(R) \subset \VIC(R,\Unit)$ (the ``O'' stands for ``Ordered''; see Lemma \ref{lemma:ovickey} below).
We emphasize that $\OVIC(R)$ does {\bf not} depend on $\Unit$.  The objects of $\OVIC(R)$ are $\Set{$R^n$}{$n \geq 0$}$,
and the $\OVIC(R)$-morphisms from $R^n$ to $R^{n'}$ are $\VIC(R,\Unit)$-morphisms
$(f,f')$ such that $f'$ is what we will call a {\it column-adapted map} (see below for the definition).  
The key property of $\OVIC(R)$ is that every $\VIC(R)$-morphism 
$(f,f')\colon R^n \rightarrow R^{n'}$ (notice that we are using $\VIC(R)$ and not $\VIC(R,\Unit)$)
can be uniquely factored as
$(f,f') = (f_1,f_1')(f_2,f_2')$, where $(f_1,f_1')\colon R^n \rightarrow R^{n'}$ is
an $\OVIC(R)$-morphism and $(f_2,f_2')\colon R^{n'} \rightarrow R^{n'}$ is an isomorphism;
see Lemma \ref{lemma:ovicfactor} below.  This implies in particular that
the category $\OVIC(R)$ has no non-identity automorphisms.  We will prove
that the category of $\OVIC(R)$-modules is locally Noetherian and that the inclusion
functor $\OVIC(R) \hookrightarrow \VIC(R,\Unit)$ is finite.
The key point here is the above factorization and the fact that $\Aut_{\VIC(R)}(R^n)$ is finite.  
By Lemma \ref{lem:finite-functor},
this will imply that the category of $\VIC(R,\Unit)$-modules
is locally Noetherian.

\paragraph{Column-adapted maps, local case.}
In this paragraph, assume that $R$ is a commutative 
local ring.  Recall that the set of non-invertible elements of $R$
forms the unique maximal ideal of $R$.
An $R$-linear map $f\colon R^{n'} \rightarrow R^n$ is {\bf column-adapted} if there is a subset
$S = \{s_1 < \cdots < s_n\} \subseteq \{1, \dots, n'\}$ with the following two properties.  Regard $f$ as an $n \times n'$ matrix and
let $\bb_1,\ldots,\bb_n$ be the standard basis for $R^n$.
\begin{compactitem}
\item For $1 \leq i \leq n$, the $s_i^{\text{th}}$ column of $f$ is $\bb_i$.
\item For $1 \leq i \leq n$ and $1 \leq j < s_i$, the entry in position $(i,j)$ of $f$ is non-invertible.
\end{compactitem}
The set $S$ is unique and will be written $S_c(f)$. 

\begin{example}
An example of a column-adapted map $f\colon R^6 \rightarrow R^3$ is the linear map represented by the matrix
\[\begin{pmatrix}
\star & 1 & 0 & * & 0 & * \\
\star & 0 & 1 & * & 0 & * \\
\star & 0 & 0 & \star & 1 & *
\end{pmatrix}.
\]
Here $*$ can be any element of $R$, while $\star$ can be any non-invertible element of $R$.
\end{example}

The following observation might clarify the nature of column-adapted maps.  It follows immediately from our definitions.

\begin{observation}
All column-adapted maps are surjective, and the only
column-adapted map $f\colon R^n \rightarrow R^n$ is the identity.
\end{observation}

\paragraph{Column-adapted maps, general case.}
Now assume that $R$ is an arbitrary finite commutative ring.  This implies that
$R$ is Artinian, so there exists an isomorphism $R \cong R_1 \times \cdots \times R_q$ where the $R_i$ are finite commutative local rings \cite[Theorem 8.7]{AtiyahMacDonald}. For the remainder of this section, we fix one such isomorphism.  Observe that
\[
\Hom_R(R^{n'},R^n) = \Hom_{R_1}(R_1^{n'},R_1^n) \times \cdots \times \Hom_{R_q}(R_q^{n'},R_q^n),
\]
so given $f \in \Hom_R(R^{n'},R^n)$ we can write
$f = (f_1,\dots,f_q)$ with $f_i \in \Hom_{R_i}(R_i^{n'},R_i^n)$.
We will say that $f$ is {\bf column-adapted} if each $f_i$ is column-adapted.

\begin{lemma}
\label{lemma:adaptedclosed}
Let $R$ be a finite commutative ring and let $f\colon R^a \to R^b$ and $g \colon R^b \to R^c$ be column-adapted maps.  Then $gf\colon R^a \to R^c$ is column-adapted.
\end{lemma}
\begin{proof}
It is enough to deal with the case where $R$ is a local ring.
Write $S_c(f) = \{s_1 < \cdots < s_b\}$ and $S_c(g) = \{t_1 < \cdots < t_c\}$.  For $1 \leq i \leq c$, let $u_i = s_{t_i}$.  We
claim that $gf$ is a column-adapted map with $S_c(gf) = \{u_1 < \cdots < u_c\}$.  The first condition is clear, so we must
verify the second condition.  Consider $1 \leq i \leq n$ and $1 \leq j < u_i$.  We then have
\[(gf)_{ij} = g_{i1} f_{1j} + g_{i2} f_{2j} + \cdots + g_{ib} f_{bj}.\]
By definition, $g_{ik}$ is non-invertible for $1 \leq k < t_i$.  Also, for $t_i \leq k \leq b$ the element
$f_{kj}$ is non-invertible since $j < u_i = s_{t_i} \leq s_k$.  It follows that all the terms in our expression
for $(gf)_{ij}$ are non-invertible, so since the non-invertible elements of $R$ form an ideal we conclude that
$(gf)_{ij}$ is non-invertible, as desired.
\end{proof}

\begin{lemma}
\label{lemma:adaptedfactor}
Let $R$ be a finite commutative ring and let $f\colon R^{n'} \rightarrow R^n$ be a
surjection.  Then we can uniquely factor $f$ as $f = f_2 f_1$, where
$f_1\colon R^{n'} \rightarrow R^n$ is column-adapted and $f_2\colon R^n \rightarrow R^n$
is an isomorphism.
\end{lemma}

\begin{proof}
It is enough to deal with the case where $R$ is a local ring.  Since $R^n$ is
projective, we can find $g\colon R^n \rightarrow R^{n'}$ such that
$fg = 1_{R^n}$.  We can view $f$ and $g$ as matrices, and the
Cauchy--Binet formula says that
\[
1 = \det(fg) = \sum_I \det(f_{[n], I}) \det (g_{I, [n]})
\]
where $I$ ranges over all $n$-element subsets of $[n'] = \{1,\ldots,n'\}$ and
$f_{I,J}$ means the minor of $f$ with rows $I$ and columns $J$.
Since $R$ is local, the non-units form an ideal, so there exists at least one $I$ such that
$\det(f_{[n], I})$ is a unit.  Let $I$ be the lexicographically minimal subset such that this
holds.  There is a unique $h \in \GL(R^n)$ with
$(h f)_{[n],I} = 1_{R^n}$.  This implies that $h f$
is column-adapted with $S_c(h f) = I$, and the desired factorization is
$f_1 = hf$, $f_2 = h^{-1}$.
\end{proof}

\paragraph{The category $\bOVIC$.}
Continue to let $R$ be a general finite commutative ring.  Define
$\OVIC(R)$ to be the category whose objects are $\Set{$R^n$}{$n \geq 0$}$ and where $\Hom_{\OVIC(R)}(R^n,R^{n'})$
consists of pairs $(f,f')$ as follows.
\begin{compactitem}
\item $f \in \Hom_R(R^{n},R^{n'})$ and $f' \in \Hom_R(R^{n'},R^n)$.
\item $f' f = 1_{R^n}$, so $f$ is an injection and $f'$ is a surjection.
\item $f'$ is column-adapted.
\end{compactitem}
For $(f,f') \in \Hom_{\OVIC(R)}(R^n,R^{n'})$ and $(g,g') \in \Hom_{\OVIC(R)}(R^{n'},R^{n''})$, we define
$(g,g')(f,f')$ to be $(gf,f'g') \in \Hom_{\OVIC(R)}(R^n,R^{n''})$.  Lemma \ref{lemma:adaptedclosed}
implies that this makes sense.  The key property of $\OVIC(R)$ is as follows.

\begin{lemma}
\label{lemma:ovicfactor}
Let $R$ be a finite commutative ring and let $(f,f') \in \Hom_{\VIC(R)}(R^n,R^{n'})$.
Then we can uniquely write $(f,f') = (f_1,f_1')(f_2,f_2')$, where
$(f_2,f_2') \in \Aut_{\VIC(R)}(R^n)$ and $(f_1,f_1') \in \Hom_{\OVIC(R)}(R^n,R^{n'})$.
\end{lemma}
\begin{proof}
Letting $f' = f_2' f_1'$ be the factorization from Lemma \ref{lemma:adaptedfactor}, the
desired factorization is $(f,f') = (f f_2',f_1')((f_2')^{-1},f_2')$.  To see that
this first factor is in $\Hom_{\OVIC(R)}(R^n,R^{n'})$, observe that
$f_1' f f_2' = (f_2')^{-1} f' f f_2' = (f_2')^{-1} f_2' = 1_{R^n}$.
\end{proof}
 
\paragraph{Free and dependent rows.}
Assume that $R$ is local.
The condition $f'f = 1_{R^n}$ in the definition of an $\OVIC(R)$-morphism
implies that the rows of $f$ indexed by $S_c(f')$ are determined
by the other rows, which can be chosen freely.  So we will call the rows of $f$ indexed by $S_c(f')$ the {\bf dependent rows}
and the other rows the {\bf free rows}.

\begin{example} 
Assume that $R$ is a local ring, that $f'\colon R^6 \rightarrow R^3$ is a column-adapted map with
$S_c(f') = \{2,3,5\}$, and that $f\colon R^3 \rightarrow R^6$ satisfies
$f' f = 1_{R^3}$.  The possible matrices representing $f'$ and $f$ are of the form
\[
f' = \begin{pmatrix}
\star & 1 & 0 & * & 0 & * \\
\star & 0 & 1 & * & 0 & * \\
\star & 0 & 0 & \star & 1 & *
\end{pmatrix}
\quad \text{and} \quad
f = \begin{pmatrix}
* & * & * \\
\diamond & \diamond & \diamond \\
\diamond & \diamond & \diamond \\
* & * & * \\
\diamond & \diamond & \diamond \\
* & * & * 
\end{pmatrix}.
\]
Here $*$ can be any element of $R$, while $\star$ can be any non-invertible element of $R$
and each $\diamond$ is completely determined by the choices of $*$ and $\star$ plus 
the fact that $f' f = 1_{R^3}$. 
\end{example}

\subsection{The category of \texorpdfstring{$\bOVIC$}{OVIC}-modules is locally Noetherian}
\label{section:noetherianovcnoetherian}

The main result of this section is Theorem~\ref{thm:ovicnoetherian} below, which says
that the category of $\OVIC(R)$-modules is locally Noetherian when $R$ is a finite commutative ring.

\paragraph{A partial order.}
We first need the existence of a certain partial ordering.  For $d \geq 0$, define 
\[
\cP_R(d) = \bigsqcup_{n=0}^{\infty} \Hom_{\OVIC(R)}(R^d, R^n).
\]
To apply the results of Sam--Snowden in \cite{tca-gb}, we would need to find a total ordering of the set
$\cP_R(d)$ which is invariant under left composition.  However, constructing such an ordering seems difficult.
Our solution is to construct a total ordering which is invariant under left composition with ``enough'' morphisms
to make an argument in the style of \cite{tca-gb} go through (see Remark~\ref{rmk:no-term-order} for some discussion 
of this).  The following lemma gives this ordering and is the key to the proof of
Theorem~\ref{thm:ovicnoetherian}, which asserts that the category of $\OVIC(R)$-modules is locally
Noetherian when $R$ is a finite commutative ring.

\begin{lemma}
\label{lemma:ovickey}
Let $R$ be a finite commutative ring, and fix $d \geq 0$.  There exists a well partial ordering $\preceq$ on $\cP_R(d)$ together with an extension $\leq$ of $\preceq$ to a total ordering such that the following holds.
Consider $(f,f') \in \Hom_{\OVIC(R)}(R^d, R^n)$ and $(g,g') \in \Hom_{\OVIC(R)}(R^d, R^{n'})$ with 
$(f,f') \preceq (g,g')$.  Then there is some $(\phi,\phi') \in \Hom_{\OVIC(R)}(R^n,R^{n'})$ with the following
two properties.
\begin{compactenum}[\indent \rm 1.]
\item We have $(g,g') = (\phi,\phi')(f,f')$.  
\item If $(f_1,f'_1) \in \Hom_{\OVIC(R)}(R^d, R^{n''})$ satisfies $(f_1,f'_1) < (f,f')$, then 
$(\phi,\phi')(f_1,f'_1) < (g,g')$.
\end{compactenum}
\end{lemma}

Since the proof of Lemma~\ref{lemma:ovickey} is lengthy, we postpone it until the end of this section and instead
show how it can be used to prove local Noetherianity.

\paragraph{Noetherianity.} 
We now come to our main result.

\begin{theorem} \label{thm:ovicnoetherian}
Let $R$ be a finite commutative ring.  Then the category of $\OVIC(R)$-modules is locally Noetherian.
\end{theorem}

\begin{proof}
Fix a Noetherian ring $\bk$.
For $d \geq 0$, define $P_d$ to be the $\OVIC(R)$-module $P_{\OVIC(R), R^d}$, where the notation is as in \S\ref{section:noetherianprelim}, so $P_d(R^n) = \bk[\Hom_{\OVIC(R)}(R^d,R^n)]$.
By Lemma~\ref{lemma:noetheriancrit}, it is enough to show that every submodule of $P_d$ is finitely generated.  

Let $\preceq$ and $\leq$ be the orderings on $\cP_R(d)$ given by Lemma \ref{lemma:ovickey}.
For an element $(f,f') \in \Hom_{\OVIC(R)}(R^d,R^n)$, denote by $e_{f,f'}$ the associated basis element of $P_d(R^n)$.
Given nonzero $x \in P_d(R^n)$, define its {\bf initial term} $\init(x)$ as follows. First, set $\init(0)=0$. For $x\ne 0$, write $x = \sum \alpha_{f,f'} e_{f,f'}$ with $\alpha_{f,f'} \in \bk$, and let $(f_0,f'_0)$ be the $\leq$-largest element such that $\alpha_{f_0,f'_0} \neq 0$. We then define $\init(x) = \alpha_{f_0,f'_0} e_{f_0,f'_0}$. Next, given a submodule $M \subset P_d$, define its {\bf initial module} $\init(M)$ to be the function that takes $n \geq 0$ to the $\bk$-module $\bk \Set{$\init(x)$}{$x \in M(R^n)$}$.  Warning: $\init(M)$ need not be an $\OVIC(R)$-submodule of $P_d$, see Remark~\ref{rmk:no-term-order}.

\begin{claim}
If $N,M \subset P_d$ are submodules with $N \subset M$ and $\init(N) = \init(M)$, then $N=M$.
\end{claim}

\begin{proof}[Proof of Claim]
Assume that $M \ne N$.  Let $n \geq 0$ be such that $M(R^n) \neq N(R^n)$.  Pick $y \in M(R^n) \setminus N(R^n)$ such that $\init(y) = \alpha e_t$ where $t$ is $\leq$-minimal.  By assumption, there exists $z \in N(R^n)$ with $\init(z) = \init(y)$.  But then
$y-z \in M(R^n) \setminus N(R^n)$ and $\init(y-z) = \beta e_{t'}$ where $t' < t$, a contradiction.
\end{proof}

We now turn to the proof that every submodule of $P_d$ is finitely generated.  Assume otherwise, so there is a strictly increasing sequence $M_0 \subsetneqq M_1 \subsetneqq \cdots$ of submodules. By the claim, we must have $\init(M_i) \neq \init(M_{i-1})$ for all $i \geq 1$, so we can find some $n_i \geq 0$ and some $\lambda_i e_{f_i,f_i'} \in \init(M_i)(R^{n_i})$ with $\lambda_i e_{f_i,f_i'} \notin \init(M_{i-1})(R^{n_i})$.  Since
$\preceq$ is a well partial ordering, we can apply Lemma \ref{lem:poset-noeth}(b) and find $i_0 < i_1 < \cdots$ such that
\[
(f_{i_0},f_{i_0}') \preceq (f_{i_1},f_{i_1}') \preceq (f_{i_2},f_{i_2}') \preceq \cdots
\]
Since $\bk$ is Noetherian, there exists some $m \geq 0$ such that 
$\lambda_{i_{m+1}}$ is in the left $\bk$-ideal generated by $\lambda_{i_0},\ldots,\lambda_{i_m}$, i.e., we can write $\lambda_{i_{m+1}} = \sum_{j=0}^m c_j \lambda_{i_j}$ with $c_j \in \bk$.
Consider some $0 \leq j \leq m$.  Let $x_j \in M_{i_j}(R^{n_{i_j}})$ be such that 
$\init(x_j) = \lambda_{i_j} e_{f_{i_j},f_{i_j}'}$. By Lemma \ref{lemma:ovickey}, there exists some
$(\phi_j,\phi'_j) \in \Hom_{\OVIC(R)}(R^{n_{i_j}},R^{n_{i_{m+1}}})$ such that 
$(f_{i_{m+1}},f'_{i_{m+1}}) = (\phi_j,\phi'_j)(f_{i_j},f'_{i_j})$.  Setting 
$y = \sum_{j=0}^m c_j (\phi_j,\phi'_j) x_j$, an element of $M_{i_{m+1} - 1}(R^{n_{i_{m+1}}})$,
Lemma~\ref{lemma:ovickey} implies that $\init(y) = \lambda_{i_{m+1}} e_{f_{i_{m+1}},f'_{i_{m+1}}}$, a contradiction.
\end{proof}

\paragraph{Insertion maps.}
For the proof of Lemma \ref{lemma:ovickey}, we will need the following lemma.

\begin{lemma}
\label{lemma:insertion}
Let $R$ be a commutative local ring.  Fix some $1 \leq k \leq \ell \leq n$ and
some $S = \{s_1 < s_2 < \cdots < s_d\} \subset \{1,\ldots,n\}$.  Set 
$\widehat{\ell} = \Max(\Set{$i$}{$s_i<\ell$} \cup \{0\})$.  Let
$v = (v_1,\ldots,v_d) \in R^d$ be such that $v_i$ 
is non-invertible for all $i > \widehat{\ell}$. 
Define 
\[
T = \{s_1 < \cdots < s_{\widehat{\ell}} < s_{\widehat{\ell}}+1 < \cdots < s_d+1\} \subset \{1,\ldots,n+1\}.
\]
There exists $(\phi,\phi') \in \Hom_{\OVIC(R)}(R^{n},R^{n+1})$ with the following properties.  

\begin{compactitem}
\item Consider $(f,f') \in \Hom_{\OVIC(R)}(R^d,R^n)$ with $S_c(f') = S$.
\begin{compactitem}
\item The matrix $f'\phi'$ is obtained from $f'$ by 
inserting the column vector $v$ between the $(\ell-1)^{\text{st}}$ and $\ell^{\text{th}}$ columns.  Observe that $S_c(f'\phi') = T$.
\item Let $w \in R^d$ be the $k^{\text{th}}$
row of $f$.  Then $\phi f$ is obtained from $f$ by inserting $w$
between the $(\ell-1)^{\text{st}}$ and $\ell^{\text{th}}$ rows (observe that this
new row is a free row!) and then modifying the dependent rows so that
$(\phi f,f' \phi') \in \Hom_{\OVIC(R)}(R^{d},R^{n+1})$. 
More precisely, the $\widehat{s}_i^{\text{th}}$ row of $\phi f$ is the $s_i^{\text{th}}$ row of $f$ minus $v_i w$.
\end{compactitem}
\item Consider $(g,g') \in \Hom_{\OVIC(R)}(R^d,R^n)$ with $S_c(g') < S$ in the lexicographic order.  Then $S_c(g' \phi') < T$ in the lexicographic order.
\end{compactitem}
\end{lemma}
\begin{proof}
Define $\widehat{v} \in R^n$ to be the column vector with $v_i$ in position $s_i$
for $1 \leq i \leq d$ and zeros elsewhere.  Next,
define $\phi' \colon R^{n+1} \rightarrow R^n$ to be the matrix obtained from the 
$n \times n$ identity matrix ${\rm Id}_{R^n}$ by inserting $\widehat{v}$ between the $(\ell-1)^{\text{st}}$ and $\ell^{\text{th}}$ columns.  
The condition on the $v_i$ ensures that $\phi'$ is a column-adapted map. 
The possible choices of $\phi\colon R^n \rightarrow R^{n+1}$ such that 
$(\phi,\phi') \in \Hom_{\OVIC(R)}(R^{n},R^{n+1})$
have a single free row, namely the $\ell^{\text{th}}$.  Let that row have a 
$1$ in position $k$ and zeros elsewhere.  The other dependent rows are determined by this.  For a
formula for this matrix, letting $e$ be the $k^{\text{th}}$ standard basis row vector, the matrix $\phi$ is 
the result of inserting $e$ between the $\ell^{\text{th}}$ and $(\ell+1)^{\text{st}}$ rows of ${\rm Id}_{R^n} - \widehat{v} e$.
We have thus constructed an element $(\phi,\phi') \in \Hom_{\OVIC(R)}(R^{n},R^{n+1})$; that it satisfies the 
conclusions of the lemma is an easy calculation.
\end{proof}

\paragraph{Constructing the orders, local case.}
We now prove Lemma \ref{lemma:ovickey} in the special case where $R$ is a local ring.

\begin{proof}[Proof of Lemma \ref{lemma:ovickey}, local case]
Let $R$ be a finite commutative local ring.  The first step is to construct $\preceq$.  
Consider
$(f,f') \in \Hom_{\OVIC(R)}(R^d,R^n)$ and $(g,g') \in \Hom_{\OVIC(R)}(R^d,R^{n'})$.  We will say that
$(f,f') \preceq (g,g')$ if there exists a sequence
\[(f,f') = (f_0,f_0'), (f_1,f_1'),\ldots,(f_{n'-n},f_{n'-n}') = (g,g')\]
with the following properties:
\begin{compactitem}
\item For $0 \leq i \leq n'-n$, we have $(f_i,f_i') \in \Hom_{\OVIC(R)}(R^d,R^{n+i})$.
\item For $0 \leq i < n'-n$, the morphism $(f_{i+1},f_{i+1}')$ is obtained from 
$(f_{i},f_{i}')$ as follows.
Choose some $1 \leq k \leq \ell \leq n+i$ such that $k \notin S_c(f_i')$.  Then $f_{i+1}'$ is obtained from $f_i'$ by inserting
a copy of the $k^{\text{th}}$ column of $f_i'$ between the $(\ell-1)^{\text{st}}$ and $\ell^{\text{th}}$ columns of $f_i'$.
Also, $f_{i+1}$ is obtained from $f_i'$ by first inserting a copy of the $k^{\text{th}}$ row of $f_i$ (observe that this is a free row!) between the $(\ell-1)^{\text{st}}$ and $\ell^{\text{th}}$ rows of $f_i$ (this new row is 
also a free row!) and then modifying
the dependent rows to ensure that $(f_{i+1},f_{i+1}') \in \Hom_{\OVIC(R)}(R^d,R^{n+i+1})$.
\end{compactitem}
This clearly defines a partial order on $\cP_R(d)$.

\begin{claim}
The partially ordered set $(\cP_R(d), \preceq)$ is well partially ordered.
\end{claim}

\begin{proof}[{Proof of Claim}]
Define 
\[
\Sigma = \left(R^{d} \amalg \{\clubsuit\}\right) \times \left(R^{d} \amalg \{\spadesuit\}\right),
\]
where $\clubsuit$ and $\spadesuit$ are formal symbols.  We will
prove that $(\cP_R(d),\preceq)$ is isomorphic to a subposet of the well partially ordered poset $\widetilde{\Sigma}^\star$ from Lemma~\ref{lem:word-poset}, so by Lemma~\ref{lem:poset-noeth}(a) is itself well partially ordered.  
Consider $(f,f') \in \Hom_{\OVIC(R)}(R^d, R^n)$.  For $1 \leq i \leq n$, let $(r_i,c_i) \in \Sigma$
be as follows.
\begin{compactitem}
\item If $i \in S_c(f')$, then $(r_i,c_i) = (\clubsuit,\spadesuit)$.
Observe that in this case the $i^{\text{th}}$ row of $f$ is a dependent row.
\item If $i \notin S_c(f')$, then $r_i$ is the $i^{\text{th}}$ row of $f$ and
$c_i$ is the $i^{\text{th}}$ column of $f'$.  Observe that in this case the $i^{\text{th}}$ row of $f$ is a free row.
\end{compactitem}
We get an element $(r_1,c_1)\cdots(r_n,c_n) \in \widetilde{\Sigma}^\star$.  We claim that the map taking
the pair $(f,f')$ to $(r_1,c_1)\cdots(r_n,c_n)$ is an injection from 
$\cP_R(d)$ to $\widetilde{\Sigma}^{\star}$.
To see this, consider an element $(r_1,c_1)\cdots(r_n,c_n) \in \widetilde{\Sigma}^\star$ in its image.  Our
goal is to reconstruct $(f,f')$.  First, we can read off all the free rows of $f$.  Second,
if the $i^{\text{th}}$ row of $f$ is free, then we can also read off the $i^{\text{th}}$ column of $f'$.  The
other columns of $f'$ are determined by the fact that they are precisely the standard basis vectors
in $R^d$, appearing in their natural order.  But now that we have all of $f'$ and the free rows of $f$, the dependent rows of $f$ are determined since $f'f = 1_{R^d}$.  The claim follows.  That
this injection is order-preserving is immediate from the definitions.  The only subtle point is the reason
we required that $k \notin S_c(f_i')$ in the second case of the definition of the ordering on
$\cP_R(d)$.  This is needed to ensure that no dependent rows of $f$ are copied when producing a larger element.
\end{proof}

We now extend $\preceq$ to a total order $\leq$.  Fix some arbitrary total order on $R^d$, and
consider $(f,f') \in \Hom_{\OVIC(R)}(R^d,R^n)$ and $(g,g') \in \Hom_{\OVIC(R)}(R^d,R^{n'})$.  We then determine if $(f,f') \leq (g,g')$ via the following
procedure.
\begin{compactitem}
\item If $n < n'$, then $(f,f')<(g,g')$.
\item Otherwise, $n=n'$.  If $S_c(f') < S_c(g')$ in lexicographic ordering, set $(f,f')<(g,g')$.
\item Otherwise, $S_c(f') = S_c(g')$.  If $f' \neq g'$, then compare the sequences
of elements of $R^d$ which form the columns of $f'$ and $g'$ using the lexicographic ordering
and the fixed total ordering on $R^d$.
\item Finally, if $f'=g'$, then compare the sequences of elements of $R^d$ which form
the free rows of $f$ and $g$ using the lexicographic ordering and the fixed total
ordering on $R^d$.
\end{compactitem}
It is clear that this determines a total order $\leq$ on $\cP_R(d)$ that extends $\preceq$.

It remains to check that these orders have the property claimed by the lemma.
Consider $(f,f') \in \Hom_{\OVIC(R)}(R^d, R^n)$ and $(g,g') \in \Hom_{\OVIC(R)}(R^d, R^{n'})$ with
$(f,f') \preceq (g,g')$.  By repeated applications of
Lemma \ref{lemma:insertion}, there
exists $(\phi,\phi') \in \Hom_{\OVIC(R)}(R^n,R^{n'})$ such that $(g,g') = (\phi,\phi')(f,f')$.
Now pick $(f_1,f'_1) \in \Hom_{\OVIC(R)}(R^d, R^n)$ such that $(f_1,f'_1) < (f,f')$.
We want to show that $(\phi,\phi')(f_1,f'_1) < (\phi,\phi')(f,f')$.  
If $S_c(f'_1) < S_c(f')$ in the lexicographic order, then this is an immediate consequence
of Lemma \ref{lemma:insertion}, so we can assume that $S_c(f_1') = S_c(f')$.
In this case, it
follows from Lemma \ref{lemma:insertion} that $f_1' \phi'$ and $f' \phi'$ 
are obtained by inserting
the same columns into $f_1'$ and $f'$, respectively.  If $f_1' \neq f'$, then it follows
immediately that $(\phi,\phi')(f_1,f'_1) < (\phi,\phi')(f,f')$.  Otherwise, it follows
from Lemma \ref{lemma:insertion} that $\phi f_1$ and $\phi f$ are obtained by duplicating
the same rows in $f_1$ and $f$, respectively, and it follows that 
$(\phi,\phi')(f_1,f'_1) < (\phi,\phi')(f,f')$, as desired.
\end{proof}

\begin{remark} \label{rmk:no-term-order}
The total ordering that we define for the proof of Lemma~\ref{lemma:ovickey} need not be compatible with all compositions in the sense that $f' < f$ implies that $gf' < gf$ for all $g$. This assumption is in place in \cite{tca-gb} which is why we cannot use the machinery from that paper.  To see this, consider the example $R = \Z/16$ and define
\begin{align*}
f &= \begin{pmatrix} 0 \\ 1 \end{pmatrix}, &f'_1 &= \begin{pmatrix} 2 & 1 \end{pmatrix}, &f'_2 &= \begin{pmatrix} 6 & 1 \end{pmatrix},\\
g &= \begin{pmatrix} 0 & 0 \\ 1 & 0 \\ 0 & 1 \end{pmatrix},  &g'_1 &= \begin{pmatrix} 2 & 1 & 0 \\ 0 & 0 & 1 \end{pmatrix}, &g'_2 &= \begin{pmatrix} 6 & 1 & 0\\ 0 & 0 & 1 \end{pmatrix}.
\end{align*}
Then $(f,f'_i) \in \Hom_{\OVIC(R)}(R^1,R^2)$ and $(g,g'_i) \in \Hom_{\OVIC(R)}(R^2,R^3)$ for $i=1,2$. Also, 
\begin{align*}
g'_1f'_1 &= \begin{pmatrix} 4 & 2 & 1 \end{pmatrix},  &g'_1 f'_2 &= \begin{pmatrix} 12 & 6 & 1 \end{pmatrix},\\
g'_2f'_1 &= \begin{pmatrix} 12 & 2 & 1 \end{pmatrix}, &g'_2f'_2  &= \begin{pmatrix} 4 & 6 & 1 \end{pmatrix}.
\end{align*}
Using our ordering, we need to fix a total ordering on $\Z/16$. If we choose $2<6$, then we would have $f'_1 < f'_2$. This forces $4<12$ by considering $g'_1f'_1 < g'_1 f'_2$, but it also forces $12<4$ by considering $g'_2 f'_1 < g'_2 f'_2$. Hence our ordering is not compatible with all compositions.

In particular, we cannot guarantee that $\init(M)$ is an $\OVIC(R)$-submodule of $P_d$. It is possible that there exist total orderings that are compatible with all compositions.
\end{remark}

\paragraph{Constructing the orders, general case.}
We finally come to the general case.

\begin{proof}[Proof of Lemma \ref{lemma:ovickey}, general case]
Let $R$ be an arbitrary finite commutative ring.  As we mentioned above, $R$ is Artinian, so we can
write $R \cong R_1 \times \cdots \times R_q$ with $R_i$ a finite local ring \cite[Theorem 8.7]{AtiyahMacDonald}.  Indeed,
we have fixed one such isomorphism.
For all $n,n' \geq 0$,
we then have
\[\Hom_R(R^{n},R^{n'}) = \Hom_{R_1}(R_1^{n},R_1^{n'}) \times \cdots \times \Hom_{R_q}(R_q^{n},R_q^{n'}).\]
The set $\cP_R(d)$ can thus be identified with the set of tuples
\[((f_1,f_1'),\ldots,(f_q,f_q')) \in \cP_{R_1}(d) \times \cdots \times \cP_{R_q}(d)\]
such that there exists some single $n \geq 0$ with $(f_i,f_i') \in \Hom_{\OVIC(R_i)}(R_i^d,R_i^n)$ for
all $1 \leq i \leq q$.  Above we constructed the desired orders $\preceq$ and $\leq$
on each $\cP_{R_i}(d)$.  Define a partial order $\preceq$ on their product using the product partial order
and define a total order $\leq$ on their product using the lexicographic order.  This restricts to
orders $\preceq$ and $\leq$ on $\cP_R(d)$.  Lemma \ref{lem:poset-noeth} implies that $\preceq$ is well partially ordered,
and the remaining conclusions are immediate.
\end{proof}

\subsection{The categories of \texorpdfstring{$\bVI$-, $\btV$-, $\bwV$-, and $\bVIC$-}{VI-, V-, V'-, and VIC-}modules
are locally Noetherian}
\label{section:noetherianvect}

We now prove several of our main theorems.

\begin{proof}[{Proof of Theorem~\ref{maintheorem:vicnoetherian}}]
We wish to prove that the category of $\VIC(R,\Unit)$-modules is locally Noetherian.
By the discussion at the beginning of \S \ref{section:noetherianovc} (and Theorem \ref{thm:ovicnoetherian}), it
is enough to show that the inclusion functor
$\Phi\colon \OVIC(R) \rightarrow \VIC(R,\Unit)$ is finite.  Fix $d \geq 0$
and set $M = P_{\VIC(R,\Unit),R^d}$, so $M(R^n) = \bk[\Hom_{\VIC(R,\Unit)}(R^d,R^n)]$.
Our goal is to prove that the $\OVIC(R)$-module $\Phi^{\ast}(M)$ is finitely generated.
For every $\varphi \in \SL_d^{\Unit}(R)$,
we get an element $(\varphi,\varphi^{-1})$ of $\Phi^{\ast}(M)(R^d)$, and thus a morphism
$P_{\OVIC(R),R^d} \rightarrow \Phi^{\ast}(M)$.  For
$\varphi \in \GL_d(R)$ and $(f,f') \in \Hom_{\OVIC(R)}(R^d,R^{d+1})$, the element $(f,f')(\varphi,\varphi^{-1})$ is in $\Phi^{\ast}(M)(R^{d+1})$, and thus determines a morphism
$P_{\OVIC(R),R^{d+1}} \rightarrow \Phi^{\ast}(M)$.  Lemma \ref{lemma:ovicfactor} implies that the resulting map
\[
\left(\bigoplus_{\SL^{\Unit}_d(R)} P_{\OVIC(R),R^d}\right) \oplus
\left(\bigoplus_{\GL_d(R) \times \Hom_{\OVIC(R)}(R^d,R^{d+1})} P_{\OVIC(R),R^{d+1}}\right)
\longrightarrow \Phi^{\ast}(M)
\]
is a surjection.  Since $R$ is finite, this proves that 
$\Phi^{\ast}(M)$ is finitely generated.
\end{proof}

\begin{proof}[{Proof of Theorem~\ref{maintheorem:vinoetherian}}]
We wish to prove that the category of $\VI(R)$-modules is locally Noetherian.  
The forgetful functor $\VIC(R) \to \VI(R)$ is essentially surjective since it is the identity map on the objects. To see that it is finite, consider the pullback of $P_{\VI(R), R^n}$ to $\VIC(R)$. Since it is generated as a $\VI(R)$-module by the identity map ${\rm id}_{R^n} \in P_{\VI(R), R^n}(R^n) = \bk[\hom_{\VI(R)}(R^n, R^n)]$, the same is true as a $\VIC(R)$-module since the forgetful map is surjective on morphism sets.  This surjectivity uses the fact that $R$ is finite.  Indeed, this
implies that every stably free $R$-module is free, so the cokernel of a splittable $R$-linear map $R^m \rightarrow R^n$ is free.  The desired statement now follows from Lemma~\ref{lem:finite-functor} and Theorem~\ref{maintheorem:vicnoetherian}.
\end{proof}

\begin{proof}[{Proof of Theorem~\ref{maintheorem:artinianconj}}]
We wish to prove that the category of $\tV(R)$-modules is locally Noetherian.  By Lemma~\ref{lem:finite-functor} and
Theorem~\ref{maintheorem:vinoetherian}, it is enough to prove that the inclusion functor
$\iota \colon \VI(R) \to \tV(R)$ is essentially surjective and finite.  Essential surjectivity is clear, so we must only prove finiteness.
Let $V$ be a free $R$-module.  By considering its image, every splittable $R$-linear map canonically factors as a composition of a surjective map followed by a splittable injective map. So we have the decomposition of $\VI(R)$-modules
\[
\iota^*(P_{\tV(R), V}) = \bigoplus_{V \to W \to 0} P_{\VI(R), W}
\]
where the sum is over all free quotients $W$ of $V$.
\end{proof}

\begin{proof}[{Proof of Theorem~\ref{maintheorem:wvnoetherian}}]
We wish to prove that the category of $\wV(R)$-modules is locally Noetherian.
By Lemma~\ref{lem:finite-functor} and Theorem~\ref{maintheorem:vinoetherian}, 
it is enough to prove that the inclusion functor $\iota \colon \VI(R) \to \wV(R)$ is 
essentially surjective and finite.  Essential surjectivity is clear, so we must only 
prove finiteness.  Consider some $d \geq 0$.  We must prove that
$\iota^{\ast}(P_{\wV(R),R^d})$ is a finitely generated $\VI(R)$-module.  
Let $N$ be the number of elements of $R^d$.  
Consider some $n \geq N$ and some $\wV(R)$-morphism $f\colon R^d \rightarrow R^n$.  
Viewing $f$ as an $n \times d$ matrix, there are at most $N$ distinct rows occurring
in $f$.  This implies that there
exists some $g \in \GL_d(R)$ such that the last $(n-N)$ rows of $f \circ g$ are 
identically $0$.  From this, we see that $f$ can be factored as
\[R^d \xrightarrow{f'} R^N \xrightarrow{f''} R^n,\]
where $f''$ is splittable and $f'$ is arbitrary.  This implies that
we have the decomposition of $\VI(R)$-modules
\[
\iota^*(P_{\wV(R), R^d}) = \bigoplus_{R^d \to R^{N'}} P_{\VI(R), R^{N'}},
\]
where the sum is over all $R$-linear maps $R^d \rightarrow R^{N'}$ with $N' \leq N$.
\end{proof}

\begin{remark} \label{rmk:schwartz-alt-proof}
A shorter proof of Theorem~\ref{maintheorem:vinoetherian} (and hence of Theorems~\ref{maintheorem:artinianconj} and~\ref{maintheorem:wvnoetherian}) is given in \cite[\S 8.3]{tca-gb} that deduces it from the locally Noetherian property of the opposite of the category of finite sets and surjective functions. We include the proof above since it follows easily from Theorem~\ref{maintheorem:vicnoetherian} which is the result that we actually need for our applications, and because the ideas for the proof were only made possible by work from both projects.
\end{remark}

\subsection{The category of \texorpdfstring{$\bSI$}{SI}-modules is locally Noetherian}
\label{section:noetheriansi}

Fix a finite commutative ring $R$.  Our goal in this section is to adapt the proof that the category
of $\VIC(R)$-modules is locally Noetherian to prove that the category of $\SI(R)$-modules is locally Noetherian.  The
proof of this is immediately after Theorem \ref{thm:osinoetherian} below.

\paragraph{Row-adapted maps.}
An $R$-linear map $f\colon R^{n} \rightarrow R^{m}$ is {\bf row-adapted} if the transpose $f^t$ of the matrix
representing $f$ is column-adapted in the sense of \S \ref{section:noetherianovc}, in which case
we define $S_r(f) = S_c(f^t)$.  Lemma \ref{lemma:adaptedclosed}
implies that the composition of two row-adapted maps is row-adapted.

\paragraph{The category $\bOSI$.}
Define $\OSI'(R)$ to be the category whose objects are pairs $(R^{2n},\omega)$, where $\omega$ is a symplectic
form on $R^{2n}$, and whose morphisms from $(R^{2n},\omega)$ to $(R^{2n'},\omega')$ are $R$-linear
maps $f\colon R^{2n} \rightarrow R^{2n'}$ which are symplectic and row-adapted.
Also, define $\OSI(R)$ to be the full 
subcategory of $\OSI'(R)$ spanned by the $(R^{2n},\omega)$ such that $\omega$
is the standard symplectic form on $R^{2n}$, i.e., if $\{\bb_1,\ldots,\bb_{2n}\}$ is the standard basis for $R^{2n}$, then
\[
\omega(\bb_{2i-1},\bb_{2j-1}) = \omega(\bb_{2i},\bb_{2j}) = 0 \quad \text{and} \quad \omega(\bb_{2i-1},\bb_{2j}) = \delta_{ij} \quad \quad (1 \leq i,j \leq n),
\]
where $\delta_{ij}$ is the Kronecker delta.  We will frequently omit the $\omega$ from objects of $\OSI(R)$.

\paragraph{Factorization of $\bSI$-morphisms.}
The following is the analogue for $\SI(R)$ of Lemma \ref{lemma:ovicfactor}.

\begin{lemma}
\label{lemma:osifactor}
Let $R$ be a finite commutative ring and $f \in \Hom_{\SI(R)}((R^{2n},\omega),(R^{2n'},\omega'))$.
Then we can uniquely write $f = f_1 f_2$, where $f_2 \in \Iso_{\SI(R)}((R^{2n},\omega),(R^{2n},\lambda))$
for some symplectic form $\lambda$ on $R^{2n}$ and $f_1 \in \Hom_{\OSI'(R)}((R^{2n},\lambda),(R^{2n'},\omega'))$.
\end{lemma}
\begin{proof}
Applying Lemma \ref{lemma:adaptedfactor} to the transpose of $f$ and then transposing the result, we can uniquely write
$f = f_1 f_2$, where $f_2\colon R^{2n} \rightarrow R^{2n}$ is an isomorphism and $f_1\colon R^{2n}\rightarrow R^{2n'}$
is a row-adapted map.  There exists a unique symplectic form $\lambda$ on $R^{2n}$
such that $f_2\colon (R^{2n},\omega) \rightarrow (R^{2n},\lambda)$ is a symplectic map.  Since
$f\colon (R^{2n},\omega) \rightarrow (R^{2n'},\omega')$ 
is a symplectic map and $f_1 = f f_2^{-1}$, it follows that $f_1\colon (R^{2n},\lambda) \rightarrow (R^{2n'},\omega')$
is a symplectic map, as desired.
\end{proof}

\paragraph{A partial order.}
Our goal now is to prove Theorem~\ref{thm:osinoetherian} below, which says
that the category of $\OSI(R)$-modules is locally Noetherian.  This requires the
existence of a certain partial ordering.  For $d \geq 0$ and $\omega$ a symplectic
form on $R^{2d}$, define
\[
\cP_R(d,\omega) = \bigsqcup_{n=0}^{\infty} \Hom_{\OSI'(R)}((R^{2d},\omega), R^{2n});
\]
here we are using our convention that $R^{2n}$ is given the standard symplectic form.

\begin{lemma} 
\label{lemma:osikey}
Let $R$ be a finite commutative ring.  Fix $d \geq 0$ and a symplectic form $\omega$ on $R^{2d}$.
There exists a well partial ordering $\preceq$ on $\cP_R(d,\omega)$ together with an extension 
$\leq$ of $\preceq$ to a total ordering such that the following holds.
Consider $f \in \Hom_{\OSI'(R)}((R^{2d}, \omega), R^{2n})$ and $g \in \Hom_{\OSI'(R)}((R^{2d}, \omega), R^{2n'})$ 
with $f \preceq g$.  Then there is some $\phi \in \Hom_{\OSI(R)}(R^{2n},R^{2n'})$ with the following two properties.
\begin{compactenum}[\indent \rm 1.]
\item We have $g = \phi f$.
\item If $f_1 \in \Hom_{\OSI'(R)}((R^{2d}, \omega), R^{2n''})$ satisfies $f_1 < f$, then
$\phi f_1 < g$.
\end{compactenum}
\end{lemma}

\noindent
We postpone the proof of Lemma \ref{lemma:osikey} until the end of this section.

\paragraph{Noetherian theorems.}
For $d \geq 0$ and $\omega$ a symplectic form
on $R^{2d}$, define $Q_{d,\omega}$ to be the $\OSI(R)$-module obtained by pulling back the $\OSI'(R)$-module
$P_{\OSI'(R),(R^{2d},\omega)}$, so
\[Q_{d,\omega}(R^{2n}) = \bk[\Hom_{\OSI'(R)}((R^{2d},\omega),R^{2n})].\]
The main consequence of Lemma \ref{lemma:osikey} is as follows.

\begin{theorem}
\label{thm:osinoetherian}
Let $R$ be a finite commutative ring.  Then the category of $\OSI(R)$-modules is locally Noetherian.  Moreover,
for $d \geq 0$ and $\omega$ a symplectic form on $R^{2d}$, the $\OSI(R)$-module $Q_{d,\omega}$ is finitely
generated.
\end{theorem}
\begin{proof}
The proof of Theorem~\ref{thm:ovicnoetherian} can easily be adapted to prove that in fact every submodule
of $Q_{d,\omega}$ is finitely generated (replace Lemma \ref{lemma:ovickey} with Lemma \ref{lemma:osikey}).
If $\omega$ is the standard symplectic form on $R^{2d}$, then $Q_{d,\omega} = P_{\OSI(R),R^{2d}}$.  We can
thus use Lemma \ref{lemma:noetheriancrit} to deduce that the category of $\OSI(R)$-modules is locally Noetherian.
\end{proof}

\noindent
We now use Theorem \ref{thm:osinoetherian} to deduce Theorem \ref{maintheorem:sinoetherian}, which asserts
that the category of $\SI(R)$-modules is locally Noetherian when $R$ is a finite commutative ring.

\begin{proof}[{Proof of Theorem \ref{maintheorem:sinoetherian}}]
By Lemma \ref{lem:finite-functor} and Theorem \ref{thm:osinoetherian}, 
it is enough to show that the inclusion functor
$\Phi\colon \OSI(R) \rightarrow \SI(R)$ is finite.  Fix $d \geq 0$ and $\omega$ a symplectic
form on $R^{2d}$.  Set $M = P_{\SI(R),(R^{2d},\omega)}$, so $M(R^{2n}) = \bk[\Hom_{\SI(R)}((R^{2d},\omega),R^{2n})]$;
here we remind the reader of our convention that $R^{2n}$ is endowed with its standard symplectic form.
Our goal is to prove that the $\OSI(R)$-module $\Phi^{\ast}(M)$ is finitely generated.
For every symplectic form $\lambda$ on $R^{2d}$ and every element of $\Iso_{\SI(R)}((R^{2d},\omega),(R^{2d},\lambda))$,
we get a morphism $Q_{d,\lambda} \rightarrow \Phi^{\ast}(M)$ in the obvious way.  Lemma \ref{lemma:osifactor}
implies that the resulting map
\[
\bigoplus_{\substack{\lambda\text{ a symplectic}\\\text{form on }R^{2d}}} \left(
\bigoplus_{\Iso_{\SI(R)}((R^{2d},\omega),(R^{2d},\lambda))} Q_{d,\lambda}\right) \longrightarrow \Phi^{\ast}(M)
\]
is a surjection.  Theorem \ref{thm:osinoetherian} implies that each $Q_{d,\lambda}$ is a finitely
generated $\OSI(R)$-module.  Since $R$ is finite, this proves that
$\Phi^{\ast}(M)$ is finitely generated.
\end{proof}

\paragraph{Insertion maps.}
For the proof of Lemma \ref{lemma:osikey}, we will need the following analogue of
Lemma \ref{lemma:insertion}.

\begin{lemma}
\label{lemma:insertionsi}
Let $R$ be a commutative local ring.  Fix $d \geq 0$, a symplectic form $\omega$ on $R^{2d}$, and
some $f \in \Hom_{\OSI'(R)}((R^{2d},\omega),R^{2n})$ and $g \in \Hom_{\OSI'(R)}((R^{2d},\omega),R^{2n'})$. 
Assume that there exists some $I \subset \{1,\ldots,n'\} \setminus \Set{$i$}{either $2i-1 \in S_r(g)$ or $2i \in S_r(g)$}$ 
such that $f$ can be obtained from $g$
by deleting the rows $J:=\Set{$2i-1$, $2i$}{$i \in I$}$ (observe that this implies that $S_r(f)$ is equal to $S_r(g)$ after removing $J$ from $\{1,\dots,2n'\}$ and renumbering the rows in order).
Let the rows of $g$ be $r_1,\ldots,r_{2n'}$. Then there exists $\phi \in \Hom_{\OSI(R)}(R^{2n},R^{2n'})$ with the following properties.
\begin{compactitem}
\item We have $g = \phi f$.
\item Let $h \in \Hom_{\OSI(R)}((R^{2d},\omega),R^{2n})$ be such that $S_r(h) = S_r(f)$. 
Then $\phi h$ is obtained from $h$ by inserting the 
row $r_j$ in position $j$ of $h$ for all $j \in J$.
\item Let $h \in \Hom_{\OSI(R)}((R^{2d},\omega),R^{2n})$ be such that $S_r(h) < S_r(f)$ in the lexicographic
order.  Then $S_r(\phi h) < S_r(g)$ in the lexicographic order.
\end{compactitem}
\end{lemma}
\begin{proof}
Set $J' = \{1,\ldots,2n'\} \setminus J$, and 
write $J = \{j_1 < \cdots < j_{2(n'-n)}\}$ and $J' = \{j_1' < \cdots < j_{2n}'\}$.
Next, write $S_r(f) = \{s_1 < \cdots < s_{2d}\} \subset \{1,\ldots,{2n}\}$, and for $1 \leq i \leq 2n'$, let
$r_i  = (r_i^1,\ldots,r_i^{2d})$.  Define $\widehat{r}_i \in R^{2n}$ to
be the element that has $r_i^j$ in position $s_j$ for $1 \leq j \leq 2d$ and zeros elsewhere.  
Finally, define $\phi\colon R^{2n} \rightarrow R^{2n'}$ to be the map given by the matrix
defined as follows.
\begin{compactitem}
\item For $1 \leq k \leq 2(n'-n)$, the $j_k^{\text{th}}$ row of $\phi$ is $\widehat{r}_k$.
\item For $1 \leq k \leq 2n$, the $(j_k')^{\text{th}}$ row of $\phi$ has a $1$ in position $k$
and zeros elsewhere.
\end{compactitem}
The fact that $g$ is row-adapted implies that $\phi$ is row-adapted, and the three listed conclusions
of the lemma are easy calculations.  The only non-trivial fact that must be proved is that 
$\phi\colon R^{2n} \rightarrow R^{2n'}$ is symplectic with respect to the standard symplectic forms $\omega_n$
and $\omega_{n'}$ on $R^{2n}$ and $R^{2n'}$, respectively.

Let $x_1,\ldots,x_{2n} \in R^{2n'}$ be the columns of $\phi$.  To
prove that $\phi$ is symplectic with respect to $\omega_{n}$ and $\omega_{n'}$ is equivalent to proving that
\begin{equation}
\label{eqn:conclusion}
\omega_{n'}(x_{2i-1},x_{2j-1}) = \omega_{n'}(x_{2i},x_{2j}) = 0 \quad \text{and} \quad \omega_{n'}(x_{2i-1},x_{2j}) = \delta_{ij} \quad \quad (1 \leq i,j \leq n).
\end{equation}
For $1 \leq i \leq 2n$, let $x_i' \in R^{2n'}$ be the element obtained by replacing the $j_k^{\text{th}}$ entry
of $x_i$ with a $0$ for $1 \leq k \leq 2(n'-n)$ and let $x_i'' = x_i - x_i'$.  By construction,
$x_1',\ldots,x_{2n}'$ is the standard basis for $R^{2n'}$ with the basis vectors indexed by $J$ removed, so
\[
\omega_{n'}(x_{2i-1}',x_{2j-1}') = \omega_{n'}(x_{2i}',x_{2j}') = 0 \quad \text{and} \quad \omega_{n'}(x_{2i-1}',x_{2j}') = \delta_{ij} \quad \quad (1 \leq i,j \leq n).
\]
Also, by construction we have $\omega_{n'}(x_i',x_j'') = 0$ for all $1 \leq i,j \leq 2n$.  It follows that \eqref{eqn:conclusion} is equivalent to showing that $\omega_{n'}(x_i'',x_j'') = 0$ for all $1 \leq i,j \leq 2n$.

The only $x_i''$ that are nonzero are $x_{s_i}''$ for $1 \leq i \leq 2d$, so it is enough
to show that that $\omega_{n'}(x_{s_i}'',x_{s_j}'')=0$ for $1 \leq i,j \leq 2d$.
Let $y_1,\ldots,y_{2d} \in R^{2n'}$ be the columns of $g$.  For $1 \leq i \leq 2d$, let
$y_i' \in R^{2n'}$ be the element obtained by replacing the $j_k^{\text{th}}$ entry
of $y_i$ with a $0$ for $1 \leq k \leq 2(n'-n)$ and let $y_i'' = y_i - y_i'$.  By construction,
we have $x_{s_i}'' = y_i''$ for $1 \leq i \leq 2d$, so we must show that
$\omega_{n'}(y_{i}'',y_{j}'') = 0$ for $1 \leq i,j \leq 2d$.

Let $z_1,\ldots,z_{2d} \in R^{2n}$ be the columns of $f$ and let $\bb_1,\ldots,\bb_{2d}$ be the standard
basis for $R^{2d}$.  Since $f\colon (R^{2d},\omega) \rightarrow R^{2n}$ and $g\colon (R^{2d},\omega) \rightarrow R^{2n'}$
are symplectic maps, it follows that
\begin{equation}
\label{eqn:symplecticcond}
\omega_{n}(z_i,z_j) = \omega(\bb_i,\bb_j) \quad \text{and} \quad \omega_{n'}(y_i,y_j) = \omega(\bb_i,\bb_j) \quad \quad (1 \leq i,j \leq 2d).
\end{equation}
Using the fact that $f$ is obtained by deleting appropriate pairs of rows of $g$, we see that
\begin{equation}
\label{eqn:deletioncond}
\omega_{n}(z_i,z_j) = \omega_{n'}(y_i',y_j') \quad \quad (1 \leq i,j \leq 2d).
\end{equation}
Combining \eqref{eqn:symplecticcond} and \eqref{eqn:deletioncond}, we see that 
$\omega_{n'}(y_i,y_j) = \omega_{n'}(y_i',y_j')$ for all $1 \leq i,j \leq 2d$.  By construction, we have
\[
\omega_{n'}(y_i,y_j') = \omega_{n'}(y_i',y_j) = \omega_{n'}(y_i',y_j') \quad \quad (1 \leq i,j \leq 2d).
\]
We conclude that for $1 \leq i,j \leq 2d$ we have $\omega_{2n'}(y_i'',y_j'')$ equal to 
\begin{align*}
\omega_{n'}(y_i - y_i',y_j-y_j') &= \omega_{n'}(y_i,y_j) - \omega_{n'}(y_i',y_j) - \omega_{n'}(y_i,y_j') + \omega_{n'}(y_i',y_j')\\
&= 2 \omega_{n'}(y_i',y_j') - 2 \omega_{n'}(y_i',y_j') = 0,
\end{align*}
as desired.
\end{proof}

\paragraph{Constructing the orders.}
We now prove Lemma \ref{lemma:osikey}.

\begin{proof}[Proof of Lemma \ref{lemma:osikey}]
The general case of Lemma \ref{lemma:osikey} can easily be
deduced from the case where $R$ is a local ring (the argument is identical to that for
Lemma \ref{lemma:ovickey}), so we can assume that $R$ is a local ring.
The first step is to construct $\preceq$.  Consider
$f \in \Hom_{\OSI'(R)}((R^{2d},\omega),R^{2n})$ and $g \in \Hom_{\OSI'(R)}((R^{2d},\omega),R^{2n'})$.  We will say that
$f \preceq g$ if $f$ can be obtained from $g$ by deleting the rows $\Set{$2i-1$, $2i$}{$i \in I$}$ for
some set 
\[I \subset \{1,\ldots,n'\} \setminus \Set{$i$}{either $2i-1 \in S_r(f)$ or $2i \in S_r(f)$}.\]  
This clearly defines a partial order on $\cP_R(d,\omega)$.

\begin{claim}
The partially ordered set $(\cP_R(d,\omega), \preceq)$ is well partially ordered.
\end{claim}
\begin{proof}[{Proof of Claim}]
Define
\[
\Sigma = \left(R^{d} \amalg \{\spadesuit\}\right) \times \left(R^{d} \amalg \{\spadesuit\}\right),
\]
where $\spadesuit$ is a formal symbol.  
We will prove that $(\cP_R(d,\omega),\preceq)$ is isomorphic to a subposet of the well partially ordered poset $\Sigma^\star$
from Lemma~\ref{lem:higman}, so by Lemma~\ref{lem:poset-noeth}(a) is itself well partially ordered.
Consider $f \in \Hom_{\OSI'(R)}((R^{2d},\omega), R^{2n})$.  For $1 \leq i \leq 2n$, let
\[r_i = \begin{cases}
\text{the $i^{\text{th}}$ row of $f$} & \text{if $i \notin S_r(f)$},\\
\spadesuit & \text{if $i \in S_r(f)$}.\end{cases}\]
Thus $r_i \in R^{2d} \amalg \{\spadesuit\}$.  For $1 \leq i \leq n$, let $\theta_i = (r_{2i-1},r_{2i}) \in \Sigma$.
We thus get an element $\theta_1 \cdots \theta_{n} \in \Sigma^{\ast}$.
We claim that the map taking $f$ to $\theta_1 \cdots \theta_{n}$ is an injection from
$\cP_R(d,\omega)$ to $\Sigma^{\star}$.
To see this, consider an element $\theta_1 \cdots \theta_n \in \Sigma^\star$ in its image.  Our
goal is to reconstruct $f$.  Write $\theta_i = (r_{2i-1},r_{2i})$, and set
$S = \Set{$1 \leq i \leq 2n$}{$r_i = \spadesuit$}$.  We can read off all the rows of $f$ except for those
whose index lies in $S$.  But we know that $S = S_r(f)$, so rows whose index lies in $S$ are precisely the
standard basis vectors in $R^d$, appearing in their natural order.  We have thus reconstructed all of $f$,
as claimed.  
That this injection is order-preserving is immediate from the definitions.  The only subtle point is the reason we used
the $\spadesuit$ rather than just using the basis vectors in $R^d$.  The point here is that
in the definition of the partial order on $\cP_R(d,\omega)$, we were not allowed to delete rows in
$S_r(f)$, so we have to distinguish them in $\Sigma$.
\end{proof}

We now extend $\preceq$ to a total order $\leq$.  Fix some arbitrary total order on $R^{2d}$, and consider
$f \in \Hom_{\OSI'(R)}((R^{2d},\omega),R^{2n})$ and $g \in \Hom_{\OSI'(R)}((R^{2d},\omega),R^{2n'})$.
We then determine if $f \leq g$ via the following procedure.  Assume without loss of
generality that $f \neq g$.
\begin{compactitem}
\item If $n < n'$, then $f<g$.
\item Otherwise, $n=n'$.  If $S_r(f) < S_r(g)$ in the lexicographic ordering, then $f < g$.
\item Otherwise, $S_r(f) = S_r(g)$.  Compare the sequences
of elements of $R^{2d}$ which form the rows of $f$ and $g$ using the lexicographic ordering
and the fixed total ordering on $R^{2d}$.
\end{compactitem}
It is clear that this determines a total order $\leq$ on $\cP_R(d,\omega)$ that extends $\preceq$.

It remains to check that these orders have the property claimed by the lemma.  Consider
$f \in \Hom_{\OSI'(R)}((R^{2d},\omega),R^{2n})$ and $g \in \Hom_{\OSI'(R)}((R^{2d},\omega),R^{2n'})$ with
$f \preceq g$.  By Lemma \ref{lemma:insertionsi}, there exists some $\phi \in \Hom_{\OSI(R)}(R^{2n},R^{2n'})$
such that $g = \phi f$.  Now consider $f_1 \in \Hom_{\OSI'(R)}((R^{2d},\omega),R^{2n})$ such that
$f_1 < f$.  We want to show that $\phi f_1 < \phi f$.  If $S_r(f_1) < S_r(f)$ in the lexicographic
order, then this is an immediate consequence of Lemma \ref{lemma:insertionsi}, so we can assume
that $S_r(f_1) = S_r(f)$.  In this case, it follows from Lemma \ref{lemma:insertionsi} that $\phi f_1$ and
$\phi f$ are obtained by inserting the same rows into the same places in $f_1$ and $f$, respectively.
It follows immediately that $\phi f_1 < \phi f$, as desired.
\end{proof}

\subsection{Counterexamples}
\label{section:noetheriannegative}

In this section, we prove Theorem \ref{maintheorem:nonnoetherian}, which asserts
that the category of $\tC$-modules is not locally Noetherian for $\tC \in \{\VI(R), \VIC(R), \SI(R)\}$ when $R$ is an infinite ring. This requires the following lemma.

\begin{lemma}
\label{lemma:nonnoetherian}
Let $\bk$ be a nonzero ring and $\Gamma$ be a group that contains a non-finitely generated subgroup.  Then the group ring $\bk[\Gamma]$ is not Noetherian.
\end{lemma}

\begin{proof}
Given a subgroup $H \subseteq \Gamma$, the kernel of the surjection $\bk[\Gamma] \rightarrow \bk[\Gamma/H]$ is a left ideal $I_H \subset \bk[\Gamma]$.  Clearly $H \subsetneqq H'$ implies $I_H \subsetneqq I_{H'}$, so the existence of an infinite ascending chain of left ideals in $\bk[\Gamma]$ follows from the existence of an infinite ascending chain of subgroups of $\Gamma$.  The existence of this is an immediate consequence
of the existence of a non-finitely generated subgroup of $\Gamma$.
\end{proof}

\begin{proof}[{Proof of Theorem \ref{maintheorem:nonnoetherian}}] 
For all of our choices of $\tC$, the group $\Gamma:=\Aut_{\tC}(R^2)$ contains $\SL_2(R)$. 
We claim that $\SL_2(R)$ contains a non-finitely generated subgroup.  There are two cases.  If $\Z \subseteq R$, then
$\SL_2(R)$ contains a rank $2$ free subgroup.  For instance, one can take the subgroup generated by
\[\begin{pmatrix} 1 & 2 \\ 0 & 1 \end{pmatrix} \quad \quad \text{and} \quad \quad 
\begin{pmatrix} 1 & 0 \\ 2 & 1 \end{pmatrix}.\]
This free subgroup contains a non-finitely generated subgroup.  If $\Z \nsubseteq R$, then the additive group of $R$
is annihilated by some positive integer $\ell$.  Since $R$ is infinite, this implies that the additive group of $R$
is not finitely generated.  The claim now follows from the fact that the additive group of $R$ is a subgroup
of $\SL_2(R)$ via the embedding
\[r \mapsto \begin{pmatrix} 1 & r \\ 0 & 1 \end{pmatrix}.\]

By Lemma~\ref{lemma:nonnoetherian}, the claim implies that $\bk[\Gamma]$ is not Noetherian. Define a $\tC$-module $M$ via the formula
\[M_V = \begin{cases}
\bk[\Aut_{\tC}(V)] & \text{if $V \cong \Z^2$},\\
0 & \text{otherwise}.\end{cases}\]
The morphisms are the obvious ones.  It is then clear that $M$ is finitely generated but contains non-finitely
generated submodules.
\end{proof}

\section{Partial resolutions and representation stability}
\label{section:abstractnonsense}

This section contains the machinery that we will use to prove our asymptotic structure theorem.
It contains three sections. Recall the definition of a (weak) complemented category from \S\ref{section:introasymptotic}. In \S\ref{section:autcomplementedcategories}, we will discuss
automorphisms of complemented categories.  Next, in \S\ref{section:exactsequence} we will state Theorem \ref{theorem:resolution}, which gives a certain partial resolution for finitely generated modules over a complemented category. Theorem \ref{theorem:resolution} is proven in \S\ref{section:proofofresolution}.  At the end of \S\ref{section:proofofresolution}, we will derive our asymptotic structure theorem (Theorem \ref{maintheorem:asymptoticstructure}) from Theorem \ref{theorem:resolution}.

The arguments in this section are inspired by \cite{ChurchEllenbergFarbNagpal}.  The formal manipulations in
\S \ref{section:autcomplementedcategories} are also reminiscent of some of those appearing in
\cite{WahlStability}.

\subsection{Automorphism groups of complemented categories}
\label{section:autcomplementedcategories}

Objects in complemented categories equipped with a generator have rich automorphism groups.  For instance, we have the following.

\begin{lemma}
\label{lemma:movedecomp}
Let $(\tA,\monpro)$ be a complemented category with generator $X$.
For $V,V' \in \tA$, the group $\Aut_{\tA}(V')$ acts transitively on the
set $\Hom_{\tA}(V,V')$.
\end{lemma}

\begin{proof}
Without loss of generality, $V' = X^k$ for some $k \geq 0$.
Consider morphisms $f,g \colon V \rightarrow X^k$.  Let $C$ and $D$ be the complements to $f(V)$ and $g(V)$, respectively,
and let $\phi \colon V \monpro C \rightarrow X^k$ and $\phi' \colon V \monpro D \rightarrow X^k$ be the associated
isomorphisms.  Writing $V \cong X^{\ell}$, we then have
$C,D \cong X^{k-\ell}$; in particular, there is an isomorphism $\eta \colon C \rightarrow D$.
The composition
\[X^k \xrightarrow{\phi^{-1}} V \monpro C \xrightarrow{\text{id} \monpro \eta} V \monpro D \xrightarrow{\phi'} X^k\]
is then an isomorphism $\Psi \colon X^k \rightarrow X^k$ such that $\Psi \circ \phi \colon V \monpro C \rightarrow X^k$ is
an isomorphism that restricts to $g \colon V \rightarrow X^k$, so $\Psi \circ f = g$.
\end{proof}

\begin{lemma}
\label{lemma:autcomplement}
Let $(\tA,\monpro)$ be a complemented category, let $V,V' \in \tA$, and define 
$\varphi\colon \Aut_{\tA}(V) \times \Aut_{\tA}(V') \rightarrow \Aut_{\tA}(V \monpro V')$ by the formula $\varphi(f,g) = f \monpro g$.  Then the following hold.
\begin{compactenum}[\indent \rm (a)]
\item The homomorphism $\varphi$ is injective.
\item Let $i \colon V \rightarrow V \monpro V'$
and $i' \colon V' \rightarrow V \monpro V'$ be the canonical maps.  Then
\begin{align*}
\Set{$\eta \in \Aut_{\tA}(V \monpro V')$}{$\eta \circ i' = i'$} &= \varphi(\Aut_{\tA}(V) \times \mathrm{id}_{V'}) \cong \Aut_\tA(V),\\
\Set{$\eta \in \Aut_{\tA}(V \monpro V')$}{$\eta \circ i = i$} &= \varphi(\mathrm{id}_{V} \times \Aut_{\tA}(V')) \cong \Aut_\tA(V').
\end{align*}
\end{compactenum}
\end{lemma}

\begin{proof}
(a) Consider $f_0 \in \Aut_{\tA}(V)$ and $g_0 \in \Aut_{\tA}(V')$
such that $\varphi(f_0,g_0) = \text{id}_{V \monpro V'}$.  The proofs that $f_0 = \text{id}_V$ and
$g_0 = \text{id}_{V'}$ are similar, so we will only give the details for $f_0$.  
Let $\mathbbm{1}$ be the identity of
$(\tA,\monpro)$ and let $u\colon\mathbbm{1} \rightarrow V'$ be the unique map (by assumption, $\mathbbm{1}$ is initial).  Since $g_0 \circ u = u$, we have a commutative diagram
\begin{equation}
\label{eqn:varphiinjectivediagram}
\begin{CD}
V \monpro \mathbbm{1} @>{\text{id}_V \monpro u}>> V \monpro V' \\
@VV{f_0 \monpro \text{id}_{\mathbbm{1}}}V       @VV{f_0 \monpro g_0}V \\
V \monpro \mathbbm{1} @>{\text{id}_V \monpro u}>> V \monpro V'.\end{CD}
\end{equation}
Since $f_0 \monpro g_0 = \text{id}_{V \monpro V'}$, we deduce that 
$(\text{id}_V \monpro u) \circ (f_0 \monpro \text{id}_{\mathbbm{1}}) = \text{id}_V \monpro u$.  
Like all $\tA$-morphisms, $\text{id}_V \monpro u$ is a monomorphism, so 
$f_0 \monpro \text{id}_{\mathbbm{1}} = \text{id}_{V \monpro \mathbbm{1}}$, and hence
$f_0 = \text{id}_V$, as desired.

(b) The isomorphisms are a consequence of (a). We next show that
\begin{align*}
\varphi(\Aut_{\tA}(V) \times \text{id}_{V'}) &\subset \Set{$\eta \in \Aut_{\tA}(V \monpro V')$}{$\eta \circ i' = i'$},\\
\varphi(\text{id}_{V} \times \Aut_{\tA}(V')) &\subset \Set{$\eta \in \Aut_{\tA}(V \monpro V')$}{$\eta \circ i = i$}.
\end{align*}
The proofs are similar, so we will only give the details for $\varphi(\text{id}_{V} \times \Aut_{\tA}(V'))$.
Consider $g_1 \in \Aut_{\tA}(V')$.  The map $\varphi(\text{id}_V,g_1) \circ i$ is the composition of
the canonical isomorphism $V \cong V \monpro \mathbbm{1}$ with the composition
\[V \monpro \mathbbm{1} \stackrel{\text{id}_{V} \monpro u}{\longrightarrow} V \monpro V' \stackrel{\text{id}_V \monpro g_1}{\longrightarrow} V \monpro V'.\]
The uniqueness of $u$ implies that this 
composition equals $\text{id}_V \monpro (g_1 \circ u) = \text{id}_V \monpro u$, so 
$\varphi(\text{id}_V,g_1) \circ i = i$, as desired.

We next show that
\begin{align*}
\varphi(\Aut_{\tA}(V) \times \text{id}_{V'}) &\supset \Set{$\eta \in \Aut_{\tA}(V \monpro V')$}{$\eta \circ i' = i'$},\\
\varphi(\text{id}_{V} \times \Aut_{\tA}(V')) &\supset \Set{$\eta \in \Aut_{\tA}(V \monpro V')$}{$\eta \circ i = i$}.
\end{align*}
The proofs are similar, so we will only give the details for $\varphi(\text{id}_{V} \times \Aut_{\tA}(V'))$.
Consider $h_2 \in \Aut_{\tA}(V \monpro V')$ such that $h_2 \circ i = i$.  We wish to find some $g_2 \in \Aut_{\tA}(V')$ such that $\varphi(\text{id}_V,g_2) = h_2$. Observe that $i'(V')$ is a complement to $i(V)$ and $h_2 \circ i'(V')$ is a complement to $h_2 \circ i(V) = i(V)$.  By the uniqueness of complements, there must exist $g_2 \in \Aut_{\tA}(V')$ such that $h_2 \circ i' = i' \circ g_2$. Under the injection
\[
\Hom_{\tA}(V \monpro V', V \monpro V') \hookrightarrow \Hom_{\tA}(V,V \monpro V') \times \Hom_{\tA}(V',V \monpro V'),
\]
both $h_2$ and $\text{id}_{V} \monpro g_2$ map to the same thing, namely $(i,i' \circ g_2)$. We conclude that $h_2 = \text{id}_{V} \monpro g_2 = \varphi(\text{id}_{V},g_2)$, as desired.
\end{proof}

\subsection{Recursive presentations and partial resolutions}
\label{section:exactsequence}

Fix a complemented category $(\tA,\monpro)$ with a generator $X$ and fix
an $\tA$-module $M$.  We wish to study several chain complexes associated to $M$.

\paragraph{Shift functor.}
Consider some $p \geq 0$ and some $V \in \tA$.  Define
\[
\ShiftSet_{V,p} = \Hom_{\tA}(X^p,V).
\]
This might be the empty set. The groups $\Aut(X)^p$ and $\Sym_p$ and the wreath product 
$\Sym_p \wr \Aut(X) = \Sym_p \ltimes \Aut(X)^p$ all 
act on $X^p$ and induce free actions on $\ShiftSet_{V,p}$. Set
\[
\ShiftSet'_{V,p} = \ShiftSet_{V,p} / \Aut(X)^p, \qquad\ShiftSet''_{V,p} = \ShiftSet_{V,p} / \Sym_p, \qquad \ShiftSet'''_{V,p} = \ShiftSet_{V,p} / (\Sym_p \wr \Aut(X)).
\]
For $h \in \ShiftSet_{V,p}$, let $W_h \subset V$ be the complement
of $h(X^p)$.  We then define an $\tA$-module $\OrderedShift_p M$ via the formula
\[
(\OrderedShift_p M)_V = \bigoplus_{h \in \ShiftSet_{V,p}} M_{W_h} \quad \quad (V \in \tA).
\]
As far as morphisms go, consider a morphism $f \colon  V \rightarrow V'$ and
some $h \in \ShiftSet_{V,p}$.  We then have $f \circ h \in \ShiftSet_{V',p}$.  Moreover,
letting $C \subset V'$ be the complement of $f(V)$, there is a natural isomorphism
$V \monpro C \cong V'$.  Under this isomorphism, we can identify $W_{f \circ h}$ with
$W_h \monpro C$.  The canonical morphism $W_h \rightarrow W_h \monpro C$ thus induces
a map $M_{W_h} \rightarrow M_{W_{f \circ h}}$.
This allows us to define (component by component) a morphism
$(\OrderedShift_p M)_{f} \colon  (\OrderedShift_p M)_V \rightarrow (\OrderedShift_p M)_{V'}$.  We will identify $\OrderedShift_0 M$ with $M$ in the obvious way.
As notation, we will write
$(\OrderedShift_p M)_{V,h}$ for the term of $(\OrderedShift_p M)_V$ associated to $h \in \ShiftSet_{V,p}$.

Define $\OrderedShift'_p M$ and $\OrderedShift''_p M$ and $\OrderedShift'''_p M$
by summing over $\ShiftSet'_{V,p}$ and $\ShiftSet''_{V,p}$ and $\ShiftSet'''_{V,p}$,
rather than $\ShiftSet_{V,p}$.
These are quotients of $\ShiftSet_{V,p}$ as $\tA$-modules by the submodules spanned by $(x,(-1)^g x)$ where $x \in M_{W,h}$ and $(-1)^g x \in M_{W, h \circ g}$ for $g \in \Aut(X)^p$
or $g \in \Sym_p$ or $g \in \Sym_p \wr \Aut(X)$, respectively, and $(-1)^g$ refers 
to the sign of the $\Sym_p$-component of $g$ in the second and third case (and is defined
to be $1$ in the first case).

\begin{lemma} 
\label{lemma:ordshift-noeth}
If $(\tA,\monpro)$ is a complemented category with generator $X$ and if 
$M$ is a finitely generated $\tA$-module, then $\OrderedShift_p M$ and 
$\OrderedShift'_p M$ and $\OrderedShift''_p M$ and $\OrderedShift'''_p M$
are finitely generated $\tA$-modules for all $p \ge 0$.
\end{lemma}
\begin{proof}
It is enough to prove this for $\OrderedShift_p M$ since this surjects onto the others.

There exist $V_1,\ldots,V_m \in \tA$ and $x_i \in M_{V_i}$ for
$1 \leq i \leq m$ such that $M$ is generated by $\{x_1,\ldots,x_m\}$.  For all $V \in \tA$,
let $h_V \colon X^p \rightarrow V \monpro X^p$ be the canonical morphism, so
$(\OrderedShift_p M)_{V \monpro X^p,h_V} = M_V$. For $1 \leq i \leq m$, define
$\overline{x}_i \in (\OrderedShift_p M)_{V_i \monpro X^p}$ to be the element
\[x_i \in (\OrderedShift_p M)_{V_i \monpro X^p,h_{V_i}} \subset (\OrderedShift_p M)_{V_i \monpro X^p}.\]
We claim that $\OrderedShift_p M$ is generated by $\overline{x}_1,\ldots,\overline{x}_m$.  Indeed, let $N$ be the submodule of $\OrderedShift_p M$ generated by the indicated elements.  Consider $V \in \tA$.  Our goal is to show that $N_V = (\OrderedShift_p M)_V$.  If
$(\OrderedShift_p M)_V \neq 0$, then $V \cong X^q$ for some $q \geq p$, and it is enough
to show that $N_{X^q} = (\OrderedShift_p M)_{X^q}$.  It is clear
that
$(\OrderedShift_p M)_{X^{q-p} \monpro X^p,h_{X^{q-p}}} \subset N_{X^q}$.
Lemma \ref{lemma:movedecomp} implies that $\Aut_{\tA}(X^q)$ acts transitively on
$\ShiftSet_{X^q,p}$.  We conclude that 
$N_{X^q}$ contains $(\OrderedShift_p M)_{X^q,h}$ for all $h \in \ShiftSet_{X^q,p}$, so
$N_{X^q} = M_{X^q}$. 
\end{proof}

\begin{remark}
The operation $\OrderedShift_p$ is an exact functor on the category of $\tA$-modules: given a morphism $f \colon M \to N$, there is a morphism
$\OrderedShift_p f \colon \OrderedShift_p M \to \OrderedShift_p N$ defined in the obvious way, and the same is true for $\OrderedShift'_p$ and $\OrderedShift''_p$ and $\OrderedShift'''_p$.
Note that $\OrderedShift_p$ can be identified with the $p^{\text{th}}$ iterate of $\OrderedShift_1$, and a similar statement holds for $\OrderedShift'_p$, but not for $\OrderedShift''_p$
or $\OrderedShift'''_p$.
\end{remark}

\paragraph{Chain complex.}
For $p \geq 1$, we now define a morphism $d \colon \OrderedShift_p M \rightarrow \OrderedShift_{p-1} M$ of 
$\tA$-modules as follows.  Consider $V \in \tA$.
For $1 \leq i \leq p$, let $t_i \colon X \rightarrow X^p$ be the canonical morphism of the $i^{\text{th}}$ term and let $s_i \colon X^{p-1} \rightarrow X^p$ be the morphism 
that corresponds to the $(p-1)$-tuple $(t_1, \dots, \widehat{t_i}, \dots, t_p)$ under the isomorphism
\[
\Hom_\tA(X^{p-1}, X^p) \cong \Hom_\tA(X, X^p) \times \cdots \times \Hom_\tA(X, X^p)
\]
from the definition of a complemented category. For $h \in \ShiftSet_{V,p}$, we have $h \circ s_i \in \ShiftSet_{V,p-1}$.  Moreover, we can 
identify $W_{h \circ s_i}$ with $W_h \monpro (h \circ t_i(X))$.
Define $d_i \colon \OrderedShift_p M \rightarrow \OrderedShift_{p-1} M$ to be the morphism that takes 
$(\OrderedShift_p M)_{V,h} = M_{W_h}$
to $(\OrderedShift_{p-1} M)_{V,h \circ s_i} = M_{W_h \monpro (h \circ t_i(X))}$
via the map induced by the canonical morphism $W_h \rightarrow W_h \monpro (h \circ t_i(X))$.  We then define
$d = \sum_{i=1}^p (-1)^{i-1} d_i$.  The usual argument shows that $d \circ d = 0$, so we have
defined a chain complex of $\tA$-modules:
\[
\OrderedShift_\ast M \colon \cdots \longrightarrow \OrderedShift_3 M \longrightarrow \OrderedShift_2 M \longrightarrow \OrderedShift_1 M \longrightarrow M.
\]
The differentials factor through the quotients $\OrderedShift'_p M$ and $\OrderedShift''_p M$ 
and $\OrderedShift'''_p M$, so we get complexes
\begin{alignat*}{10}
&\OrderedShift_\ast' M &&\colon \cdots &&\longrightarrow &&\OrderedShift'_3 M &&\longrightarrow &&\OrderedShift'_2 M &&\longrightarrow &&\OrderedShift'_1 M &&\longrightarrow &&M,\\
&\OrderedShift_\ast'' M &&\colon \cdots &&\longrightarrow &&\OrderedShift''_3 M &&\longrightarrow &&\OrderedShift''_2 M &&\longrightarrow &&\OrderedShift''_1 M &&\longrightarrow &&M,\\
&\OrderedShift_\ast''' M &&\colon \cdots &&\longrightarrow &&\OrderedShift'''_3 M &&\longrightarrow &&\OrderedShift'''_2 M &&\longrightarrow &&\OrderedShift'''_1 M &&\longrightarrow &&M.
\end{alignat*}

\begin{remark} \label{rmk:groupoid-alg}
Let $\tA^\circ$ be the underlying groupoid of $\tA$, i.e., the subcategory where we take all objects of $\tA$ and only keep the isomorphisms. Using the complemented structure of $\tA$, there is a natural symmetric monoidal structure on $\tA^\circ$-modules defined by 
\[
(F \otimes G)(V) = \bigoplus_{V = V' \monpro V''} F(V') \otimes_\bk G(V'').
\]
So we can define commutative algebras and their modules. Define an $\tA^\circ$-module $A_1$ by $V \mapsto \bk$ if $V$ is isomorphic to the generator of $\tA$ and $V \mapsto 0$ otherwise and set $\widetilde{A} = {\rm Sym}(A_1)$ where ${\rm Sym}$ denotes the free symmetric algebra. Then 
\[
\widetilde{A}(V) = \bigoplus_{\substack{\{L_1, \dots, L_n\}\\V = L_1 \monpro \cdots \monpro L_n}} \bk
\]
where the sum is over all unordered decompositions of $V$ into rank $1$ subspaces. There is a quotient $A$ of $\widetilde{A}$ given by $V \mapsto \bk$ for all $V$ (identify all summands above), and the category of $\tA$-modules is equivalent to the category of $A$-modules: a map $A \otimes M \to M$ is equivalent to giving a map $M(V') \to M(V)$ for each decomposition $V = V' \monpro V''$, and the associativity of multiplication is equivalent to these maps being compatible with composition.

Under this interpretation, $\OrderedShift'_\ast M$ becomes
\[
\cdots \to A_1^{\otimes 3} \otimes M \to A_1^{\otimes 2} \otimes M \to A_1 \otimes M \to M,
\]
while $\OrderedShift'''_\ast M$ is the Koszul complex of $M$ thought of as an $\widetilde{A}$-module:
\[
\cdots \to \mybigwedge{3} A_1 \otimes M \to \mybigwedge{2} A_1 \otimes M \to A_1 \otimes M \to M. \qedhere
\]
\end{remark}

\paragraph{Relation to finite generation.}
We pause now to make the following observation.

\begin{lemma}
\label{lemma:shiftfinitegen}
Let $(\tA,\monpro)$ be a complemented category with generator $X$ and 
let $M$ be an $\tA$-module over a ring $\bk$.  Assume that $M_V$ is a finitely generated $\bk$-module for all $V \in \tA$.
Then $M$ is finitely generated if and only if there exists some $N \geq 0$ such that
the map
$d \colon  (\OrderedShift_1 M)_{V} \rightarrow M_{V}$ is surjective for all $V \in \tA$ whose
$X$-rank is at least $N$.
\end{lemma}
\begin{proof}
If $M$ is generated by elements $x_1,\ldots,x_k$ with $x_i \in M_{V_i}$, then
we can take $N$ to be the maximal $X$-rank of $V_1,\ldots,V_k$ plus $1$.  Indeed, if $V \in \tA$ has $X$-rank at least $N$, then every morphism $V_i \rightarrow V$ factors through $W$ for some $W \subset V$ whose $X$-rank is one less than that of $V$, which implies that the map $d \colon  (\OrderedShift_1 M)_{V} \rightarrow M_{V}$ is surjective.  Conversely, if such an $N$ exists, then for a generating set we can combine generating sets for $M_{X^i}$ for $0 \leq i \leq N-1$ to get a generating set for $M$.
\end{proof}

\paragraph{Partial resolutions and representation stability.}
We finally come to our main theorem which will be used in the proof of Theorem~\ref{maintheorem:asymptoticstructure}.

\begin{theorem}
\label{theorem:resolution}
Let $(\tA,\monpro)$ be a complemented category with generator $X$.
Assume that the category of $\tA$-modules is locally Noetherian, 
and let $M$ be a finitely generated $\tA$-module.
Fix some $q \geq 1$.  If the $X$-rank of $V \in \tA$ is sufficiently large,
then the chain complex
\[
(\OrderedShift_q M)_V \longrightarrow (\OrderedShift_{q-1} M)_V \longrightarrow \cdots \longrightarrow (\OrderedShift_1 M)_V \longrightarrow M_V \longrightarrow 0
\]
is exact. The same holds if we replace $\OrderedShift_p$ by $\OrderedShift'_p$ or $\OrderedShift''_p$ or $\OrderedShift'''_p$.
\end{theorem}

The proof of Theorem \ref{theorem:resolution} is in \S \ref{section:proofofresolution}
after some preliminaries. 

\begin{remark}
\label{remark:differentresolutions}
The different resolutions given by Theorem \ref{theorem:resolution} are useful in
different contexts.  The resolution $\OrderedShift_\ast$ will be used in our
finite generation machine in \S \ref{section:themachine} (though the other
resolutions could also be used at the cost of complicating the necessary spaces).  For
$\tA = \FI$, the resolution $\OrderedShift''_{\ast}$ is a version of the
``central stability chain complex'' from \cite{PutmanRepStabilityCongruence}; this
also played an important role in \cite{ChurchEllenbergFarbNagpal}.  Finally,
the resolution $\OrderedShift'''_{\ast}$ is the ``most efficient'' of our resolutions in the sense that the terms are smallest.
\end{remark}

\subsection{Proof of Theorem~\ref{theorem:resolution}}
\label{section:proofofresolution}

We begin with some preliminary results needed for the proof of
Theorem \ref{theorem:resolution}.

\paragraph{Torsion submodule.}
If $M$ is an $\tA$-module, then define the {\bf torsion submodule} of $M$, denoted $T(M)$, to be the submodule defined via the formula
\[
T(M)_V = \Set{$x \in M_V$}{there exists a morphism $f \colon V \rightarrow W$ with $M_f(x)=0$}.
\]

\begin{lemma}
\label{lemma:torsionsubmodule}
Let $(\tA,\monpro)$ be a complemented category with generator $X$ and let $M$ be an $\tA$-module.  Then
$T(M)$ is an $\tA$-submodule of $M$.
\end{lemma}
\begin{proof}
Fix $V \in \tA$.  It suffices to show that $T(M)_V$ is closed under addition.  Given $x,x' \in T(M)_V$, pick morphisms 
$f \colon V \to W$ and $f' \colon V \to W'$ such that $M_{f}(x)=0$ and $M_{f'}(x') = 0$.  Let
$g \colon W \rightarrow W \monpro W'$ and $g' \colon W' \rightarrow W \monpro W'$ be the canonical morphisms.
Lemma \ref{lemma:movedecomp} implies that there exists some $h \in \Aut_{\tA}(W \monpro W')$ such that
$h \circ g \circ f = g' \circ f'$.  This morphism kills both $x$ and $x'$, and hence
also $x+x'$.
\end{proof}

\begin{lemma}
\label{lemma:torsionvanish}
Let $(\tA,\monpro)$ be a complemented category with generator $X$.  Assume that 
the category of $\tA$-modules is locally Noetherian
and let $M$ be a finitely generated $\tA$-module.  Then
there exists some $N \geq 0$ such that if $V \in \tA$ has $X$-rank at least $N$, then $T(M)_V = 0$.
\end{lemma}
\begin{proof}
Combining Lemma~\ref{lemma:torsionsubmodule} with our locally Noetherian assumption, we get that $T(M)$ is finitely generated, say by elements
$x_1,\ldots,x_p$ with $x_i \in T(M)_{V_i}$ for $1 \leq i \leq p$.  Let $f_i \colon  V_i \rightarrow W_i$
be a morphism such that $M_{f_i}(x_i) = 0$.  Define
$W = W_1 \monpro \cdots \monpro W_p$, and consider some $1 \leq i \leq p$.  Define
$\widetilde{f}_i \colon  V_i \rightarrow W$ to be the composition of $f_i$ with the canonical morphism
$W_i \hookrightarrow W$.
It is then clear that $M_{\widetilde{f}_i}(x_i)=0$.
If $N$ is the $X$-rank of $W$, we claim that $T(M)_V = 0$ whenever $V \in \tA$ has $X$-rank at least $N$. Indeed, every morphism $g \colon  V_i \rightarrow V$ can be factored as
\[
V_i \xrightarrow{\widetilde{f}_i} W \longrightarrow V,
\]
and thus $M_{g}(x_i) = 0$.  Since the $x_i$ generate $T(M)$, it follows that $T(M)_V = 0$.
\end{proof}

\paragraph{Torsion homology.}
Since $\OrderedShift_{\ast} M$ is a chain complex
of $\tA$-modules, we can take its homology and obtain an $\tA$-module
$\HH_q(\OrderedShift_{\ast} M)$ for each $q \geq 0$.  We then have the following.

\begin{lemma}
\label{lemma:homologytorsion}
Let $(\tA,\monpro)$ be a complemented category with generator $X$.
Fix some $q \geq 0$, and consider $V \in \tA$.  Let $\iota \colon V \rightarrow V \monpro X$ be the canonical morphism.  Then for all $x \in (\HH_q(\OrderedShift_{\ast} M))_V$, we have
$(\HH_q(\OrderedShift_{\ast} M))_{\iota}(x) = 0$. In particular, the same is true for any morphism $f$ that maps $V$ to $W$ with larger $X$-rank. The same holds if we replace $\OrderedShift_\ast$ by $\OrderedShift'_\ast$ or $\OrderedShift''_\ast$ or $\OrderedShift'''_\ast$.
\end{lemma}
\begin{proof}
For $\OrderedShift_\ast$, the relevant map on homology is induced by a map
\[\begin{CD}
\cdots @>>> (\OrderedShift_2 M)_{V \monpro X} @>>> (\OrderedShift_1 M)_{V \monpro X} @>>> M_{V \monpro X}     @>>> 0 \\
@.          @AA{I}A                   @AA{I}A                   @AA{I}A                 @.\\
\cdots @>>> (\OrderedShift_2 M)_V            @>>> (\OrderedShift_1 M)_V            @>>> M_{V}              @>>> 0\end{CD}\]
of chain complexes.  We will prove that $I$ is chain homotopic to the zero map. The chain homotopy will descend to 
a similar chain homotopy on $\OrderedShift'_\ast$ and $\OrderedShift''_\ast$ and
$\OrderedShift'''_\ast$, so this will imply the lemma.

Let $\iota \colon V \rightarrow V \monpro X$ be the canonical morphism.  Recall that
\[
(\OrderedShift_p M)_V = \bigoplus_{h \in \ShiftSet_{V,p}} M_{W_h} \quad \text{and} \quad (\OrderedShift_p M)_{V \monpro X} = \bigoplus_{h \in \ShiftSet_{V \monpro X,p}} M_{W_h}.
\]
The map $I \colon  (\OrderedShift_p M)_V \rightarrow (\OrderedShift_p M)_{V \monpro X}$ takes 
$(\OrderedShift_p M)_{V,h} = M_{W_h}$ to $(\OrderedShift_p M)_{V \monpro X,\iota \circ h} = M_{W_h \monpro X}$ 
via the map induced by the canonical morphism $W \rightarrow W \monpro X$.

We now define a chain homotopy map 
\[
G \colon  (\OrderedShift_p M)_V \rightarrow (\OrderedShift_{p+1} M)_{V \monpro X}
\] 
as follows. Given $h \in \ShiftSet_{V,p} = \Hom_\tA(X^p,V)$, define $\overline{h} \colon X^{p+1} \rightarrow V \monpro X$ to be the composition
\[
X^{p+1} = X \monpro X^{p} \longrightarrow X^p \monpro X \xrightarrow{h \monpro \text{id}} V \monpro X,
\]
where the first arrow flips the two factors using the symmetric monoidal structure on $\tA$.
We then have $\overline{h} \in \ShiftSet_{V \monpro X,p+1}$.  Since
\[
(\OrderedShift_p M)_{V,h} = (\OrderedShift_{p+1} M)_{V \monpro X,\overline{h}} = M_{W_h},
\]
we can define $G$ on $(\OrderedShift_p M)_{V,h}$ to be the identity map
$(\OrderedShift_p M)_{V,h} = (\OrderedShift_{p+1} M)_{V \monpro X,\overline{h}}$.

We now claim that $d G + G d = I$.  Indeed, on $(\OrderedShift_p M)_V$ the map $d G + G d$ takes the form
\[
\left(\sum_{i=1}^p (-1)^{i-1} G d_i \right) + \left(\sum_{j=1}^{p+1} (-1)^{j-1} d_j G\right).
\]
Straightforward calculations show that $d_1 G = I$ and that $G d_i = d_{i+1} G$ for $1 \leq i \leq p$. 
\end{proof}

\paragraph{Endgame.}
All the ingredients are now in place for the proof of Theorem \ref{theorem:resolution}.

\begin{proof}[{Proof of Theorem \ref{theorem:resolution}}]
For all $i \geq 0$, Lemma \ref{lemma:ordshift-noeth} implies that the $\tA$-module $\OrderedShift_{i} M$ is finitely
generated.  Our assumption that the category of $\tA$-modules 
is locally Noetherian then implies that $\HH_i(\OrderedShift_{\ast} M)$ is finitely
generated.  Lemma \ref{lemma:homologytorsion} implies that for all $i \geq 0$, we have
\[
T(\HH_i(\OrderedShift_{\ast} M)) = \HH_i(\OrderedShift_{\ast} M).
\]
The upshot of all of this is that we can apply Lemma \ref{lemma:torsionvanish} to obtain some $N_q \geq 0$ 
such that if $V \in \tA$ has $X$-rank at least $N_q$ and $0 \leq i \leq q$, then
$\HH_i(\OrderedShift_{\ast} M)_V = 0$.
In other words, the chain complex
\[
(\OrderedShift_q M)_V \longrightarrow (\OrderedShift_{q-1} M)_V \longrightarrow \cdots \longrightarrow (\OrderedShift_1 M)_V \longrightarrow M_V \longrightarrow 0
\]
is exact, as desired. For $\OrderedShift'_\ast$ and $\OrderedShift''_\ast$ and
$\OrderedShift'''_\ast$, the proof is the same.
\end{proof}

\paragraph{Asymptotic structure theorem.}
We conclude by proving Theorem~\ref{maintheorem:asymptoticstructure} (the main argument
here is very similar to that of \cite[Lemma 2.23]{ChurchPutmanJohnson}).

\begin{proof}[{Proof of Theorem \ref{maintheorem:asymptoticstructure}}]
Injective representation stability follows from Lemma \ref{lemma:torsionvanish}, and surjective
representation stability is immediate from finite generation.  All that remains to prove is 
central stability.  Using Theorem \ref{theorem:resolution}, choose $N$ large enough
so that the chain complex
\[(\OrderedShift_2 M)_V \longrightarrow (\OrderedShift_1 M)_V \longrightarrow M_V \longrightarrow 0\]
is exact whenever the $X$-rank of $V \in \tA$ is at least $N$.  Define $M'$ to be
the left Kan extension to $\tA$ of the restriction of $M$ to $\tA^N$.  The universal
property of the left Kan extension gives a natural transformation 
$\theta\colon M' \rightarrow M$ such that $\theta_V\colon (M')_V \rightarrow M_V$ is the
identity for all $V \in \tA$ whose $X$-rank is at most $N$.  We must prove that
$\theta_V$ is an isomorphism for all $V \in \tA$.  Letting $r$ be the $X$-rank of $V$,
the proof is by induction on $r$.  The base cases are when $0 \leq r \leq N$, where
the claim is trivial.  Assume now that $r>N$ and that the claim is true for all
smaller $r$.  The natural transformation $\theta$ induces natural transformations
$\theta_i\colon \OrderedShift_i M' \rightarrow \OrderedShift_i M$ for all $i \geq 0$, and
our inductive hypothesis implies that the maps
$(\theta_i)_V \colon (\OrderedShift_i M')_V \rightarrow (\OrderedShift_i M)_V$ are
isomorphisms for all $i \geq 1$.  We have a commutative diagram 
\[\begin{CD}
(\OrderedShift_2 M')_V @>>> (\OrderedShift_1 M')_V @>>> (M')_V @>>> 0 \\
@V{(\theta_2)_V}V{\cong}V  @V{(\theta_1)_V}V{\cong}V @V{\theta_V}VV @. \\
(\OrderedShift_2 M)_V @>>> (\OrderedShift_1 M)_V @>>> M_V @>>> 0 \end{CD}\]
We warn the reader that while the bottom row is exact, the top row is as yet only
a chain complex.  If $W \in \tA$ has $X$-rank at most $N$, then every $\tA$-morphism
$W \rightarrow V$ factors through an object whose $X$-rank is $r-1$ (for instance,
this is an easy consequence of Lemma \ref{lemma:movedecomp}).  Combining this
with the usual formula for a left Kan extension as a colimit, we deduce that
the map $(\OrderedShift_1 M')_V \rightarrow (M')_V$ is surjective.  
The fact that $\theta_V$ is an isomorphism now follows from an easy chase of the above
diagram.
\end{proof}

\section{Twisted homological stability}
\label{section:twistedstability}

In this section, we prove a general theorem which allows us to deduce twisted homological stability from untwisted
homological stability and local Noetherianity (see Theorem \ref{theorem:twistedstability} below;
this abstracts an argument of Church \cite{ChurchHomologicalAlgebra}).  Using this, we will prove Theorems \ref{maintheorem:gltwisted} and \ref{maintheorem:sptwisted}, which give very general twisted homological stability theorems
for $\SL_n^{\Unit}(R)$ and $\Sp_{2n}(R)$.  We remark that Wahl--Randal-Williams \cite{WahlStability}
have also recently and independently proved twisted homological stability theorems.  They require their coefficients
to be either ``polynomial'' or ``abelian'', but they allow much more general rings $R$.  Their
proofs are very different from ours.

\begin{lemma}
\label{lemma:identifyinduced}
Let $(\tA,\monpro)$ be a complemented category with a generator $X$.  Then for all $n \geq r$, we have
an isomorphism of $\Aut_{\tA}(X^n)$-representations
\[
\bk[\Hom_{\tA}(X^r,X^n)] \cong \bk[\Aut_\tA(X^n) / \Aut_\tA(X^{n-r})] \cong \Ind_{\Aut_{\tA}(X^{n-r})}^{\Aut_{\tA}(X^n)} \bk. 
\]
Consequently,
$\HH_k(\Aut_{\tA}(X^n);\bk[\Hom_{\tA}(X^r,X^n)]) \cong \HH_k(\Aut_{\tA}(X^{n-r});\bk)$.
\end{lemma}

\begin{proof}
Let $i \in \Hom_\tA(X^r, X^n)$ be the canonical morphism $X^r \to X^r \monpro X^{n-r}$. By Lemma~\ref{lemma:movedecomp}, $\Aut_\tA(X^n)$ acts transitively on the set $\Hom_\tA(X^r, X^n)$. By Lemma~\ref{lemma:autcomplement}(b), the stabilizer of $i$ in $\Aut_\tA(X^n)$ is identified with $\Aut_\tA(X^{n-r})$, which establishes the first isomorphism. 
The second follows from Shapiro's lemma \cite[Proposition III.6.2]{BrownCohomology}.
\end{proof}

Our main theorem is then as follows.  If $(\tA,\monpro)$ is a complemented category with a generator
$X$ and $M$ is an $\tA$-module, then denote $M_{X^n}$ by $M_n$.

\begin{theorem}
\label{theorem:twistedstability}
Let $(\tA,\monpro)$ be a complemented category with a generator $X$ and let $M$ be a finitely generated $\tA$-module
over a ring $\bk$.  Assume that the following hold.
\begin{compactenum}[\indent \rm 1.]
\item The category of $\tA$-modules over $\bk$ is locally Noetherian.
\item For all $k \geq 0$, the map 
$\HH_k(\Aut_{\tA}(X^n);\bk) \rightarrow \HH_k(\Aut_{\tA}(X^{n+1});\bk)$
induced by the map $\Aut_{\tA}(X^n) \rightarrow \Aut_{\tA}(X^{n+1})$ that
takes $f \in \Aut_{\tA}(X^n)$ to $f \monpro \mathrm{id}_{X} \in \Aut_{\tA}(X^{n+1})$
is an isomorphism for $n \gg 0$.
\end{compactenum}
Then for $k \geq 0$, the map 
$\HH_k(\Aut_{\tA}(X^n);M_n) \rightarrow \HH_k(\Aut_{\tA}(X^{n+1});M_{n+1})$
is an isomorphism for $n \gg 0$.
\end{theorem}
\begin{proof}
Combining Lemma~\ref{lemma:identifyinduced} with our second assumption, we see that the theorem is true if $M$ is a finite direct sum of representable $\tA$-modules $P_{\tA,x}$ with
$x \in A$ (see \S \ref{section:noetherianprelim}).
Our local Noetherian assumption implies that there exists an $\tA$-module resolution
\[
\cdots \longrightarrow P_2 \longrightarrow P_1 \longrightarrow P_0 \longrightarrow M \longrightarrow 0,
\]
where $P_i$ is a finite direct sum of representable $\tA$-modules for all $i \ge 0$. For all $n \geq 0$, we can combine \cite[Proposition VII.5.2]{BrownCohomology} and \cite[Eqn. VII.5.3]{BrownCohomology} to obtain a spectral sequence 
\[
{\rm E}^1_{pq}(n) = \HH_p(\Aut_{\tA}(X^n);(P_q)_n) \Longrightarrow \HH_{p+q}(\Aut_{\tA}(X^n);M_n).
\]
This spectral sequence is natural, so there is a map ${\rm E}^1_{\ast \ast}(n) \rightarrow {\rm E}^1_{\ast \ast}(n+1)$
which converges to the map $\HH_{\ast}(\Aut_{\tA}(X^n);M_n) \rightarrow \HH_{\ast}(\Aut_{\tA}(X^{n+1});M_{n+1})$.
By the above, for each pair $(p,q)$, the map ${\rm E}^1_{pq}(n) \rightarrow {\rm E}^1_{pq}(n+1)$ is an isomorphism
for $n \gg 0$. 
This implies that for each $k \geq 0$, the map
\[
\HH_k(\Aut_{\tA}(X^n);M_n) \rightarrow \HH_k(\Aut_{\tA}(X^{n+1});M_{n+1})
\]
is an isomorphism for $n \gg 0$, as desired.
\end{proof}

\begin{proof}[{Proof of Theorems \ref{maintheorem:gltwisted} and \ref{maintheorem:sptwisted}}]
Theorems \ref{maintheorem:gltwisted} and
\ref{maintheorem:sptwisted} are obtained by 
applying Theorem \ref{theorem:twistedstability} to the
categories $\VIC(R,\Unit)$ and $\SI(R)$, respectively.  The necessary
locally Noetherian theorems are Theorem \ref{maintheorem:vicnoetherian},
and Theorem \ref{maintheorem:sinoetherian}.
The necessary untwisted homological stability theorems are due
to van der Kallen \cite{VanDerKallenStability}
and Mirzaii--van der Kallen \cite{MirzaiiVanDerKallenStability}.  
\end{proof}

\section{A machine for finite generation}
\label{section:themachine}

This section contains the machine we will use to prove our finite generation results.  We begin with two sections of preliminaries: \S \ref{section:coef} is devoted to systems of coefficients, and \S \ref{section:equivariant} is devoted to the basics of equivariant homology.  Finally, \S \ref{section:congruence} contains our machine.
Our machine is based on an unpublished argument of Quillen for proving homological stability (see, e.g., \cite{HatcherWahl}).  The insight that that argument interacts well with central stability is due to the first author \cite{PutmanRepStabilityCongruence}.  This was later reinterpreted in the language of $\FI$-modules in \cite{ChurchEllenbergFarbNagpal}.  The arguments in this section abstract and generalize the arguments in \cite{PutmanRepStabilityCongruence} and \cite{ChurchEllenbergFarbNagpal}.
Throughout this section, $\bk$ is a fixed commutative ring.

\subsection{Systems of coefficients}
\label{section:coef}

Letting $\HDelta$ be the category whose objects are the finite sets $[n]_+ = \{0,\ldots,n\}$ and whose morphisms are strictly increasing maps, recall that a semisimplicial set $X$ is a contravariant functor from $\HDelta$ to the category of sets.  The image of $[n]_+$ is called the set of $n$-simplices and is denoted $X^n$.
A simplex $\sigma' \in X^m$ is a face of a simplex $\sigma \in X^n$ if $\sigma'$ is the image of $\sigma$ under the map $X^n \rightarrow X^m$ induced by some map $[m]_+ \rightarrow [n]_+$.  The geometric realization of $X$ is denoted $|X|$.  For more on semisimplicial sets, see \cite{FriedmanSimplicial} (where they are called $\Delta$-sets).

Fix a semisimplicial set $X$.  Observe that the set $\sqcup_{n=0}^{\infty} X^n$ forms
the objects of a category with a unique morphism $\sigma' \rightarrow \sigma$
whenever $\sigma'$ is a face of $\sigma$.  We will call this the {\bf simplex
category} of $X$. 

\begin{definition}
A {\bf coefficient system} on $X$ is a contravariant functor from the simplex
category of $X$ to the category of $\bk$-modules.
\end{definition}

\begin{remark}
In other words, a coefficient system $\mathfrak{F}$ on $X$ consists of
$\bk$-modules $\mathfrak{F}(\sigma)$ for simplices
$\sigma$ of $X$ and homomorphisms $\mathfrak{F}(\sigma' \rightarrow \sigma) \colon  \mathfrak{F}(\sigma) \rightarrow \mathfrak{F}(\sigma')$ whenever $\sigma'$ is a
face of $\sigma$.  These homomorphisms must satisfy the obvious compatibility condition.
\end{remark}

\begin{definition}
Let $\mathfrak{F}$ be a coefficient system on $X$.  The {\bf simplicial chain complex} 
of $X$ with coefficients in $\mathfrak{F}$ is as follows.  Define
\[\Chain_k(X;\mathfrak{F}) = \bigoplus_{\sigma \in X^{k}} \mathfrak{F}(\sigma).\]
Next, define a differential 
$\partial \colon  \Chain_k(C;\mathfrak{F}) \rightarrow \Chain_{k-1}(C;\mathfrak{F})$
in the following way.  Consider $\sigma \in X^{k}$.
We will denote an element of $\mathfrak{F}(\sigma) \subset \Chain_k(X;\mathfrak{F})$
by $c \cdot \sigma$ for $c \in \mathfrak{F}(\sigma)$.  
For $0 \leq i \leq k$, let $\sigma_i$ be the face of $\sigma$ associated to the unique
morphism $[k-1]_+ \rightarrow [k]_+$ of $\HDelta$ whose image does not contain $i$.
For $c \in \mathfrak{F}(\sigma)$, we then define
\[
\partial(c \cdot \sigma) = \sum_{i=0}^k (-1)^i c_i \cdot \sigma_i,
\]
where $c_i$ is the image of $c$ under the homomorphism
$\mathfrak{F}(\sigma_i \rightarrow \sigma) \colon  \mathfrak{F}(\sigma) \longrightarrow \mathfrak{F}(\sigma_i)$.
Taking the homology of $\Chain_{\ast}(X;\mathfrak{F})$ yields the {\bf homology groups of $X$ with coefficients in $\mathfrak{F}$},
which we will denote by $\HH_{\ast}(X;\mathfrak{F})$.
\end{definition}

\begin{remark}
If $\mathfrak{F}$ is the coefficient system that assigns $\bk$ to every simplex 
and the identity map to every inclusion of a face,
then $\HH_{\ast}(X;\mathfrak{F}) \cong \HH_{\ast}(X;\bk)$.
\end{remark}

\subsection{Equivariant homology}
\label{section:equivariant}

A good reference that contains proofs of everything we state in this section is \cite[\S VII]{BrownCohomology}.

Semisimplicial sets form a category $\SSset$ whose morphisms are natural transformations.  Given a 
semisimplicial set $X$ and a group $G$, an action of $G$ on $X$ consists of a homomorphism
$G \rightarrow \Aut_{\SSset}(X)$.  Unpacking this, an action of $G$ on $X$ consists
of actions of $G$ on $X^n$ for all $n \geq 0$ which satisfy the obvious compatibility
condition.  Observe that this induces an action of $G$ on $X$.  Also, there 
is a natural quotient semisimplicial set $X/G$ with $(X/G)^n = X^n / G$. 

\begin{remark}
If $G$ is a group acting on a semisimplicial set $X$, then there is a natural
continuous map $|X| \rightarrow |X/G|$ that factors through $|X|/G$.  In fact, it is easy
to see that $|X|/G = |X/G|$.
\end{remark}

\begin{definition}
Consider a group $G$ acting on a semisimplicial set $X$.
Let $EG$ be a contractible CW complex on which $G$ acts properly discontinuously
and freely, so $EG / G$ is a classifying space for $G$.
Define $EG \times_G X$ to be the quotient of $EG \times |X|$ 
by the diagonal action of $G$.  The {\bf $G$-equivariant homology groups of $X$}, 
denoted $\HH_\ast^G(X;\bk)$, are defined to be $\HH_\ast(EG \times_G X;\bk)$.
This definition does not depend on the choice of $EG$.  
The construction of $EG \times_G X$ is known as the {\bf Borel construction}.
\end{definition}

The following lemma summarizes two key properties of these homology groups.

\begin{lemma}
\label{lemma:equivarianthomology}
Consider a group $G$ acting on a semisimplicial set $X$.
There is a canonical map $\HH_{\ast}^G(X;\bk) \rightarrow \HH_{\ast}(G;\bk)$ such that if $X$ is $k$-acyclic, then the map $\HH_i^G(X;\bk) \rightarrow \HH_i(G;\bk)$ is an isomorphism for $i \leq k$.
\end{lemma}

\begin{proof}
The map $\HH_{\ast}^G(X;\bk) \rightarrow \HH_{\ast}(G;\bk)$ comes 
from the map $EG \times_G X \rightarrow EG/G$
induced by the projection of $EG \times X$ onto its first factor.  
The second claim is an immediate consequence of the spectral sequence whose ${\rm E}^2$ page is \cite[(7.2), \S VII.7]{BrownCohomology}.
\end{proof}

To calculate equivariant homology, we use a certain spectral sequence.  First, a definition.

\begin{definition}
Consider a group $G$ acting on a semisimplicial set $X$.  Define a coefficient
system $\HHHH_q(G,X;\bk)$ on $X/G$ as follows.  Consider a simplex $\sigma$ of $X/G$.
Let $\widetilde{\sigma}$ be any lift of $\sigma$ to $X$.  Set
\[
\HHHH_q(G,X;\bk)(\sigma) = \HH_q(G_{\widetilde{\sigma}};\bk),
\]
where $G_{\widetilde{\sigma}}$ is the stabilizer of $\widetilde{\sigma}$. This does 
not depend on the choice of $\widetilde{\sigma}$ since conjugation induces the identity 
on group homology, so it defines a coefficient system on $X/G$.
\end{definition}

Our spectral sequence is then as follows.  It can be easily extracted from
\cite[\S VII.8]{BrownCohomology}.

\begin{theorem}
\label{theorem:mainspectralsequence}
Let $G$ be a group acting on a semisimplicial set $X$.  There is a spectral sequence
\[
{\rm E}^1_{p,q} = \Chain_p(X/G;\HHHH_q(G,X;\bk)) \Longrightarrow \HH_{p+q}^G(X;\bk)
\]
with $d^1 \colon {\rm E}^1_{p,q} \rightarrow {\rm E}^1_{p-1,q}$ equal to the differential
of $\Chain_{\ast}(X/G; \HHHH_q(G,X;\bk))$.
\end{theorem}

\subsection{A machine for finite generation}
\label{section:congruence}

We now discuss a machine for proving that the homology groups of certain congruence subgroups form finitely generated
modules over a complemented category.

\paragraph{Congruence subgroups.}
Let $(\tB,\monpro)$ be a weak complemented category with a generator $Y$.  The assignment
$V \mapsto \Aut_{\tB}(\tilV)$ for $V \in \tB$
is functorial.  Indeed, given a $\tB$-morphism $\tilf \colon \tilV_1 \rightarrow \tilV_2$, let
$\tilC$ be the complement to $\tilf(\tilV_1) \subset \tilV_2$ and write $\tilV_2 \cong \tilV_1 \monpro \tilC$;
we can then extend automorphisms of $\tilV_1$ to $\tilV_2$ by letting them act as the identity on $\tilC$.  It
is clear that this is well-defined.  Next, let $(\tA,\monpro)$ be a complemented category
and let $\Psi \colon \tB \rightarrow \tA$ be a strong monoidal functor (recall that this is a monoidal
functor whose coherence maps are invertible).  An argument similar to the above
shows that the assignment
\[
\tilV \mapsto \Ker(\Aut_{\tB}(\tilV) \rightarrow \Aut_{\tA}(\Psi(\tilV))) \quad \quad (\tilV \in \tB)
\]
is functorial.  We will call this functor the {\bf level $\Psi$ congruence subgroup} of $\Aut_{\tB}$ and denote its value on $\tilV$ by $\Gamma_{\tilV}(\Psi)$.

\paragraph{Highly surjective functors.}
The above construction is particularly well-behaved on certain kinds of functors that we now define.
Let $(\tA,\monpro)$ be a complemented category with a generator $X$ and let $(\tB,\monpro)$ be a weak complemented
category with a generator $Y$.  A {\bf highly surjective functor} from $\tB$ to $\tA$ is a strong monoidal functor
$\Psi \colon \tB \rightarrow \tA$ satisfying:

\begin{compactitem}
\item $\Psi(Y) = X$.
\item For $\tilV \in \tB$, the map $\Psi_{\ast} \colon \Aut_{\tB}(\tilV) \rightarrow \Aut_{\tA}(\Psi(\tilV))$ is surjective.
\end{compactitem}

\begin{remark}
We only assume that $\Psi$ is a strong monoidal functor (as opposed to a strict monoidal functor), so for all $q \geq 0$ we only know that $\Psi(Y^q)$ and $X^q$ are naturally isomorphic rather than identical.  
To simplify notation, we will simply identify $\Psi(Y^q)$ and $X^q$ henceforth.  The
careful reader can insert appropriate natural isomorphisms as needed.
\end{remark}

\begin{lemma}
\label{lemma:highlysurjective}
Let $(\tA,\monpro)$ be a complemented category with a generator $X$, let $(\tB,\monpro)$ be a weak complemented
category with a generator $Y$, and let $\Psi \colon \tB \rightarrow \tA$ be a highly surjective functor. Then the following hold.
\begin{compactenum}[\indent \rm (a)]
\item For all $V \in \tA$, there exists some $\tilV \in \tB$ such that $\Psi(\tilV) \cong V$.
\item For all $\tilV_1,\tilV_2 \in \tB$, the map 
$\Psi_{\ast} \colon \Hom_{\tB}(\tilV_1,\tilV_2) \rightarrow \Hom_{\tA}(\Psi(\tilV_1),\Psi(\tilV_2))$ is surjective.
\end{compactenum}
\end{lemma}

\begin{proof}
For the first claim, we have $V \cong X^q$ for some $q \geq 0$, so $\Psi(Y^q) = X^q \cong V$.  

For the second claim, write $\tilV_1 \cong Y^q$ and $\tilV_2 \cong Y^r$, so $\Psi(\tilV_1) \cong X^q$ and $\Psi(\tilV_2) \cong X^r$.  
If $q>r$, then $\Hom_{\tA}(X^q,X^r) = \emptyset$ and the claim is trivial, so we can assume that
$q \le r$.  Consider some $f \in \Hom_{\tA}(X^q,X^r)$.  Let $\tilf_1 \in \Hom_{\tB}(Y^q,Y^r)$ be arbitrary (this
set is nonempty; for instance, it contains the composition of the canonical map $Y^q \rightarrow Y^q \monpro Y^{r-q}$
with an isomorphism $Y^q \monpro Y^{r-q} \cong Y^r$).  Lemma \ref{lemma:movedecomp} implies that there exists
some $f_2 \in \Aut_{\tA}(X^r)$ such that $f = f_2 \circ \Psi(\tilf_1)$.  Since $\Psi$ is highly surjective,
we can find $\tilf_2 \in \Aut_{\tB}(Y^r)$ such that $\Psi(\tilf_2) = f_2$.  Setting $\tilf = \tilf_2 \circ \tilf_1$,
we then have $\Psi(\tilf) = f$, as desired.
\end{proof}

\begin{lemma}
\label{lemma:hhhwelldefined}
Let $(\tA,\monpro)$ be a complemented category with generator $X$, let $(\tB,\monpro)$ be a weak complemented category
with generator $Y$, and let $\Psi \colon \tB \rightarrow \tA$ be a highly surjective functor.  Consider
$\tilf,\tilf' \in \Hom_{\tB}(Y^q,Y^r)$ such that $\Psi(\tilf) = \Psi(\tilf')$.  Then there
exists some $\tilg \in \Gamma_{Y^r}(\Psi)$ such that $\tilf' = \tilg \circ \tilf$.
\end{lemma}

\begin{proof}
Let $\tilC$ be the complement to $\tilf(Y^q) \subset Y^r$ and $\tilh \colon Y^r \rightarrow Y^q \monpro \tilC$ 
be the isomorphism such that $\tilh \circ \tilf \colon Y^q \rightarrow Y^q \monpro \tilC$ is the canonical map. Define $C = \Psi(\tilC)$, $h = \Psi(\tilh)$, and $f = \Psi(\tilf)$, so $h \circ f \colon X^q \rightarrow X^q \monpro C$ is the canonical map.  Lemma \ref{lemma:movedecomp} implies that
there exists some $\tilg_1 \in \Aut_{\tB}(Y^q \monpro \tilC)$ such that 
$\tilh \circ \tilf' = \tilg_1 \circ \tilh \circ \tilf$.  Define $g_1 = \Psi(\tilg_1)$, so
$h \circ f = g_1 \circ h \circ f$.  Lemma \ref{lemma:autcomplement}(b) implies that there
exists some $g_2 \in \Aut_{\tA}(C)$ such that $g_1 = \text{id} \monpro g_2$.  Since
$\Psi$ is highly surjective, we can find some $\tilg_2 \in \Aut_{\tB}(\tilC)$ such that $\Psi(\tilg_2) = g_2$.  

Define $\tilg \in \Aut_{\tB}(Y^r)$ to equal $\tilh^{-1} \circ \tilg_1 \circ (\text{id} \monpro \tilg_2^{-1}) \circ \tilh$.
Then $\Psi(\tilg) = h^{-1} \circ g_1 \circ (\text{id} \monpro g_2^{-1}) \circ h$, which is
the identity since $g_1 = \text{id} \monpro g_2$.  Thus $\tilg \in \Gamma_{Y^r}(\Psi)$.  Finally,
since $\tilh \circ \tilf$ is the canonical map, it follows that 
$\tilh \circ \tilf \circ (\text{id} \monpro \tilg_2^{-1}) = \tilh \circ \tilf$.  Thus
\[
\tilg \circ \tilf = \tilh^{-1} \circ \tilg_1 \circ (\text{id} \monpro \tilg_2^{-1}) \circ \tilh \circ \tilf = \tilh^{-1} \circ \tilg_1 \circ \tilh \circ \tilf = \tilh^{-1} \circ \tilh \circ \tilf' = \tilf'. \qedhere
\]
\end{proof}

\paragraph{Homology of congruence subgroups.}
Let $(\tA,\monpro)$ be a complemented category with generator $X$, let $(\tB,\monpro)$ be a weak complemented category
with generator $Y$,
and let $\Psi \colon \tB \rightarrow \tA$ be a highly surjective functor.  Fixing a ring $\bk$ and some $k \geq 0$, we
define an $\tA$-module $\HHH_k(\Psi; \bk)$ as follows.  Since $\tA$ is equivalent to the full subcategory spanned
by the $X^q$, it is enough to define $\HHH_k(\Psi;\bk)$ on these.  Set
\[
\HHH_k(\Psi;\bk)_{X^q} = \HH_k(\Gamma_{Y^q}(\Psi);\bk).
\]
To define $\HHH_k(\Psi;\bk)$ on morphisms, 
consider $f \in \Hom_{\tA}(X^q,X^r)$.  By Lemma \ref{lemma:highlysurjective},
there exists some $\tilf \in \Hom_{\tB}(Y^q,Y^r)$ such that $\Psi(\tilf) = f$, and we define
\[
\HHH_k(\Psi;\bk)_{f} \colon \HH_k(\Gamma_{Y^q}(\Psi);\bk) \rightarrow \HH_k(\Gamma_{Y^r}(\Psi);\bk)
\]
to be the map on homology induced by the map $\Gamma_{Y^q}(\Psi) \rightarrow \Gamma_{Y^r}(\Psi)$ induced by $\tilf$.
To see that this is well-defined, observe that if $\tilf' \in \Hom_{\tB}(Y^q,Y^r)$ also satisfies
$\Psi(\tilf') = f$, then Lemma \ref{lemma:hhhwelldefined} implies that there exists
some $\tilg \in \Gamma_{Y^r}(\Psi)$ such that $\tilf' = \tilg \circ \tilf$.  Since inner
automorphisms act trivially on group homology, we conclude that $\tilf$ and $\tilf'$ induce the same
map $\HH_k(\Gamma_{Y^q}(\Psi);\bk) \rightarrow \HH_k(\Gamma_{Y^r}(\Psi);\bk)$.

\paragraph{Space of $Y$-embeddings.}
We wish to give a sufficient condition for the $\tA$-module $\HHH_k(\Psi;\bk)$ to be finitely generated. This requires the following definition.  Consider $\tilV \in \tB$.  Recall that for all $k \geq 0$, we defined
\[
\ShiftSet_{\tilV,k} = \Hom_{\tB}(Y^k,\tilV).
\]
The {\bf space of $Y$-embeddings} is the semisimplicial set $\ShiftSet_{\tilV}$ defined as follows. First, define
$(\ShiftSet_{\tilV})^{k} = \ShiftSet_{\tilV,k+1}$. 
For $f \in \ShiftSet_{\tilV,k+1}$, we will denote the associated $k$-simplex of $\ShiftSet_{\tilV}$ by $[f]$.
Next, if $[f] \in (\ShiftSet_{\tilV})^{k}$ and $\phi \colon [\ell]_+ \rightarrow [k]_+$ 
is a morphism of $\HDelta$ (i.e., a strictly increasing map from $[\ell]_+ = \{0,\ldots,\ell\}$ to $[k]_+ = \{0,\ldots,k\}$), then there is a natural induced map $\phi_{\ast} \colon Y^{\ell+1} \rightarrow Y^{k+1}$ (here we are indexing the $Y$-factors starting at $0$).  We define the face of $[f]$ associated to $\phi$ to be $[f \circ \phi_{\ast}]$. 

\paragraph{Finite generation, statement.}
We can now state our main theorem.

\begin{theorem}
\label{theorem:finitegen}
Let $(\tA,\monpro)$ be a complemented category with a generator $X$, let $(\tB,\monpro)$ be a weak complemented
category with a generator $Y$, and let $\Psi \colon \tB \rightarrow \tA$ be a highly surjective functor.  Fix a
Noetherian ring $\bk$, and assume that the following hold.
\begin{compactenum}[\indent \rm 1.]
\item The category of $\tA$-modules is locally Noetherian.

\item For all $V \in \tA$ and $k \geq 0$, the $\bk$-module $\HHH_k(\Psi;\bk)_V$ is finitely generated.

\item For all $p \geq 0$, there exists some $N_p \geq 0$ such that for all $\tilV \in \tB$ whose $Y$-rank is at least $N_p$, the space $\ShiftSet_{\tilV}$ is $p$-acyclic.
\end{compactenum}
Then $\HHH_k(\Psi;\bk)$ is a finitely generated $\tA$-module for all $k \geq 0$.
\end{theorem}

We prove Theorem~\ref{theorem:finitegen} at the end of this section after some preliminary results.

\paragraph{Identifying the spectral sequence.}
Let $(\tB,\monpro)$ be a weak complemented category with generator $Y$.  For all $\tilV \in \tB$, the group $\Aut_{\tB}(\tilV)$ acts on $\ShiftSet_{\tilV}$.  The following lemma describes the restriction of this action to a congruence subgroup.

\begin{lemma}
\label{lemma:quotientemb}
Let $(\tA,\monpro)$ be a complemented category with a generator $X$, let $(\tB,\monpro)$ be a weak complemented
category with a generator $Y$, and let $\Psi \colon \tB \rightarrow \tA$ be a highly surjective functor.  Fix
some $\tilV \in \tB$, and set $V = \Psi(\tilV) \in \tA$.  Then
the group $\Gamma_{\tilV}(\Psi)$ acts on $\ShiftSet_{\tilV}$ 
and we have an isomorphism $\ShiftSet_{\tilV} / \Gamma_{\tilV}(\Psi) \cong \ShiftSet_{V}$
of semisimplicial sets.
\end{lemma}

\begin{proof}
The claimed action of $\Gamma_{\tilV}(\Psi)$ is the restriction of the obvious
action of $\Aut_{\tB}(\tilV)$.
We have a $\Gamma_{\tilV}(\Psi)$-invariant map $\ShiftSet_{\tilV} \rightarrow \ShiftSet_{V}$
of semisimplicial sets.  This induces a map
$\rho \colon \ShiftSet_{\tilV} / \Gamma_{\tilV}(\Psi) \rightarrow \ShiftSet_{V}$.
For all $k \geq 0$, Lemma \ref{lemma:highlysurjective} implies that $\rho$ is surjective on
$k$-simplices and Lemma \ref{lemma:hhhwelldefined}
implies that $\rho$ is injective on $k$-simplices.
We conclude that $\rho$ is an isomorphism of semisimplicial 
sets, as desired.
\end{proof}

Let $(\tA,\monpro)$ be a complemented category with a generator $X$ and let $\Psi \colon \tB \rightarrow \tA$ be a highly surjective functor.
Fix some $\tilV \in \tB$.  Our next goal is to understand the spectral sequence given by
Theorem \ref{theorem:mainspectralsequence} for the action of $\Gamma_{\tilV}(\Psi)$ on
$\ShiftSet_{\tilV}$.  Set $V = \Psi(\tilV) \in \tA$.  Using Lemma \ref{lemma:quotientemb}, this
spectral sequence is of the form
\[
{\rm E}^1_{p,q} = \Chain_p(\ShiftSet_{V}; \HHHH_q(\Gamma_{\tilV}(\Psi),\ShiftSet_{\tilV};\bk)) \Longrightarrow \HH_{p+q}^{\Gamma_{\tilV}(\Psi)}(\ShiftSet_{\tilV};\bk).
\]

Recall the definition of $\OrderedShift_p$ from \S\ref{section:exactsequence}.

\begin{lemma}
\label{lemma:identifye1}
Let $(\tA,\monpro)$ be a complemented category with a generator $X$, let $(\tB,\monpro)$ be a weak complemented
category with a generator $Y$, and let $\Psi \colon \tB \rightarrow \tA$ be a highly surjective functor.
Fix $\tilV \in \tB$, and set $V = \Psi(\tilV)$.  Finally, fix a ring $\bk$.  Then for $p,q \geq 0$ we have
\[
\Chain_{p-1}(\ShiftSet_{V}; \HHHH_q(\Gamma_{\tilV}(\Psi),\ShiftSet_{\tilV};\bk)) 
\cong (\OrderedShift_{p} \HHH_q(\Psi;\bk))_{V}.
\]
\end{lemma}

\begin{proof}
Recall that the $(p-1)$-simplices of $\ShiftSet_V$ are in bijection with the set
$\ShiftSet_{V,p}$ and that the simplex associated to $h \in \ShiftSet_{V,p}$ is denoted $[h]$.  It follows that
\[
\Chain_{p-1}(\ShiftSet_{V}; \HHHH_q(\Gamma_{\tilV}(\Psi), \ShiftSet_{\tilV};\bk)) 
= \bigoplus_{h \in \ShiftSet_{V,p}} \HHHH_q(\Gamma_{\tilV}(\Psi), \ShiftSet_{\tilV};\bk)([h]).
\]
Fix some $h \in \ShiftSet_{V,p}$, and let $\tilh \in \ShiftSet_{\tilV,p}$ be a lift of $\tilh$.  By definition, we have
\[
\HHHH_q(\Gamma_{\tilV}(\Psi),\ShiftSet_{\tilV};\bk)([h]) = \HH_q((\Gamma_{\tilV}(\Psi))_{[\tilh]};\bk).
\]
We have 
\[
(\Aut_{\tB}(\tilV))_{[\tilh]} = \Set{$\tilg \in \Aut_{\tB}(\tilV)$}{$\tilg \circ \tilh = \tilh$} \quad \text{and} \quad (\Gamma_{\tilV}(\Psi))_{[\tilh]} = \Set{$\tilg \in \Gamma_{\tilV}(\Psi)$}{$\tilg \circ \tilh = \tilh$}.
\]
Letting $\tilW_{\tilh} \subset \tilV$ be a complement to $\tilh(Y^{p})$, Lemma \ref{lemma:autcomplement}(b) implies that
$(\Aut_{\tB}(\tilV))_{[\tilh]} \cong \Aut_{\tB}(\tilW_{\tilh})$.  Under this isomorphism,
$(\Gamma_{\tilV}(\Psi))_{[\tilh]}$ goes to $\Gamma_{\tilW_{\tilh}}(\Psi)$.  Letting $W_h = \Psi(\tilW_{\tilh}) \subset V$, we
have that $W_h$ is a complement to $h(X^{p})$.  By definition, we have
\[
\HH_q((\Gamma_{\tilV}(\Psi))_{[\tilh]};\bk) = \HH_q(\Gamma_{\tilW_{\tilh}}(\Psi);\bk) = (\HHH_q(\Psi;\bk))_{W_h}.
\]
The upshot of all of this is that
\[
\Chain_{p-1}(\ShiftSet_{V}; \HHHH_q(\Gamma_{\tilV}(\Psi), \ShiftSet_{\tilV};\bk))
= \bigoplus_{h \in \ShiftSet_{V,p}} (\HHH_q(\Psi;\bk))_{W_h} = (\OrderedShift_{p} \HHH_q(\Psi;\bk))_{V}. \qedhere
\]
\end{proof}

\paragraph{Finite generation, proof.}
We finally prove Theorem \ref{theorem:finitegen}.

\begin{proof}[{Proof of Theorem \ref{theorem:finitegen}}]
The proof is by induction on $k$.  The base case $k=0$ is trivial, so assume that
$k>0$ and that $\HHH_{k'}(\Psi;\bk)$ is a finitely generated $\tA$-module for all $0 \leq k' < k$.
Since we are assuming that the $\bk$-module $\HHH_k(\Psi;\bk)_V$ is finitely generated
for all $V \in \tA$ and all $k \geq 0$, we can apply Lemma \ref{lemma:shiftfinitegen}
and deduce that it is enough to prove that if $V \in \tA$ has $X$-rank sufficiently
large, then the map
\begin{equation}
\label{eqn:surjectivegoal}
(\OrderedShift_1 \HHH_k(\Psi;\bk))_{V} \rightarrow \HHH_k(\Psi;\bk)_V
\end{equation}
is surjective.  Letting $\tilV \in \tB$ satisfy $\Psi(\tilV) = V$, we have
$\HHH_k(\Psi;\bk)_V = \HH_k(\Gamma_{\tilV}(\Psi);\bk)$.  To prove that \eqref{eqn:surjectivegoal} is
surjective, we will use the action of $\Gamma_{\tilV}(\Psi)$ on $\ShiftSet_{\tilV}$.

By Lemma \ref{lemma:equivarianthomology} and our assumption about the acyclicity of
$\ShiftSet_{\tilV}$, the map
\[
\HH_k^{\Gamma_{\tilV}(\Psi)}(\ShiftSet_{\tilV};\bk) \rightarrow \HH_k(\Gamma_{\tilV};\bk)
\]
is an isomorphism if the $Y$-rank of $\tilV$ (which equals the $X$-rank of $V$) is sufficiently large.
Combining Theorem~\ref{theorem:mainspectralsequence} with Lemma~\ref{lemma:identifye1}, 
we get a spectral sequence 
\[
{\rm E}^1_{p,q} \cong (\OrderedShift_{p+1} \HHH_q(\Psi;\bk))_{V} \Longrightarrow \HH_{p+q}^{\Gamma_{\tilV}(\Psi)}(\ShiftSet_{\tilV};\bk).
\]
Moreover, examining the isomorphism in Lemma~\ref{lemma:identifye1}, we see
that the differential $d \colon {\rm E}^1_{p,q} \rightarrow {\rm E}^1_{p-1,q}$ is identified
with the differential
\[
d \colon (\OrderedShift_{p+1} \HHH_q(\Psi;\bk))_{V} \rightarrow (\OrderedShift_{p} \HHH_q(\Psi;\bk))_{V}.
\]
Our inductive hypothesis says that $\HHH_q(\Psi;\bk)$
is a finitely generated $\tA$-module for $0 \leq q < k$.  Theorem \ref{theorem:resolution}
therefore implies that if the $X$-rank of $V$ is sufficiently large, we have
${\rm E}^2_{p,q} = 0$ for $0 \leq q < k$ and $1 \leq p \leq k+1$.  This implies that the map
\[
(\OrderedShift_1 \HHH_k(\Psi;\bk))_{V} = {\rm E}^1_{0,k} \rightarrow \HH_k^{\Gamma_{\tilV}(\Psi)}(\ShiftSet_{\tilV};\bk) \cong \HH_k(\Gamma_{\tilV};\bk)\]
is a surjection, as desired.
\end{proof}

\section{Finite generation theorems}
\label{section:finitegen}

In this section, we construct the modules claimed in the introduction and prove
Theorems \ref{maintheorem:glfinite}, \ref{maintheorem:slfinite},
\ref{maintheorem:autfinite}, \ref{maintheorem:spfinite}, and \ref{maintheorem:modfinite};
they are proved in \S \ref{section:glfinite}, \S \ref{section:glfinite}, \S \ref{section:autfinite},
\S \ref{section:spfinite}, and \S \ref{section:modfinite}, respectively.
These theorems assert that the homology groups of various congruence subgroups form
finitely generated modules over the appropriate categories.  They will all be deduced
from Theorem \ref{theorem:finitegen}.  We also introduce the weak complemented category
$\Surf$ in \S \ref{section:surfaces}.

There is an interesting parallel between the complexes that appear in this section and those that
appear in the paper \cite{WahlStability} of Wahl--Randal-Williams, which proves a variety of more
traditional homological stability theorems.  Our paper and theirs were written independently (though
in later revisions to save space we removed some proofs that duplicated work of theirs).

\subsection{The general and special linear groups: Theorems \ref{maintheorem:glfinite} and \ref{maintheorem:slfinite}}
\label{section:glfinite}

Let $\O_K$ be the ring of integers in an algebraic number field $K$, let
$\alpha \subset \O_K$ be a proper nonzero ideal, let $\Unit \subset \O_K/\alpha$ be the
image of the group of units $\widehat{\Unit} \subset \O_K$, and let $\bk$ be a Noetherian
ring.  In this section, we construct
$\VIC(\O_K/\alpha,\Unit)$-modules $\HHH_k(\GL(\O_K,\alpha);\bk)$ as in
Theorem \ref{maintheorem:glfinite} for all $k \geq 0$; 
see Corollary \ref{corollary:constructgl}.
We then prove Theorem \ref{maintheorem:glfinite},
which says that $\HHH_k(\GL(\O_K,\alpha);\bk)$ is a finitely generated 
$\VIC(\O_K/\alpha,\Unit)$-module for all $k \geq 0$.  Theorem \ref{maintheorem:slfinite} is a similar
result for the special linear group.  Its proof is similar to that of the general linear group, and we
leave it as an exercise to the reader.

\paragraph{Categories and functors.}
As was noted in Example \ref{example:vic}, for any ring $R$ the category $\VIC(R)$ is
a complemented category whose monoidal structure is direct sum.  We
have
\begin{equation}
\label{eqn:vicraut}
\Aut_{\VIC(R)}(R^n) = \GL_n(R) \quad \quad (n \geq 0).
\end{equation}
Define a strong monoidal functor $\Psi \colon \VIC(\O_K) \rightarrow \VIC(\O_K/\alpha, \Unit)$ by $\Psi(V) = (\O_K/\alpha) \otimes_{\O_K} V$.
The reason for the presence of $\Unit$ in
the target of $\Psi$ is that objects $V$ of $\VIC(\O_K)$ are equipped with $\widehat{\Unit}$-orientations (a vacuous
condition since $\widehat{\Unit}$ acts transitively on the set of orientations!), and these induce
$\Unit$-orientations on $(\O_K/\alpha) \otimes_{\O_K} V$.
The functor $\Psi$ takes the generator $\O_K^1$ of $\VIC(\O_K)$ to the generator $(\O_K/\alpha)^1$
of $\VIC(\O_K/\alpha,\Unit)$.  Define $\GL(\O_K,\alpha) = \Gamma(\Psi)$.  Using the identification
in \eqref{eqn:vicraut}, we have
$\GL(\O_K,\alpha)_{\O_K^n} = \GL_n(\O_K,\alpha)$.

\begin{lemma}
\label{lemma:glhighlysurjective}
The functor $\Psi$ is highly surjective.
\end{lemma}
\begin{proof}
This is equivalent to the fact that the map $\GL_n(\O_K) \rightarrow \SL_n^{\Unit}(\O_K/\alpha)$
is surjective for all $n \geq 0$.  This follows from the fact that the map $\SL_n(\O_K) \rightarrow \SL_n(\O_K/\alpha)$
is surjective, which follows from strong approximation (see, e.g., 
\cite[Chapter 7]{PlatonovRapinchuk}; this is a special feature of rings of integers).
\end{proof}

\noindent
As we discussed in \S \ref{section:congruence}, Lemma \ref{lemma:glhighlysurjective}
has the following corollary.

\begin{corollary}
\label{corollary:constructgl}
For all $k \geq 0$, there exists a $\VIC(\O_K/\alpha,\Unit)$-module
$\HHH_k(\GL(\O_K,\alpha);\bk)$ such that
$\HHH_k(\GL(\O_K,\alpha);\bk)_{(\O_K/\alpha)^n} = \HH_k(\GL_n(\O_K,\alpha);\bk)$.
\end{corollary}

\paragraph{Identifying the space of embeddings.}
We now wish to give a concrete description of the space of $\O_K^1$-embeddings
of $\O_K^n$ (recall the definition of the space of $Y$-embeddings from \S\ref{section:congruence}).  A {\bf split partial basis} for $\O_K^n$ consists of elements
$((v_0,w_0),\ldots,(v_p,w_p))$ as follows.
\begin{compactitem}
\item The $v_i$ are elements of $\O_K^n$ that form part of a free
basis for the free $\O_K$-module $\O_K^n$.
\item The $w_i$ are elements of $\Hom_{\O_K}(\O_K^n,\O_K)$.
\item For all $0 \leq i,j \leq p$, we have $w_i(v_j) = \delta_{i,j}$ (Kronecker delta).
\end{compactitem}
These conditions imply that the $w_i$ also form part
of a free basis for the free $\O_K$-module $\Hom_{\O_K}(\O_K^n,\O_K)$.  The
{\bf space of split partial bases} for $\O_K^n$, denoted $\PartialBases_n(\O_K)$, is the semisimplicial
set whose $p$-simplices are split partial bases
$((v_0,w_0),\ldots,(v_p,w_p))$ for $\O_K^n$.

\begin{lemma}
\label{lemma:glidentify}
The space of $\O_K^1$-embeddings of $\O_K^n$ is isomorphic to $\PartialBases_n(\O_K)$.
\end{lemma}
\begin{proof}
Fix some $p \geq 0$.
The $p$-simplices of the space of $\O_K^1$-embeddings of $\O_K^n$ are in bijection
with elements of $\Hom_{\VIC(\O_K)}(\O_K^{p+1},\O_K^n)$.  We
will define set maps
\[
\eta_1 \colon \Hom_{\VIC(\O_K)}(\O_K^{p+1},\O_K^n) \rightarrow (\PartialBases_n(\O_K))^p, \quad \eta_2 \colon (\PartialBases_n(\O_K))^p \rightarrow \Hom_{\VIC(\O_K)}(\O_K^{p+1},\O_K^n)
\]
that are compatible with the semisimplicial structure and satisfy
$\eta_1 \circ \eta_2 = \text{id}$ and $\eta_2 \circ \eta_1 = \text{id}$.  The
constructions are as follows.
\begin{compactitem}
\item Consider $(f,C) \in \Hom_{\VIC(\O_K)}(\O_K^{p+1},\O_K^n)$,
so by definition $f \colon \O_K^{p+1} \rightarrow \O_K^n$ is an $\O_K$-linear 
injection and $C \subset \O_K^n$ is a complement to $f(\O_K^{p+1})$.  
For $0 \leq i \leq p$, define $v_i \in \O_K^n$ to be the image under $f$ of
the $(i-1)^{\text{st}}$ basis element of $\O_K^{p+1}$ and define 
$w_i \in \Hom_{\O_K}(\O_K^n,\O_K)$ via the formulas
\[
w_i(v_j) = \delta_{ij} \quad \text{and} \quad w_i|_C = 0 \quad \quad (0 \leq j \leq p).
\]
Then $\eta_1(f,C):=((v_0,w_0),\ldots,(v_{p},w_{p}))$ is a split partial basis.
\item Consider a split partial basis $((v_0,w_0),\ldots,(v_{p+1},w_{p+1}))$.
Define $f \colon \O_K^{p+1} \rightarrow \O_K^n$ to be the map that takes the 
$i^{\text{th}}$ basis element to $v_{i-1}$.  Also, define $C = \bigcap_{i=0}^p \Ker(w_i)$.  
Then $\eta_2(((v_0,w_0),\ldots,(v_{p},w_{p}))) := (f,C)$ is an 
element of $\Hom_{\VIC(\O_K)}(\O_K^{p+1},\O_K^n)$. \qedhere
\end{compactitem}
\end{proof}

\noindent
The space $\PartialBases_n(\O_K)$ was introduced by Charney \cite{CharneyCongruence},
who proved the following theorem.

\begin{theorem}[{\cite[Theorem 3.5]{CharneyCongruence}}]
\label{theorem:partialbasescon}
The space $\PartialBases_n(\O_K)$ is $\frac{n-4}{2}$-acyclic.
\end{theorem}

\begin{remark}
Charney actually worked over rings that satisfy Bass's stable
range condition $\text{SR}_r$; the above follows from her work
since $\O_K$ satisfies $\text{SR}_3$.  Also, \cite[Theorem 3.5]{CharneyCongruence}
contains the data of an ideal $\alpha \subset \O_K$ (not the same
as the ideal we are working with!); the reference
reduces to Theorem \ref{theorem:partialbasescon} when $\alpha=0$.  Finally,
Charney's definition of $\PartialBases_n(R)$ is different from ours, but
for Dedekind domains the definitions are equivalent; see \cite{ReinerUnimodular}.
\end{remark}

\paragraph{Putting it all together.}
We now prove Theorem \ref{maintheorem:glfinite}, which asserts that the
$\VIC(\O/\alpha,\Unit)$-module $\HHH_k(\GL(\O_K,\alpha);\bk)$ is finitely generated.

\begin{proof}[{Proof of Theorem \ref{maintheorem:glfinite}}]
We apply Theorem \ref{theorem:finitegen} to the highly surjective
functor $\Psi$.  This theorem has three hypotheses.
The first is that the category of $\VIC(\O_K/\alpha,\Unit)$-modules is locally Noetherian, which is Theorem 
\ref{maintheorem:vicnoetherian}.  The second is that
$\HH_k(\GL_n(\O_K,\alpha);\bk)$ is a finitely generated $\bk$-module for all
$k \geq 0$, which is a theorem of Borel--Serre \cite{BorelSerreCorners}.  The
third is that the space of $\O_K^1$-embeddings is highly acyclic, which
follows from Lemma \ref{lemma:glidentify} and Theorem \ref{theorem:partialbasescon}.
We conclude that we can apply Theorem \ref{theorem:finitegen} and deduce
that $\HHH_k(\GL(\O_K,\alpha);\bk)$ is a finitely generated $\VIC(\O_K/\alpha,\Unit)$-module
for all $k \geq 0$, as desired.
\end{proof}

\subsection{The automorphism group of a free group: Theorem \ref{maintheorem:autfinite}}
\label{section:autfinite}

Fix $\ell \geq 2$, and let $\bk$ be a Noetherian ring.  In this section,
we construct $\VIC(\Z/\ell,\pm 1)$-modules $\HHH_k(\Aut(\ell);\bk)$ as in
Theorem \ref{maintheorem:autfinite} for all $k \geq 0$; see Corollary
\ref{corollary:constructaut}.  We then prove Theorem \ref{maintheorem:autfinite},
which says that $\HHH_k(\Aut(\ell);\bk)$ is a finitely generated 
$\VIC(\Z/\ell,\pm 1)$-module for all $k \geq 0$.

\paragraph{The category of free groups.}
Define $\Fr$ to be a category whose objects are the finite-rank free groups and whose monoidal structure is the free product $\ast$.  The $\Fr$-morphisms
from $F$ to $F'$ are pairs $(f,C) \colon F \rightarrow F'$, where $f \colon F \rightarrow F'$ is an injective homomorphism and $C \subseteq F'$ is a subgroup such that $F' = f(F) \ast C$.  The category $\Fr$ is a complemented category.
The subobject attached to a morphism $(f,C)\colon F \rightarrow F'$ can be thought of as the pair $(f(F), C)$ and its complement is the subobject $(C, f(F))$. A generator is the rank $1$ free group $F_1 \cong \Z$, and
\begin{equation}
\label{eqn:autaut}
\Aut_{\Fr}(F_n) = \Aut(F_n) \quad \quad (n \geq 0).
\end{equation}
The category $\Fr$ was first introduced by Djament--Vespa \cite{DjamentVespa2} to study the twisted homology of $\Aut(F_n)$.

\paragraph{Functor.}
The category $\VIC(\Z/\ell,\pm 1)$ is a complemented category whose monoidal structure is given by the direct sum. 
There is a strong monoidal functor $\Psi \colon \Fr \rightarrow \VIC(\Z/\ell,\pm 1)$ defined via the
formula $\Psi(F) = \HH_1(F;\Z/\ell)$.
The reason for the $\pm 1$ here is that $\Psi$ factors through the functor $\Fr \rightarrow \VIC(\Z)$ that
takes $F$ to $\HH_1(F;\Z)$, and elements of $\VIC(\Z)$ are equipped with $\pm 1$-orientations (a vacuous
condition since the group of units of $\Z$ is $\pm 1$).
The functor $\Psi$ takes the generator $F_1 \cong \Z$ of $\Fr$ to the generator $(\Z/\ell)^1$
of $\VIC(\Z/\ell, \pm 1)$.  Define $\Aut(\ell) = \Gamma(\Psi)$.  Using the identification
in \eqref{eqn:autaut}, we have
$\Aut(\ell)_{F_n} = \Aut(F_n,\ell)$.

\begin{lemma}
\label{lemma:authighlysurjective}
The functor $\Psi$ is highly surjective.
\end{lemma}
\begin{proof}
This is equivalent to the fact that the map $\Aut(F_n) \rightarrow \GL_n(\Z)$
is surjective for all $n \geq 0$, which is classical (see, e.g., 
\cite[Chapter I, Proposition 4.4]{LyndonSchupp}).
\end{proof}

\noindent
As we discussed in \S \ref{section:congruence}, Lemma \ref{lemma:authighlysurjective}
has the following corollary.

\begin{corollary}
\label{corollary:constructaut}
For all $k \geq 0$, there exists a $\VIC(\Z/\ell, \pm 1)$-module
$\HHH_k(\Aut(\ell);\bk)$ such that
$\HHH_k(\Aut(\ell);\bk)_{(\Z/\ell)^n} = \HH_k(\Aut(F_n,\ell);\bk)$.
\end{corollary}

\paragraph{Identifying the space of embeddings.}
We now wish to give a concrete description of the space of $F_1$-embeddings
of $F_n$.   A {\bf split partial basis} for $F_n$ consists of elements
$((v_0,C_0),\ldots,(v_p,C_p))$ as follows.
\begin{compactitem}
\item The $v_i$ are elements of $F_n$.
\item The $C_i$ are subgroups of $F_n$ such that $F_n = \langle v_i \rangle \ast C_i$.
\item For all $0 \leq i,j \leq p$ with $i \neq j$, we have $v_j \in C_i$.
\end{compactitem}
We remark that the condition $F_n = \langle v_i \rangle \ast C_i$ implies that
$C_i \cong F_{n-1}$.  Also, the conditions together imply that
$\bigcap_{i=0}^p C_i \cong F_{n-p-1}$ and that the $v_i$ can be extended to a free basis
for $F_n$.
The {\bf space of split partial bases} for $F_n$, 
denoted $\PartialBases(F_n)$, is the semisimplicial
set whose $p$-simplices are split partial bases
$((v_0,C_0),\ldots,(v_p,C_p))$ for $F_n$.  We remark that the $C_i$ in this
space play the same role as the kernels of the $w_i$ in $\PartialBases_n(\O_K)$.

\begin{lemma}
\label{lemma:autidentify}
The space of $F_1$-embeddings of $F_n$ is isomorphic to $\PartialBases(F_n)$.
\end{lemma}
\begin{proof}
Fix some $p \geq 0$.
The $p$-simplices of the space of $F_1$-embeddings of $F_n$ are in bijection
with elements of $\Hom_{\Fr}(F_{p+1},F_n)$.  We
will define set maps
\[
\eta_1 \colon \Hom_{\Fr}(F_{p+1},F_n) \rightarrow (\PartialBases(F_n))^p, \qquad \eta_2 \colon (\PartialBases(F_n))^p \rightarrow \Hom_{\Fr}(F_{p+1},F_n)
\]
that are compatible with the semisimplicial structure and satisfy
$\eta_1 \circ \eta_2 = \text{id}$ and $\eta_2 \circ \eta_1 = \text{id}$.  The
constructions are as follows.
\begin{compactitem}
\item Consider $(f,C) \in \Hom_{\Fr}(F_{p+1},F_n)$,
so by definition $f \colon F_{p+1} \rightarrow F_n$ is an injective homomorphism
and $C \subset F_n$ is a subgroup satisfying $F_n = f(F_{p+1}) \ast C$.
For $0 \leq i \leq p$, define $v_i \in F_n$ to be the image under $f$ of
the $(i-1)^{\text{st}}$ basis element of $F_{p+1}$ and define
$C_i \subset F_n$ to be the subgroup generated by $\Set{$v_j$}{$i \neq j$} \cup C$.
Then $\eta_1(f,C):=((v_0,C_0),\ldots,(v_{p},C_{p}))$ is a split partial basis.
\item Consider a split partial basis $((v_0,C_0),\ldots,(v_{p+1},C_{p+1}))$.
Define $f \colon F_{p+1} \rightarrow F_n$ to be the map that takes the
$i^{\text{th}}$ basis element to $v_{i-1}$.  Also, define $C = \bigcap_{i=0}^p C_i$.
Then $\eta_2(((v_0,C_0),\ldots,(v_{p},C_{p}))) := (f,C)$ is an
element of $\Hom_{\Fr}(F_{p+1},F_n)$. \qedhere
\end{compactitem}
\end{proof}

\noindent
The space $\PartialBases(F_n)$ appears in work of Wahl--Randal-Williams, where they prove the following theorem (making
essential use of work of Hatcher--Vogtmann \cite{HatcherVogtmannAut}).

\begin{theorem}[{\cite[Proposition 5.3]{WahlStability}}]
\label{theorem:partialbasesfncon}
The space $\PartialBases(F_n)$ is $\frac{n-3}{2}$-connected.
\end{theorem}

\paragraph{Putting it all together.}
We now prove Theorem \ref{maintheorem:autfinite}, which asserts that the
$\VIC(\Z/\ell, \pm 1)$-module $\HHH_k(\Aut(\ell);\bk)$ is finitely generated.

\begin{proof}[{Proof of Theorem \ref{maintheorem:autfinite}}]
We apply Theorem~\ref{theorem:finitegen} to the highly surjective
functor $\Psi$.  This theorem has three hypotheses.
The first is that the category of $\VIC(\Z/\ell, \pm 1)$-modules is locally Noetherian, which is Theorem~\ref{maintheorem:vicnoetherian}.  The second is that
$\HH_k(\Aut(F_n);\bk)$ is a finitely generated $\bk$-module for all
$k \geq 0$, which is a theorem of Culler--Vogtmann 
\cite[Corollary on p.\ 93]{CullerVogtmannAut}.
The third is that the space of $F_1$-embeddings is highly acyclic, which
follows from Lemma \ref{lemma:autidentify} and Theorem \ref{theorem:partialbasesfncon}.
We conclude that we can apply Theorem \ref{theorem:finitegen} and deduce
that $\HHH_k(\Aut(\ell);\bk)$ is a finitely generated $\VIC(\Z/\ell,\pm 1)$-module
for all $k \geq 0$, as desired.
\end{proof}

\subsection{The symplectic group: Theorem \ref{maintheorem:spfinite}}
\label{section:spfinite}

Let $\O_K$ be the ring of integers in an algebraic number field $K$, let
$\alpha \subset \O_K$ be a proper nonzero ideal, and let $\bk$ be a Noetherian
ring.  In this section, we construct
$\SI(\O_K/\alpha)$-modules $\HHH_k(\Sp(\O_K,\alpha);\bk)$ as in
Theorem \ref{maintheorem:spfinite} for all $k \geq 0$;
see Corollary \ref{corollary:constructsp}.
We then prove Theorem \ref{maintheorem:spfinite}, which says that these are finitely generated modules.

\paragraph{Categories and functors.}
As was noted in Example \ref{example:si}, for any ring $R$ the category $\SI(R)$ is
a complemented category whose monoidal structure is given by the orthogonal direct sum.  We
have
\begin{equation}
\label{eqn:siraut}
\Aut_{\SI(R)}(R^n) = \Sp_{2n}(R) \quad \quad (n \geq 0).
\end{equation}
Define a strong monoidal functor $\Psi\colon \SI(\O_K) \rightarrow \SI(\O_K/\alpha)$ by $\Psi(V) = (\O_K/\alpha) \otimes_{\O_K} V$.
This functor takes the generator $\O_K^{2}$ of $\SI(\O_K)$ to the generator $(\O_K/\alpha)^2$
of $\SI(\O_K/\alpha)$ (both are equipped with the standard symplectic form).  
Define $\Sp(\O_K,\alpha) = \Gamma(\Psi)$.  Using the identification
in \eqref{eqn:siraut}, we have
$\Sp(\O_K,\alpha)_{\O_K^{2n}} = \Sp_{2n}(\O_K,\alpha)$.

\begin{lemma}
\label{lemma:sphighlysurjective}
The functor $\Psi$ is highly surjective.
\end{lemma}
\begin{proof}
This is equivalent to the fact that the map 
$\Sp_{2n}(\O_K) \rightarrow \Sp_{2n}(\O_K/\alpha)$
is surjective for all $n \geq 0$, which follows from strong approximation (see, e.g.,
\cite[Chapter 7]{PlatonovRapinchuk}).
\end{proof}

\noindent
As we discussed in \S \ref{section:congruence}, Lemma \ref{lemma:sphighlysurjective}
has the following corollary.

\begin{corollary}
\label{corollary:constructsp}
For all $k \geq 0$, there exists an $\SI(\O_K/\alpha)$-module
$\HHH_k(\Sp(\O_K,\alpha);\bk)$ such that
$\HHH_k(\Sp(\O_K,\alpha);\bk)_{(\O_K/\alpha)^{2n}} = \HH_k(\Sp_{2n}(\O_K,\alpha);\bk)$.
\end{corollary}

\paragraph{Identifying the space of embeddings.}
We now wish to give a concrete description of the space of $\O_K^{2}$-embeddings
of $\O_K^{2n}$.  Letting $\ialg$ be the symplectic form on $\O_K^{2n}$, a
{\bf symplectic basis} for $\O_k^{2n}$ is an ordered basis $(a_1,b_1,\ldots,a_n,b_n)$ such that
$\ialg(a_i,a_j) = \ialg(b_i,b_j) = 0$ and $\ialg(a_i,b_j) = \delta_{ij}$ for $1 \leq i,j \leq n$,
where $\delta_{ij}$ is the Kronecker delta.  A {\bf partial symplectic basis} for $\O_K^{2n}$ is
a sequence $(a_1,b_1,\ldots,a_p,b_p)$ of elements of $\O_K^{2n}$ that
can be extended to a symplectic basis $(a_1,b_1,\ldots,a_n,b_n)$.
The {\bf space of partial symplectic bases} for $\O_K^{2n}$, denoted
$\PartialSpBases_n(\O_K)$, is the semisimplicial set whose $(p-1)$-cells
are partial symplectic bases $(a_1,b_1,\ldots,a_p,b_p)$.

\begin{lemma}
\label{lemma:spidentify}
The space of $\O_K^2$-embeddings of $\O_K^{2n}$ is isomorphic to $\PartialSpBases_n(\O_K)$.
\end{lemma}
\begin{proof}
Fix some $p \geq 0$.
The $p$-simplices of the space of $\O_K^2$-embeddings of $\O_K^{2n}$ are in bijection
with elements of $\Hom_{\SI(\O_K)}(\O_K^{p+1},\O_K^{2n})$.
We can define a set map
\[\eta \colon \Hom_{\SI(\O_K)}(\O_K^{p+1},\O_K^{2n}) \rightarrow (\PartialSpBases_n(\O_K))^p\]
by letting $\eta(f)$ be the image under $f \in \Hom_{\SI(\O_K)}(\O_K^{p+1},\O_K^{2n})$
of the standard symplectic basis for $\O_K^{p+1}$.  Clearly
$\eta$ is a bijection compatible with the semisimplicial structure.
\end{proof}

\noindent
The space $\PartialSpBases_n(\O_K)$ was introduced by Charney \cite{CharneyVogtmann},
who proved the following theorem.

\begin{theorem}[{\cite[Corollary 3.3]{CharneyVogtmann}}]
\label{theorem:partialspbasescon}
The space $\PartialSpBases_n(\O_K)$ is $\frac{n-6}{2}$-connected.
\end{theorem}

\begin{remark}
Just like for $\PartialBases_n(\O_K)$ in \S \ref{section:glfinite}, the definition
of $\PartialSpBases_n(R)$ is different from Charney's, but can be
proved equivalent for Dedekind domains using the ideas from \cite{ReinerUnimodular}.
\end{remark}

\paragraph{Putting it all together.}
We now prove Theorem \ref{maintheorem:spfinite}, which asserts that the $\SI(\O/\alpha)$-module $\HHH_k(\Sp(\O_K,\alpha);\bk)$ is finitely generated.

\begin{proof}[{Proof of Theorem \ref{maintheorem:spfinite}}]
We apply Theorem \ref{theorem:finitegen} to the highly surjective
functor $\Psi$.  This theorem has three hypotheses.
The first is that the category of $\SI(\O_K/\alpha)$-modules is locally Noetherian, which is Theorem
\ref{maintheorem:sinoetherian}.  The second is that
$\HH_k(\Sp_{2n}(\O_K,\alpha);\bk)$ is a finitely generated $\bk$-module for all
$k \geq 0$, which is a theorem of Borel--Serre \cite{BorelSerreCorners}.  The
third is that the space of $\O_K^{2}$-embeddings is highly acyclic, which
follows from Lemma \ref{lemma:spidentify} and Theorem \ref{theorem:partialspbasescon}.
We conclude that we can apply Theorem \ref{theorem:finitegen} and deduce
that $\HHH_k(\Sp(\O_K,\alpha);\bk)$ is a finitely generated $\SI(\O_K/\alpha)$-module
for all $k \geq 0$, as desired.
\end{proof}

\subsection{The category of surfaces}
\label{section:surfaces}

In preparation for proving Theorem \ref{maintheorem:modfinite} (which concerns
congruence subgroups of the mapping class group), this section
is devoted to a weak complemented category $\Surf$ of surfaces.  This
category was introduced by Ivanov \cite[\S 2.5]{IvanovTwisted}; our
contribution is to endow it with a monoidal structure.
It is different from previous monoidal categories
of surfaces (for instance, in \cite{MillerMod}), which do not include morphisms
between non-isomorphic surfaces.  A related monoidal category of surfaces
was constructed independently by Wahl--Randal-Williams \cite[\S 5.6]{WahlStability}, though
their morphisms are different from ours. 

\Figure{figure:surfbasic}{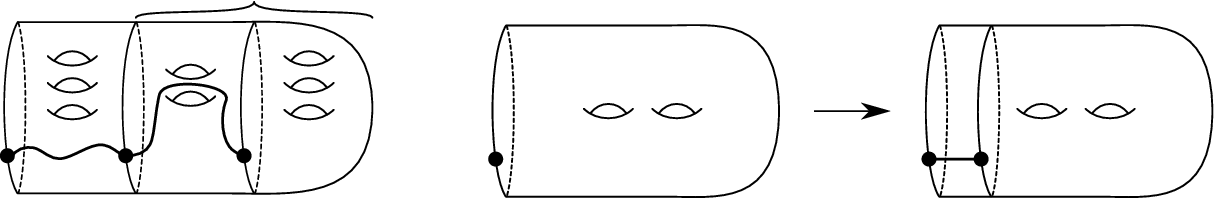}{The LHS shows how to compose morphisms
$(f,\tau) \colon S \rightarrow S'$ and 
$(f',\tau') \colon S' \rightarrow S''$.  The picture depicts $S''$.
The RHS shows the identity morphism $(f_0,\tau_0) \colon S \rightarrow S$, which pushes
$\partial S$ across a collar neighborhood.}{65}

\paragraph{The category $\bSurf$.}
The objects of the category $\Surf$ are pairs $(S,\eta)$ as follows.
\begin{compactitem}
\item $S$ is a compact oriented surface with one boundary component.
\item $\eta \colon [0,1] \rightarrow \partial S$ is an embedding such that $S$ lies
to the left of $\Image(\eta)$.  
We will call $\eta$ the {\bf boundary arc} of $S$.
\end{compactitem}
This endows $S$ with a natural basepoint, namely $\eta(1/2) \in \partial S$.
To keep our notation from getting out of hand, we will usually refer to $(S,\eta)$ simply as $S$, leaving $\eta$ implicit.  
For $S,S' \in \Surf$, the set $\Hom_{\Surf}(S,S')$ consists of homotopy classes of pairs
$(f,\tau)$ as follows.
\begin{compactitem}
\item $f \colon S \rightarrow \Interior(S')$ is an orientation-preserving embedding.
\item $\tau$ is a properly embedded arc in $S' \setminus \Interior(f(S))$ which goes from the basepoint of $S'$ to the image under $f$ of the basepoint of $S$.
\end{compactitem}

If $(f,\tau) \colon S \rightarrow S'$ and $(f',\tau') \colon S' \rightarrow S''$ are morphisms, then 
\[
(f',\tau') \circ (f,\tau) = (f' \circ f,\tau' \cdot f'(\tau)) \colon S \rightarrow S'',
\]
where arcs are composed as in the fundamental groupoid (this goes from left to right rather than right to left as
in function composition; see Figure \ref{figure:surfbasic}).  
The identity morphism of $S$ is the morphism $(f_0,\tau_0)$ 
depicted in Figure \ref{figure:surfbasic}; the
embedding $f_0$ ``pushes'' $\partial S$ across a collar neighborhood.

\begin{remark}
The boundary arc of an object of $\Surf$ is necessary to define the monoidal structure on $\Surf$. However, aside from providing a basepoint it plays little role in the structure of $\Surf$. Indeed, if $S, S' \in \Surf$ are objects with the same underlying surface, then it is easy to see that there exists an isomorphism $(f,\tau) \colon S \stackrel{\cong}{\rightarrow} S'$.
\end{remark}

\paragraph{Automorphisms.}
Ivanov proved the following theorem.

\begin{theorem}[{\cite[\S 2.6]{IvanovTwisted}}]
\label{theorem:surfauto}
Let $S \in \Surf$.  Then $\Aut_{\Surf}(S)$ is isomorphic to the mapping class group of $S$.
\end{theorem}

\begin{remark}
One potentially confusing point here is that Ivanov allows his surfaces to have multiple
boundary components. The theorem in \cite[\S 2.6]{IvanovTwisted}
is more complicated than Theorem \ref{theorem:surfauto}, but it reduces
to Theorem \ref{theorem:surfauto} when $S$ has one boundary component.
\end{remark}

\paragraph{A bifunctor.}
Let $\Rect = [0,3] \times [0,1] \subset \R^2$.
Consider objects $S_1$ and $S_2$ of $\Surf$.  Define
$S_1 \monpro S_2$ to be the following object of $\Surf$.
For $i=1,2$, let $\eta_i\colon [0,1] \rightarrow S_i$ be the boundary arc
of $S_i$.
Then $S_1 \monpro S_2$ is the result of gluing
$S_1$ and $S_2$ to $\Rect$ as follows.
\begin{compactitem}
\item For $t \in [0,1]$, identify $\eta_1(t) \in S_1$ and $(0,t) \in \Rect$.
\item For $t \in [0,1]$, identify $\eta_2(t) \in S_2$ and $(3,1-t) \in \Rect$.
\end{compactitem}
We have $1-t$ in the above to ensure that the resulting surface is oriented.
The boundary arc of $S_1 \monpro S_2$ is $t \mapsto (1+t,0) \in \Rect$. 

\Figure{figure:bifunctor}{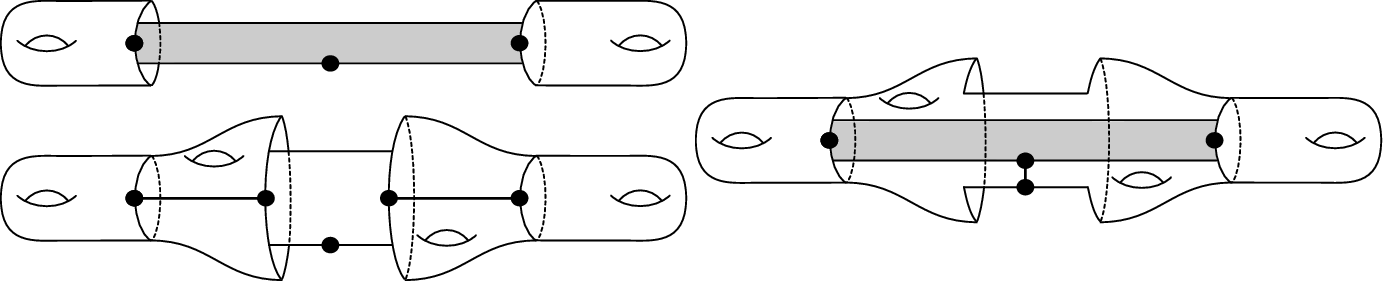}
{The upper left figure is $S_1 \monpro S_2$, divided into the three
subsurfaces $S_1$ and $S_2$ and $\Rect$ ($\Rect$ is shaded).
The bottom left figure is $S_1' \monpro S_2'$; here the rectangle is
labeled $\Rect'$.  For $i=1,2$, the morphism
$(f_i,\tau_i) \colon S_i \rightarrow S_i'$ is depicted.  The right hand
figure shows how $S_1 \monpro S_2$ is embedded in $S_1' \monpro S_2'$.}{65}

\begin{lemma}
\label{lemma:surfbifunctor}
The operation $\monpro$ is a bifunctor on $\Surf$.
\end{lemma}
\begin{proof}
For $i=1,2$, let $(f_i,\tau_i) \colon S_i \rightarrow S_i'$ be a
$\Surf$-morphism.  We must construct a canonical
morphism 
\[
(g,\lambda) \colon S_1 \monpro S_2 \longrightarrow S_1' \monpro S_2'.
\]
The based surface $S_1 \monpro S_2$ consists of three pieces, namely $S_1$ and $S_2$
and $\Rect$.  Identify all three of these pieces with their
images in $S_1 \monpro S_2$.  Similarly, $S_1' \monpro S_2'$ consists
of three pieces, namely $S_1'$ and $S_2'$ and a rectangle which we will call
$\Rect'$ to distinguish it from $\Rect \subset S_1 \monpro S_2$, and we will
identify these three pieces with their images in $S_1' \monpro S_2'$.
Let $g\colon S_1 \monpro S_2 \rightarrow S_1' \monpro S_2'$ be the embedding
depicted in Figure \ref{figure:bifunctor}, so $g(\Rect)$ is a regular neighborhood
in $S_1' \monpro S_2' \setminus \Interior(f_1(S_1) \cup f_2(S_2))$ of
the union of $\tau_1$ and $\tau_2$ and the arc $t \mapsto (3t,1/2)$ in $\Rect' = [0,3] \times [0,1] \subset \R^2$.  
Also, let $\lambda$ be the arc shown in Figure \ref{figure:bifunctor}.  It
is clear that $(g,\lambda)$ is well-defined up to homotopy.
\end{proof}

\paragraph{Monoidal structure.}
We now come to the following result.

\begin{lemma}
\label{lemma:surfmonoidal}
The category $\Surf$ is a monoidal category (not symmetric!) with respect to the bifunctor $\monpro$.
\end{lemma}
\begin{proof}
Clearly a disc (with some boundary arc; any two such boundary
arcs determine isomorphic objects of $\Surf$)
is a two-sided identity for $\monpro$.  Associativity is clear from
Figure \ref{figure:surfmonoidal}.  The various diagrams that need to commute
are all easy exercises.
\end{proof}

\Figure{figure:surfmonoidal}{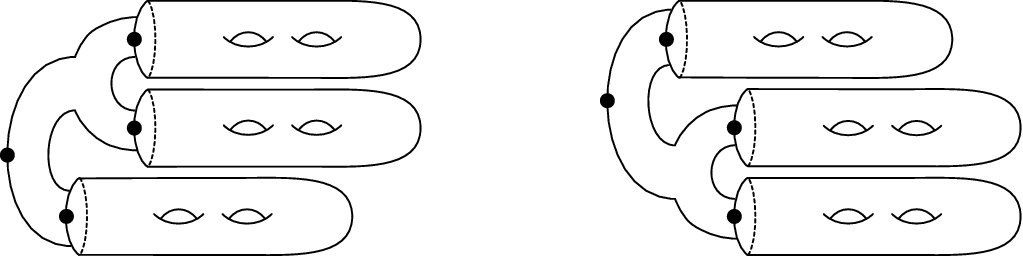}{The LHS is
$(S_1 \monpro S_2) \monpro S_3$ and the RHS is
$S_1 \monpro (S_2 \monpro S_3)$.}{65}

\paragraph{Weak complemented category.}
We finally come to the main result of this section.

\begin{proposition}
\label{proposition:surfcomplemented}
The monoidal category $(\Surf,\monpro)$ is a weak complemented category with
generator a one-holed torus.
\end{proposition}
\begin{proof}
If $\Surface{1}{1}$ is a one-holed torus with a fixed boundary arc and
$S$ is an arbitrary object of $\Surf$, then letting $g$ be the genus of $S$ we
clearly have $S \cong \monpro_{i=1}^g \Surface{1}{1}$.
It follows that $\Surface{1}{1}$ is a generator of $(\Surf,\monpro)$.  Now
let $(f,\tau) \colon S_1 \rightarrow S_2$ be a $\Surf$-morphism.
We must construct a complement to $(f,\tau)$ which is unique up to precomposing
the inclusion map by an isomorphism.  Define $T$ to be the complement
of an open regular neighborhood in $S_2$ of $\partial S_2 \cup \tau \cup f(S_1)$.
Thus $T$ is a compact oriented surface with one boundary component.  Fix a
boundary arc of $T$.  As is shown in Figure \ref{figure:surfcomplement}, there
exists a $\Surf$-morphism $(g,\lambda)\colon T \rightarrow S_2$ such that $T$ is a complement to
$(f,\tau)$.  Uniqueness of this complement is an easy exercise.
\end{proof}

\Figure{figure:surfcomplement}{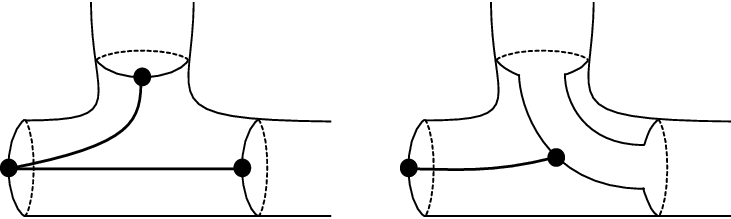}{The LHS shows $f(S_1)$ and $T$ together with the
arcs $\tau$ and $\lambda$.  The RHS shows the desired morphism $T \monpro S_1 \rightarrow S_2$.}{65}
 
\subsection{The mapping class group: Theorem \ref{maintheorem:modfinite}}
\label{section:modfinite}

Fix $\ell \geq 2$, and let $\bk$ be a Noetherian ring.
In this section, we construct
$\SI(\Z/\ell)$-modules $\HHH_k(\MCG(\ell);\bk)$ as in
Theorem \ref{maintheorem:modfinite} for all $k \geq 0$;
see Corollary \ref{corollary:constructmod}.
We then prove Theorem \ref{maintheorem:modfinite},
which says that $\HHH_k(\MCG(\ell);\bk)$ is a finitely generated
$\SI(\Z/\ell)$-module for all $k \geq 0$.

\paragraph{Categories and functors.}
Let $(\Surf,\monpro)$ be the weak complemented
category of surfaces discussed in \S \ref{section:surfaces} 
(see Proposition \ref{proposition:surfcomplemented}).  For
all $g \geq 0$, fix a boundary arc of $\Surface{g}{1}$.  This
allows us to regard each $\Surface{g}{1}$ as an object of $\Surf$.
By Theorem \ref{theorem:surfauto}, we have
\begin{equation}
\label{eqn:modraut}
\Aut_{\Surf}(\Surface{g}{1}) = \MCG_g^1 \quad \quad (g \geq 0).
\end{equation}
In Example \ref{example:si} we noted that the category $\SI(\Z/\ell)$ is
a complemented category whose monoidal structure is given by the orthogonal direct sum.
There is a strong monoidal functor $\Psi\colon \Surf \rightarrow \SI(\Z/\ell)$ defined via the
formula $\Psi(S) = \HH_1(S;\Z/\ell)$;
here $\HH_1(S;\Z/\ell)$ is equipped with the symplectic algebraic intersection
pairing.
This functor takes the generator $\Surface{1}{1}$ of $\Surf$ to the generator
$\HH_1(\Surface{1}{1};\Z/\ell) \cong (\Z/\ell)^2$ for $\SI(\Z/\ell)$.
Define $\MCG(\ell)= \Gamma(\Psi)$.  Using the identification
in \eqref{eqn:modraut}, we have
$\MCG(\ell)_{(\Z/\ell)^{2g}} = \MCG_g^1(\ell)$.

\begin{lemma}
\label{lemma:modhighlysurjective}
The functor $\Psi$ is highly surjective.
\end{lemma}
\begin{proof}
This is equivalent to the fact that the map
$\MCG_g^1 \rightarrow \Sp_{2g}(\Z/\ell)$
arising from the action of $\MCG_g^1$ on $\HH_1(\Surface{g}{1};\Z/\ell)$
is surjective for all $g \geq 0$.  This is classical.  For instance,
it can be factored as 
$\MCG_g^1 \rightarrow \Sp_{2g}(\Z) \rightarrow \Sp_{2g}(\Z/\ell)$; the map
$\MCG_g^1 \rightarrow \Sp_{2g}(\Z)$ is surjective by \cite[\S 6.3.2]{FarbMargalitPrimer},
and the map $\Sp_{2g}(\Z) \rightarrow \Sp_{2g}(\Z/\ell)$ is surjective by strong
approximation (see, e.g., \cite[Chapter 7]{PlatonovRapinchuk}).
\end{proof}

\noindent
As we discussed in \S \ref{section:congruence}, Lemma \ref{lemma:modhighlysurjective}
has the following corollary.

\begin{corollary}
\label{corollary:constructmod}
For all $k \geq 0$, there exists an $\SI(\Z/\ell)$-module
$\HHH_k(\MCG(\ell);\bk)$ such that
$\HHH_k(\MCG(\ell);\bk)_{(\Z/\ell)^{2g}} = \HH_k(\MCG_g^1(\ell);\bk)$.
\end{corollary}

\paragraph{Identifying the space of embeddings.}
We now wish to give a concrete description of the space of $\Surface{1}{1}$-embeddings
of $\Surface{g}{1}$.  A {\bf tethered torus system} on $\Surface{g}{1}$ is a sequence
$((f_0,\tau_0),\ldots,(f_p,\tau_p))$ as follows (see the right hand
side of Figure \ref{figure:tetheredtori}).
\begin{compactitem}
\item Each $(f_i,\tau_i)$ is a morphism from $\Surface{1}{1}$ to $\Surface{g}{1}$.
\item After homotoping the $(f_i,\tau_i)$, the following hold.  Let 
$\ast \in \partial \Surface{g}{1}$ be the basepoint.
\begin{compactitem}
\item $(f_i(\Surface{1}{1}) \cup \tau_i) \cap (f_j(\Surface{1}{1}) \cup \tau_j) = \{\ast\}$
for all distinct $0 \leq i, j \leq p$.
\item Going clockwise, the $\tau_i$ leave the basepoint $\ast$ in their natural
order (i.e., $\tau_i$ leaves before $\tau_j$ when $i < j$); see Remark \ref{remark:clockwise} below
for more on this.
\end{compactitem}
\end{compactitem}
The {\bf space of tethered tori} on $\Surface{g}{1}$, denoted $\TetheredTori_g$, is the
semisimplicial set whose $p$-simplices are tethered torus systems
$((f_0,\tau_0),\ldots,(f_p,\tau_p))$.

\begin{remark}
\label{remark:clockwise}
The condition on the order in which the $\tau_i$ leave the basepoint is needed because of the fact that the monoidal
structure on $(\Surf,\monpro)$ is not symmetric; in our proof below, a $p$-simplex of $\TetheredTori_g$
corresponds to a morphism $\monpro_{i=0}^p \Surface{1}{1} \rightarrow \Surface{g}{1}$, and the ordering
on the $\tau_i$ reflects the ordering on the $\monpro$-factors.
\end{remark}

\Figure{figure:tetheredtori}{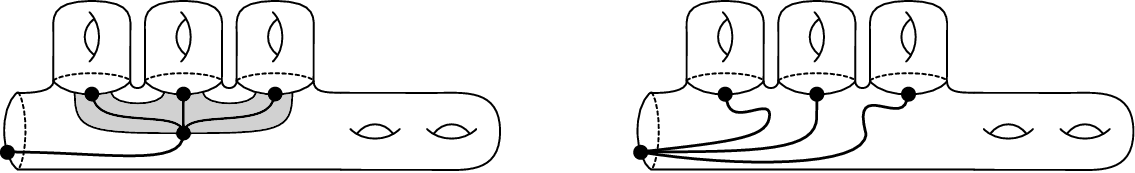}{The LHS shows a morphism
$(g,\lambda) \colon \monpro_{i=0}^2 \Surface{1}{1} \rightarrow \Surface{5}{1}$.  The
canonical morphisms 
$(h_j,\delta_j) \colon \Surface{1}{1} \rightarrow \monpro_{i=0}^2 \Surface{1}{1}$
are shown, and the union of the rectangles involved in forming $\monpro_{i=0}^2 \Surface{1}{1}$ is shaded.  The
RHS shows the associated tethered torus system.}{65}

\begin{lemma}
\label{lemma:modidentify}
The space of $\Surface{1}{1}$-embeddings of $\Surface{g}{1}$ is isomorphic to $\TetheredTori_g$.
\end{lemma}
\begin{proof}
Consider a morphism $(g,\lambda) \colon \monpro_{i=0}^p \Surface{1}{1} \rightarrow \Surface{g}{1}$.
For $0 \leq j \leq p$, let 
$(h_j,\delta_j) \colon \Surface{1}{1} \rightarrow \monpro_{i=0}^p \Surface{1}{1}$ be the
canonical map coming from the $j^{\text{th}}$ term.  Then 
$((g \circ h_0,\lambda \cdot g(\delta_0)),\ldots,(g \circ h_p,\lambda \cdot g(\delta_p)))$
is a tethered torus system (see Figure \ref{figure:tetheredtori}).  This
 gives an isomorphism between the space of $\Surface{1}{1}$-embeddings
of $\Surface{g}{1}$ and the space of tethered tori.
\end{proof}

\noindent
Using results of Hatcher--Vogtmann \cite{HatcherVogtmannTethers},
we will prove the following in \S \ref{section:tetheredtoricon}.

\begin{theorem}
\label{theorem:tetheredtoricon}
The space $\TetheredTori_g$ is $\frac{g-3}{2}$-connected.
\end{theorem}

\paragraph{Putting it all together.}
We now prove Theorem \ref{maintheorem:modfinite}, which asserts that the
$\SI(\Z/\ell)$-module $\HHH_k(\MCG(\ell);\bk)$ is finitely generated.

\begin{proof}[{Proof of Theorem \ref{maintheorem:modfinite}}]
We apply Theorem \ref{theorem:finitegen} to the highly surjective
functor $\Psi$.  This theorem has three hypotheses.
The first is that the category of $\SI(\Z/\ell)$-modules is locally Noetherian, which is Theorem
\ref{maintheorem:sinoetherian}.  The second is that
$\HH_k(\MCG_g^1(\ell);\bk)$ is a finitely generated $\bk$-module for all
$k \geq 0$.  This follows
easily from the fact that the moduli space of curves is a quasi-projective variety; see
\cite[p.\ 127]{FarbMargalitPrimer} for a discussion.
The third is that the space of $\Surface{1}{1}$-embeddings is highly acyclic, which
follows from Lemma \ref{lemma:modidentify} and Theorem \ref{theorem:tetheredtoricon}.
We conclude that we can apply Theorem \ref{theorem:finitegen} and deduce
that $\HHH_k(\MCG(\ell);\bk)$ is a finitely generated $\SI(\Z/\ell)$-module
for all $k \geq 0$, as desired.
\end{proof}

\subsection{The space of tethered tori}
\label{section:tetheredtoricon}

Our goal is to prove Theorem \ref{theorem:tetheredtoricon}, which asserts
that $\TetheredTori_g$ is $\frac{g-3}{2}$-connected.  We will deduce
this from work of Hatcher--Vogtmann.  Recall that we have endowed
each $\Surface{g}{1}$ with a basepoint $\ast \in \partial \Surface{g}{1}$.
A {\bf tethered chain} on $\Surface{g}{1}$ is a triple $(\alpha,\beta,\tau)$ as follows.
\begin{compactitem}
\item $\alpha$ and $\beta$ are oriented properly embedded
simple closed curves on $\Surface{g}{1}$ that intersect once
transversely with a positive orientation.  
\item $\tau$ is a properly embedded arc that starts at $\ast$, ends
at a point of $\beta \setminus \alpha$, and
is otherwise disjoint from $\alpha \cup \beta$.  
\end{compactitem}
We will identify homotopic tethered chains.
The {\bf space of tethered chains}, denoted $\TetheredChains_g$, is the
semisimplicial set whose $p$-simplices are sequences
$((\alpha_0,\beta_0,\tau_0),\ldots,(\alpha_p,\beta_p,\tau_p))$
of tethered chains with the following properties.
\begin{compactitem}
\item They can be homotoped such that they only intersect at $\ast$.
\item Going clockwise, the $\tau_i$ leave the basepoint $\ast$ in their natural
order (i.e., $\tau_i$ leaves before $\tau_j$ when $i < j$).
\end{compactitem}
We remark that any unordered sequence of tethered chains that satisfies the first
condition can be uniquely ordered such that it satisfies the second.
The complex $\TetheredChains_g$ was introduced by Hatcher--Vogtmann
\cite{HatcherVogtmannTethers}, who proved the following theorem.

\begin{theorem}[{Hatcher--Vogtmann, \cite{HatcherVogtmannTethers}}]
\label{theorem:tetheredchains}
The space $\TetheredChains_g$ is $\frac{g-3}{2}$-connected.
\end{theorem}

\begin{proof}[{Proof of Theorem \ref{theorem:tetheredtoricon}}]
We begin by defining a map $\pi^0 \colon \TetheredChains_g^0 \rightarrow \TetheredTori_g^0$.
Fix once and for all a tethered chain
$(A,B,\Gamma)$ in $\Surface{1}{1}$.
Consider a tethered chain $(\alpha,\beta,\tau) \in \TetheredChains_g^0$.  
Let $T$ be a closed regular neighborhood
of $\alpha \cup \beta$.  Choosing $T$ small enough, we can assume that
$\tau$ only intersects $\partial T$ once.  Write $\tau = \tau' \cdot \tau''$, where
$\tau'$ is the segment of $\tau$ from the basepoint of $\Surface{g}{1}$ to $\partial T$
and $\tau''$ is the segment of $\tau$ from $\partial T$ to $\beta$.  Using
the change of coordinates principle of \cite[\S 1.3]{FarbMargalitPrimer}, there
exists an orientation-preserving homeomorphism $f \colon \Surface{1}{1} \rightarrow T$ such
that $f(A) = \alpha$ and $f(B) = \beta$ and $f(\Gamma) = \tau''$.  
By \cite[Proposition 2.8]{FarbMargalitPrimer} (the ``Alexander method''), these conditions determine $f$ up to homotopy.
The pair $(f,\tau')$ is a tethered torus, and we define 
$\pi^0((\alpha,\beta,\tau)) = (f,\tau') \in \TetheredTori_g^0$.  
It is clear that this is well-defined and surjective (but not injective).

The map $\pi^0$ extends over the higher-dimensional simplices of $\TetheredChains_g$
in an obvious way to define a map $\pi \colon \TetheredChains_g \rightarrow \TetheredTori_g$.
Arbitrarily pick a set map $\phi^0 \colon \TetheredTori_g^0 \rightarrow \TetheredChains_g^0$ 
such that $\pi^0 \circ \phi^0 = \text{id}$.  If $((f_0,\tau_0),\ldots,(f_p,\tau_p))$ is a simplex
of $\TetheredTori_g$ and $(\alpha_i,\beta_i,\tau_i') = \phi^0((f_i,\tau_i))$ for $0 \leq i \leq p$,
then by construction we have $\alpha_i,\beta_i \subset \Image(f_i)$.  Since the images of the $f_i$
are disjoint, we conclude that $((\alpha_0,\beta_0,\tau_0'),\ldots,(\alpha_p,\beta_p,\tau_p'))$ is
a simplex of $\TetheredChains_g$.  Thus $\phi^0$ extends over the
higher-dimensional simplices of $\TetheredTori_g$ to define a map
$\phi \colon \TetheredTori_g \rightarrow \TetheredChains_g$ satisfying 
$\pi \circ \phi = \text{id}$.  This implies that $\pi$ induces a surjective map
on all homotopy groups, so by Theorem \ref{theorem:tetheredchains} we deduce
that $\TetheredTori_g$ is $\frac{g-3}{2}$-connected.
\end{proof}

\begin{footnotesize}
\noindent
\begin{tabular*}{\linewidth}[t]{@{}p{\widthof{Department of Mathematics}+0.5in}@{}p{\linewidth - \widthof{Department of Mathematics} - 0.5in}@{}}
{\raggedright
Andrew Putman\\
Department of Mathematics\\
Rice University, MS 136 \\
6100 Main St.\\
Houston, TX 77005\\
~\\
{\it Current address: }\\
Department of Mathematics\\
University of Notre Dame\\
279 Hurley Hall\\
South Bend, IN 46556\\
{\tt andyp@nd.edu}}
&
{\raggedright
Steven V Sam\\
Department of Mathematics\\
University of California, Berkeley\\
933 Evans Hall\\
Berkeley, CA 94720\\
~\\
{\it Current address: }\\
Department of Mathematics\\
University of Wisconsin, Madison\\
480 Lincoln Drive\\
Madison, WI 53706\\
{\tt svs@math.wisc.edu}}
\end{tabular*}\hfill
\end{footnotesize}

\end{document}